\documentclass[a4paper,11pt,reqno]{amsart}

\usepackage[T1]{fontenc}
\usepackage[utf8x]{inputenc}
\usepackage{hyperref}
\usepackage{xcolor}
\hypersetup{
    colorlinks,
    linkcolor={red!50!black},
    citecolor={blue!50!black},
    urlcolor={blue!80!black}
}

\usepackage{microtype}
\usepackage{MnSymbol}
\usepackage{textcomp}
\usepackage{stmaryrd}

\usepackage[mathcal]{euscript}
\let\mathcaltmp\mathcal
\usepackage{mathrsfs}
\let\mathcal\mathscr
\let\mathscr\mathcaltmp

\newtheoremstyle{plain}
  {.5\baselineskip\@plus.2\baselineskip\@minus.2\baselineskip}
  {0\baselineskip\@plus.2\baselineskip\@minus.2\baselineskip\@plus.5em}
  {\slshape}
  {}
  {\bfseries}
  {.}
  { }
  {}

\newtheoremstyle{definition}
  {.5\baselineskip\@plus.2\baselineskip\@minus.2\baselineskip}
  {0\baselineskip\@plus.2\baselineskip\@minus.2\baselineskip\@plus.5em}
  {}
  {}
  {\bfseries}
  {.}
  { }
  {}

\theoremstyle{plain}

\makeatletter
\newcommand{\eqnum}{\refstepcounter{equation}\textup{\tagform@{\theequation}}}
\makeatother

\makeatletter \@addtoreset{equation}{section} \makeatother
\renewcommand{\theequation}{\thesection.\arabic{equation}}

\newtheorem{thmX}{Theorem}

\newtheorem{thm}[subsubsection]{Theorem}
\newtheorem{cor}[subsubsection]{Corollary}
\newtheorem{lem}[subsubsection]{Lemma}
\newtheorem{prop}[subsubsection]{Proposition}

\theoremstyle{definition}
\newtheorem{defn}[subsubsection]{Definition}

\newcommand\arXiv[1]{\href{http://arxiv.org/abs/#1}{arXiv:#1}}

\usepackage{xparse}

\usepackage{xspace}

\newcommand{\changelocaltocdepth}[1]{%
  \addtocontents{toc}{\protect\setcounter{tocdepth}{#1}}%
  \setcounter{tocdepth}{#1}}

\newcommand{\nc}{\newcommand}
\nc{\renc}{\renewcommand}
\nc{\ssec}{\subsection}
\nc{\sssec}{\subsubsection}
\nc{\on}{\operatorname}
\nc{\term}[1]{#1\xspace}

\usepackage[frame,cmtip,curve,arrow,matrix,line,graph]{xy}
\usepackage{tikz}
\usetikzlibrary{matrix}
\usepackage{tikz-cd}

\tikzset{
  commutative diagrams/.cd,
  arrow style=tikz,
  diagrams={>=latex}}
\makeatletter
\tikzset{
  column sep/.code=\def\pgfmatrixcolumnsep{\pgf@matrix@xscale*(#1)},
  row sep/.code   =\def\pgfmatrixrowsep{\pgf@matrix@yscale*(#1)},
  matrix xscale/.code=%
    \pgfmathsetmacro\pgf@matrix@xscale{\pgf@matrix@xscale*(#1)},
  matrix yscale/.code=%
    \pgfmathsetmacro\pgf@matrix@yscale{\pgf@matrix@yscale*(#1)},
  matrix scale/.style={/tikz/matrix xscale={#1},/tikz/matrix yscale={#1}}}
\def\pgf@matrix@xscale{1}
\def\pgf@matrix@yscale{1}
\tikzcdset{scale cd/.style={every label/.append style={scale=#1}, cells={nodes={scale=#1}}}}
\makeatother

\usepackage[inline]{enumitem}
\setlist[enumerate,1]{label={(\alph*)},itemsep=\parskip}
\setlist[enumerate,2]{label={(\roman*)},itemsep=\parskip}
\setlist[itemize,1]{itemsep=\parskip,leftmargin=2em}

\newlist{enumcompress}{enumerate}{1}
\setlist[enumcompress,1]{label={(\alph*)},
  itemsep=0.3\parskip,leftmargin=*,align=left,topsep=0em}

\newlist{thmlist}{enumerate}{1}
\setlist[thmlist,1]{label={\em(\roman*)},ref={\upshape(\roman*)},
  itemsep=0.5em,leftmargin=*,align=right,widest=vi)}

\newlist{thmlistbis}{enumerate}{1}
\setlist[thmlistbis,1]{label={\em(\roman*~\textit{bis})},
  ref={(\roman*}~\textit{bis}\upshape{)},
  itemsep=0.5em,leftmargin=*,align=right, widest=vi)}

\newlist{defnlist}{enumerate}{2}
\setlist[defnlist,1]{label={(\roman*)},ref={(\roman*)},
  itemsep=0.5em,leftmargin=*,align=right,widest=vi)}
\setlist[defnlist,2]{label={(\alph*)}, ref={(\alph*)}, itemsep=0.75em,
  labelsep=0em,labelindent=0em,leftmargin=*,align=left,widest=vi),topsep=0.75em}

\newlist{inlinelist}{enumerate*}{1}
\setlist[inlinelist,1]{label={(\alph*)}}

\newlist{inlinedefnlist}{enumerate*}{1}
\definecolor{green}{HTML}{38550C}
\setlist[inlinedefnlist,1]{label={\color{green}(\roman*)}}

\nc{\cA}{\ensuremath{\mathcal{A}}\xspace}
\nc{\cB}{\ensuremath{\mathcal{B}}\xspace}
\nc{\cC}{\ensuremath{\mathcal{C}}\xspace}
\nc{\cD}{\ensuremath{\mathcal{D}}\xspace}
\nc{\cE}{\ensuremath{\mathcal{E}}\xspace}
\nc{\cF}{\ensuremath{\mathcal{F}}\xspace}
\nc{\cG}{\ensuremath{\mathcal{G}}\xspace}
\nc{\cH}{\ensuremath{\mathcal{H}}\xspace}
\nc{\cI}{\ensuremath{\mathcal{I}}\xspace}
\nc{\cJ}{\ensuremath{\mathcal{J}}\xspace}
\nc{\cK}{\ensuremath{\mathcal{K}}\xspace}
\nc{\cL}{\ensuremath{\mathcal{L}}\xspace}
\nc{\cM}{\ensuremath{\mathcal{M}}\xspace}
\nc{\cN}{\ensuremath{\mathcal{N}}\xspace}
\nc{\cO}{\ensuremath{\mathcal{O}}\xspace}
\nc{\cP}{\ensuremath{\mathcal{P}}\xspace}
\nc{\cQ}{\ensuremath{\mathcal{Q}}\xspace}
\nc{\cR}{\ensuremath{\mathcal{R}}\xspace}
\nc{\cS}{\ensuremath{\mathcal{S}}\xspace}
\nc{\cT}{\ensuremath{\mathcal{T}}\xspace}
\nc{\cU}{\ensuremath{\mathcal{U}}\xspace}
\nc{\cV}{\ensuremath{\mathcal{V}}\xspace}
\nc{\cW}{\ensuremath{\mathcal{W}}\xspace}
\nc{\cX}{\ensuremath{\mathcal{X}}\xspace}
\nc{\cY}{\ensuremath{\mathcal{Y}}\xspace}
\nc{\cZ}{\ensuremath{\mathcal{Z}}\xspace}

\nc{\sA}{\ensuremath{\mathscr{A}}\xspace}
\nc{\sB}{\ensuremath{\mathscr{B}}\xspace}
\nc{\sC}{\ensuremath{\mathscr{C}}\xspace}
\nc{\sD}{\ensuremath{\mathscr{D}}\xspace}
\nc{\sE}{\ensuremath{\mathscr{E}}\xspace}
\nc{\sF}{\ensuremath{\mathscr{F}}\xspace}
\nc{\sG}{\ensuremath{\mathscr{G}}\xspace}
\nc{\sH}{\ensuremath{\mathscr{H}}\xspace}
\nc{\sI}{\ensuremath{\mathscr{I}}\xspace}
\nc{\sJ}{\ensuremath{\mathscr{J}}\xspace}
\nc{\sK}{\ensuremath{\mathscr{K}}\xspace}
\nc{\sL}{\ensuremath{\mathscr{L}}\xspace}
\nc{\sM}{\ensuremath{\mathscr{M}}\xspace}
\nc{\sN}{\ensuremath{\mathscr{N}}\xspace}
\nc{\sO}{\ensuremath{\mathscr{O}}\xspace}
\nc{\sP}{\ensuremath{\mathscr{P}}\xspace}
\nc{\sQ}{\ensuremath{\mathscr{Q}}\xspace}
\nc{\sR}{\ensuremath{\mathscr{R}}\xspace}
\nc{\sS}{\ensuremath{\mathscr{S}}\xspace}
\nc{\sT}{\ensuremath{\mathscr{T}}\xspace}
\nc{\sU}{\ensuremath{\mathscr{U}}\xspace}
\nc{\sV}{\ensuremath{\mathscr{V}}\xspace}
\nc{\sW}{\ensuremath{\mathscr{W}}\xspace}
\nc{\sX}{\ensuremath{\mathscr{X}}\xspace}
\nc{\sY}{\ensuremath{\mathscr{Y}}\xspace}
\nc{\sZ}{\ensuremath{\mathscr{Z}}\xspace}

\nc{\bA}{\ensuremath{\mathbf{A}}\xspace}
\nc{\bB}{\ensuremath{\mathbf{B}}\xspace}
\nc{\bC}{\ensuremath{\mathbf{C}}\xspace}
\nc{\bD}{\ensuremath{\mathbf{D}}\xspace}
\nc{\bE}{\ensuremath{\mathbf{E}}\xspace}
\nc{\bF}{\ensuremath{\mathbf{F}}\xspace}
\nc{\bG}{\ensuremath{\mathbf{G}}\xspace}
\nc{\bH}{\ensuremath{\mathbf{H}}\xspace}
\nc{\bI}{\ensuremath{\mathbf{I}}\xspace}
\nc{\bJ}{\ensuremath{\mathbf{J}}\xspace}
\nc{\bK}{\ensuremath{\mathbf{K}}\xspace}
\nc{\bL}{\ensuremath{\mathbf{L}}\xspace}
\nc{\bM}{\ensuremath{\mathbf{M}}\xspace}
\nc{\bN}{\ensuremath{\mathbf{N}}\xspace}
\nc{\bO}{\ensuremath{\mathbf{O}}\xspace}
\nc{\bP}{\ensuremath{\mathbf{P}}\xspace}
\nc{\bQ}{\ensuremath{\mathbf{Q}}\xspace}
\nc{\bR}{\ensuremath{\mathbf{R}}\xspace}
\nc{\bS}{\ensuremath{\mathbf{S}}\xspace}
\nc{\bT}{\ensuremath{\mathbf{T}}\xspace}
\nc{\bU}{\ensuremath{\mathbf{U}}\xspace}
\nc{\bV}{\ensuremath{\mathbf{V}}\xspace}
\nc{\bW}{\ensuremath{\mathbf{W}}\xspace}
\nc{\bX}{\ensuremath{\mathbf{X}}\xspace}
\nc{\bY}{\ensuremath{\mathbf{Y}}\xspace}
\nc{\bZ}{\ensuremath{\mathbf{Z}}\xspace}

\nc{\bbA}{\ensuremath{\mathbb{A}}\xspace}
\nc{\bbB}{\ensuremath{\mathbb{B}}\xspace}
\nc{\bbC}{\ensuremath{\mathbb{C}}\xspace}
\nc{\bbD}{\ensuremath{\mathbb{D}}\xspace}
\nc{\bbE}{\ensuremath{\mathbb{E}}\xspace}
\nc{\bbF}{\ensuremath{\mathbb{F}}\xspace}
\nc{\bbG}{\ensuremath{\mathbb{G}}\xspace}
\nc{\bbH}{\ensuremath{\mathbb{H}}\xspace}
\nc{\bbI}{\ensuremath{\mathbb{I}}\xspace}
\nc{\bbJ}{\ensuremath{\mathbb{J}}\xspace}
\nc{\bbK}{\ensuremath{\mathbb{K}}\xspace}
\nc{\bbL}{\ensuremath{\mathbb{L}}\xspace}
\nc{\bbM}{\ensuremath{\mathbb{M}}\xspace}
\nc{\bbN}{\ensuremath{\mathbb{N}}\xspace}
\nc{\bbO}{\ensuremath{\mathbb{O}}\xspace}
\nc{\bbP}{\ensuremath{\mathbb{P}}\xspace}
\nc{\bbQ}{\ensuremath{\mathbb{Q}}\xspace}
\nc{\bbR}{\ensuremath{\mathbb{R}}\xspace}
\nc{\bbS}{\ensuremath{\mathbb{S}}\xspace}
\nc{\bbT}{\ensuremath{\mathbb{T}}\xspace}
\nc{\bbU}{\ensuremath{\mathbb{U}}\xspace}
\nc{\bbV}{\ensuremath{\mathbb{V}}\xspace}
\nc{\bbW}{\ensuremath{\mathbb{W}}\xspace}
\nc{\bbX}{\ensuremath{\mathbb{X}}\xspace}
\nc{\bbY}{\ensuremath{\mathbb{Y}}\xspace}
\nc{\bbZ}{\ensuremath{\mathbb{Z}}\xspace}

\nc{\mrm}[1]{\ensuremath{\mathrm{#1}}\xspace}
\nc{\mfr}[1]{\ensuremath{\mathfrak{#1}}\xspace}
\nc{\mit}[1]{\ensuremath{\mathit{#1}}\xspace}
\nc{\mbf}[1]{\ensuremath{\mathbf{#1}}\xspace}
\nc{\mcal}[1]{\ensuremath{\mathcal{#1}}\xspace}
\nc{\msc}[1]{\ensuremath{\mathscr{#1}}\xspace}

\nc{\sub}{\subseteq}
\nc{\too}{\longrightarrow}
\nc{\hook}{\hookrightarrow}
\nc{\hooklongrightarrow}{\lhook\joinrel\longrightarrow}
\nc{\hooklong}{\hooklongrightarrow}
\nc{\hooklongleftarrow}{\longleftarrow\joinrel\rhook}
\nc{\surj}{\twoheadrightarrow}
\nc{\twoheadlongrightarrow}{\relbar\joinrel\twoheadrightarrow}
\nc{\longrightleftarrows}{\ \raisebox{0.3ex}{\(\mathrel{\substack{\xrightarrow{\rule{1em}{0em}} \\[-1ex] \xleftarrow{\rule{1em}{0em}}}}\)}\ }

\renc{\ge}{\geqslant}
\renc{\geq}{\geqslant}
\renc{\le}{\leqslant}
\renc{\leq}{\leqslant}

\nc{\id}{\mathrm{id}}

\DeclareMathOperator{\Hom}{\on{Hom}}
\nc{\uHom}{\underline{\smash{\Hom}}}
\DeclareMathOperator{\Maps}{\on{Maps}}

\DeclareMathOperator{\End}{\on{End}}

\nc{\uEnd}{\underline{\smash{\End}}}

\nc{\colim}{\varinjlim}
\renc{\lim}{\varprojlim}
\nc{\Cofib}{\on{Cofib}}
\nc{\Fib}{\on{Fib}}
\nc{\initial}{\varnothing}
\nc{\op}{\mathrm{op}}

\DeclareMathOperator*{\fibprod}{\times}

\renc{\setminus}{\smallsetminus}

\usepackage{mathtools}
\DeclarePairedDelimiter\abs{\lvert}{\rvert}%

\newcommand{\thmref}[1]{Theorem~\ref{#1}}

\newcommand{\secref}[1]{Sect.~\ref{#1}}
\newcommand{\ssecref}[1]{Subsect. ~\ref{#1}}
\newcommand{\sssecref}[1]{(\ref{#1})}
\newcommand{\lemref}[1]{Lemma~\ref{#1}}
\newcommand{\propref}[1]{Proposition~\ref{#1}}
\newcommand{\corref}[1]{Corollary~\ref{#1}}

\renewcommand{\eqref}[1]{(\ref{#1})}

\newcommand{\itemref}[1]{\ref{#1}}

\nc{\Spec}{\on{Spec}}
\nc{\bDelta}{\mathbf{\Delta}}
\nc{\Cech}{\textnormal{\v{C}}}
\nc{\rightrightrightarrows}
  {\mathrel{\substack{\textstyle\rightarrow\\[-0.6ex]
   \textstyle\rightarrow \\[-0.6ex]
   \textstyle\rightarrow}}}
\nc{\rightrightrightrightarrows}
  {\mathrel{\substack{\textstyle\rightarrow\\[-0.6ex]
   \textstyle\rightarrow \\[-0.6ex]
   \textstyle\rightarrow \\[-0.6ex]
   \textstyle\rightarrow}}}
\nc{\Einfty}{{\sE_\infty}}
\nc{\Mod}{{\mrm{Mod}}}
\nc{\Fun}{\on{Fun}}
\nc{\pt}{\mrm{pt}}
\nc{\an}{\mrm{an}}
\nc{\pr}{\mrm{pr}}
\nc{\Lis}{\mrm{Lis}}
\nc{\LisStk}{\underline{\Lis}}
\nc{\aff}{{\mrm{aff}}}
\nc{\dash}{\textnormal{-}}
\nc{\un}{\mbf{1}}
\nc{\Spt}{\mrm{Spt}}
\nc{\Stk}{\mrm{Stk}}
\nc{\Corr}{\on{Corr}}
\nc{\Cat}{\mrm{Cat}}
\nc{\CatL}{\mrm{Cat}^{\mrm{L}}}
\nc{\PrL}{\mrm{Pr}^{\mrm{L}}}
\nc{\PrR}{\mrm{Pr}^{\mrm{R}}}
\nc{\Ex}{\mrm{Ex}}
\nc{\unit}{\mrm{unit}}
\nc{\counit}{\mrm{counit}}
\nc{\Tot}{\on{Tot}}
\nc{\Shv}{\on{Shv}}
\nc{\Sch}{\mrm{Sch}}
\nc{\Asp}{\mrm{Asp}}
\nc{\Top}{\mrm{Top}}
\nc{\lis}{\triangleleft}
\nc{\Ani}{\mrm{Ani}}
\nc{\D}{\on{\mbf{D}}}
\nc{\Nis}{\mrm{Nis}}
\nc{\et}{{\mrm{\acute{e}t}}}
\nc{\sm}{\mrm{sm}}
\nc{\etcov}{\mrm{\acute{e}tcov}}
\nc{\repr}{\mrm{repr}}
\nc{\Spc}{\on{Spc}}
\nc{\Art}{\on{Art}}
\nc{\Pt}{\on{Pt}}

\nc{\inftyCat}{\term{$\infty$-category}}
\nc{\inftyCats}{\term{$\infty$-categories}}

\nc{\inftyGrpd}{\term{$\infty$-groupoid}}
\nc{\inftyGrpds}{\term{$\infty$-groupoids}}

\title{Lisse extensions of weaves\vspace{-2mm}}
\author[A.\,A. Khan]{Adeel A. Khan}
\date{2024-01-07}

\makeatletter
\def\l@subsection{\@tocline{2}{0pt}{4pc}{6pc}{}}
\def\l@subsubsection{\@tocline{3}{0pt}{8pc}{8pc}{}}
\makeatother

\begin{document}

\begin{abstract}
  Any sheaf theory on schemes extends canonically to Artin stacks via a procedure called lisse extension.
  In this paper we show that lisse extension preserves the formalism of Grothendieck's six operations: more precisely, the lisse extension of a weave on schemes determines a weave on (higher) Artin stacks.
  The setup is general enough to apply to the stable motivic homotopy category with the six functor formalism of Voevodsky--Ayoub--Cisinski--Déglise, for instance, and is not specific to algebraic geometry: for example, it also applies to sheaves of spectra on topological stacks.
  \vspace{-5mm}
\end{abstract}

\maketitle

\renewcommand\contentsname{\vspace{-1cm}}
\tableofcontents

\parskip 0.5em

\thispagestyle{empty}


\changelocaltocdepth{1}
\section*{Introduction}

In the groundbreaking paper \cite{LiuZheng}, Y.~Liu and W.~Zheng extended the formalism of Grothendieck's six operations on derived categories of étale sheaves from schemes to Artin stacks.
While the technology they developed is not specific to derived categories of étale sheaves, their construction does have certain limitations in its applicability.
In this paper, we will prove a general extension theorem for six functor formalisms.
Some archetypical new examples are as follows:
\begin{enumerate}
  \item \emph{Spectral Betti sheaves}:\label{item:ex/an}
  The sheaf theory $X \mapsto \on{D}_{\mrm{an}}(X; \bS)$ sending a locally of finite type $\bC$-scheme $X$ to the stable \inftyCat of sheaves of spectra on the analytification $X(\bC)$.\footnote{%
    Here $\bS$ denotes the sphere spectrum.
    We could more generally take $X \mapsto \on{D}_{\mrm{an}}(X; R)$ for any $\Einfty$-ring spectrum $R$.
  }

  \item \emph{Motivic sheaves}:\label{item:ex/mot}
  The sheaf theory $X \mapsto \on{D}_{\mrm{mot}}(X; R)$ sending a scheme $X$ to the stable \inftyCat of $R$-linear motivic sheaves on $X$ (see \sssecref{sssec:weavenisex} for what we mean by this), where $R$ is a commutative ring.

  \item \emph{Motivic spectra}:\label{item:ex/sh}
  The sheaf theory $X \mapsto \on{SH}(X)$ sending a scheme $X$ to the stable \inftyCat of motivic spectra on $X$.
\end{enumerate}

In example~\itemref{item:ex/an}, the difficulty is that the unit object (the constant sheaf on the sphere spectrum $\bS$) does not lie in the heart of the cohomological t-structure on $\D(X)$; the fact that this problem does not arise in the context of derived categories of étale sheaves is exploited in \cite{LiuZheng} to bypass some major homotopy coherence issues.

In examples~\itemref{item:ex/mot} and \itemref{item:ex/sh}, there is a further orthogonal difficulty: these sheaf theories do not satisfy étale descent.\footnote{%
  unless $R$ is not a $\bQ$-algebra, in example~\itemref{item:ex/mot}
}
In particular, there is no descent along the \v{C}ech nerve of an atlas $X \surj \sX$ of an Artin stack $\sX$ by a scheme $X$.

While still relying essentially on some foundational machinery developed in \cite{LiuZheng} (as distilled conveniently in \cite[\S A.5]{Mann}), we give a different construction of the six operations on Artin stacks which circumvents both of the above-mentioned difficulties.
Given a sheaf theory $X \mapsto \D(X)$ on schemes, define the \emph{lisse extension} $\D^{\lis}$ by the limit of \inftyCats
\begin{equation}\label{eq:prjvcpda}
  \D^{\lis}(\sX) := \lim_{(S,s)} \D(S)
\end{equation}
for an Artin stack $\sX$, where the limit is taken over the \inftyCat $\Lis_\sX$ of pairs $(S, s : S \to \sX)$ where $S$ is a scheme and $s$ a smooth morphism; the transition functors are $*$-inverse image.
We would like to show that when the categories $\D(X)$ admit the six operations as $X$ varies over schemes, the \inftyCats $\D^{\lis}(\sX)$ still admit the six operations as $\sX$ varies over Artin stacks.

To state the result, we use the language of \emph{weaves} introduced in \cite{weaves}.

\begin{thmX}
  Let $\D^*_!$ be a weave on the category of schemes.
  \begin{thmlist}
    \item
    If $\D^*_!$ satisfies étale descent\footnote{%
      meaning that the underlying presheaf $\D^* : (\Sch)^\op \to \Cat$ satisfies étale descent
    }, then there exists an essentially unique extension of $\D^*_!$ to a weave $\D^{\lis,*}_!$ on the \inftyCat of (higher) Artin stacks such that the underlying presheaf $\D^{\lis,*}$ is the lisse extension of $\D^*$.

    \item
    If $\D^*_!$ satisfies Nisnevich descent, then there exists an essentially unique extension of $\D^*_!$ to a weave $\D^{\lis,*}_!$ on the \inftyCat of (higher) $\Nis$-Artin stacks such that the underlying presheaf $\D^{\lis,*}$ is the lisse extension of $\D^*$.
  \end{thmlist}
\end{thmX}

See Subsects.~\ref{ssec:weaveschet} and \ref{ssec:weaveschnis}.
Let us point out that the weave $\D^{\lis,*}_!$ in particular incorporates Poincaré duality isomorphisms for all smooth morphisms of Artin stacks (see \corref{cor:poinc}).

We have evidently bypassed the second difficulty mentioned above by simply modifying our notion of Artin stack.
The definition of \emph{$\Nis$-Artin} requires the existence of an atlas $X \surj \sX$ which is smooth and admits \emph{Nisnevich}-local sections.\footnote{%
  whereas in the usual notion of Artin stacks the atlas is required to be smooth and surjective, or equivalently smooth and admitting \emph{étale}-local sections
}
See \cite[0.2.2]{equilisse} for background on $\Nis$-Artin stacks.

In fact, we work in an axiomatic geometric context that applies not only to Artin stacks in algebraic geometry, but also to the analogues of Artin stacks on topological spaces, complex-analytic spaces, and so on.
For example, we study the example of sheaves of spectra on topological stacks in \secref{sec:shvtop}.
See \corref{cor:weave} for the abstract extension theorem.

\subsection{Conventions and notation}

\subsubsection{}

$\Ani$ denotes the \inftyCat of anima (a.k.a. \inftyGrpds or homotopy types).

\subsubsection{}

$\Cat$ denotes the \inftyCat of (large) \inftyCats, $\CatL$ denotes the \inftyCat of (large) \inftyCats and left-adjoint functors, $\PrL$ denotes the \inftyCat of presentable \inftyCats and left-adjoint functors, and $\PrR$ denotes the \inftyCat of presentable \inftyCats and right-adjoint functors.

\subsubsection{}

Let $\cC$ be an \inftyCat.
Given a presheaf of \inftyCats denoted $\D^* : \cC^\op \to \Cat$, we adopt the following notation:
\begin{itemize}
  \item For every object $X \in \cC$, $\D(X)$ denotes the \inftyCat $\D^*(X)$.
  \item For every morphism $f : X \to Y$ in $\cC$, $f^* : \D(Y) \to \D(X)$ denotes the induced functor $\D^*(f)$.
  \item Suppose $\D^*$ takes values in $\CatL$.
  Then for every morphism $f : X \to Y$ in $\cC$, $f^*$ admits a right adjoint $f_*$.
\end{itemize}

Analogously, given a presheaf denoted $\D^! : \cC^\op \to \Cat$, we write instead $f^! : \D(Y) \to \D(X)$ for the functor induced by a morphism $f : X \to Y$ in $\cC$.

Suppose $\cC$ admits finite products and $\D^*$ is lax symmetric monoidal with respect to the cartesian product on $\cC$.
Then we have:
\begin{itemize}
  \item 
  The tautological commutative comonoid structure on any object $X \in \cC$ gives rise to a symmetric monoidal structure on $\D(X)$.
  
  \item
  For any morphism $f : X \to Y$ in $\cC$, the functor $f^*$ inherits a symmetric monoidal structure.
  
  \item
  If $\D^*$ takes values in $\CatL$, then for every $X \in \cC$ the symmetric monoidal structure on $\D(X)$ is \emph{closed}.
  Thus there is an internal hom bifunctor $\uHom_{\D(X)}(-, -) : \D(X)^\op \times \D(X) \to \D(X)$ for every $X$ in $\cC$, right adjoint to the bifunctor $(-) \otimes (-) : \D(X) \times \D(X) \to \D(X)$.
\end{itemize}

\subsubsection{}

Let $\cC$ be an \inftyCat with fibred products.
By \emph{a class of morphisms} in $\cC$ we will mean a subcategory $\cE \sub \cC$ which contains all isomorphisms.

We say that $\cE$ is \emph{closed under base change} if for any morphism $f : X \to Y$ in $\cE$ and any morphism $Y' \to Y$ in $\cC$, the base change $f' : X\fibprod_Y Y' \to Y'$ belongs to $\cE$.

We say that $\cE$ is \emph{closed under sections} if for any morphism $f : X \to Y$ belonging to $\cE$, any section $i : Y \to X$ (so that $f \circ i \simeq \id$) also belongs to $\cE$.

Note that if $\cE$ is closed under base change and sections, and $f : X \to Y$ belongs to $\cE$, then the diagonal $\Delta_f : X \to X\fibprod_Y X$ also belongs to $\cE$.
Similarly, it follows that $\cE$ is closed under two-of-three.

\subsection{Acknowledgments}

I would like to thank Marc Hoyois, Tasuki Kinjo, Hyeonjun Park, Charanya Ravi, and Pavel Safronov for discussions and/or encouragement to finally write this paper.

\changelocaltocdepth{2}

\section{Weaves}

\subsection{Correspondences}

\subsubsection{}

Let $\cE$ be a class of morphisms in $\cC$ which is closed under base change.
Given two objects $X_1$ and $X_2$ in $\cC$, an \emph{$\cE$-correspondence}\footnote{%
  or simply \emph{correspondence}, when $\cE$ is clear from context
} from $X_1$ to $X_2$ in $\cC$ is a diagram of the form
\begin{equation}\label{eq:corr}
  \begin{tikzcd}
    X_{1,2} \ar{r}{g}\ar{d}{f}
    & X_2
    \\
    X_1
  \end{tikzcd}
\end{equation}
where $g$ belongs to $\cE$.

\subsubsection{}

Given a class of morphisms $\cE$ which is closed under base change, we let $\Corr_\cE(\cC)$ denote the symmetric monoidal \inftyCat of $\cE$-correspondences in $\cC$, defined as in \cite[\S 5.1]{GaitsgoryIndCoh}, \cite[Chaps.~7 and~9]{GaitsgoryRozenblyum}, or \cite[\S 6.1]{LiuZheng}.

Objects of $\Corr_\cE(\cC)$ are objects of $\cC$.
Given objects $X_1$ and $X_2$, morphisms $X_1 \to X_2$ in $\Corr_\cE(\cC)$ are $\cE$-correspondences from $X_1$ to $X_2$.
Given objects $X_1$, $X_2$ and $X_3$, and morphisms $X_1\to X_2$ and $X_2\to X_3$ given by $\cE$-correspondences
\begin{equation*}
  \begin{tikzcd}
    X_{1,2} \ar{d}{f_1}\ar{r}{g_1}
    & X_2,
    \\
    X_1
  \end{tikzcd}
  \qquad
  \begin{tikzcd}
    X_{2,3} \ar{d}{f_2}\ar{r}{g_2}
    & X_3,
    \\
    X_2
  \end{tikzcd}
\end{equation*}
respectively, the composite $(X_2 \to X_3) \circ (X_1 \to X_2)$ is given by the $\cE$-correspondence
\begin{equation*}
  \begin{tikzcd}
    X_{1,2,3} \ar{r}{g_1'}\ar{d}{f_2'}
    & X_{2,3} \ar{d}{f_2}\ar{r}{g_2}
    & X_3,
    \\
    X_{1,2} \ar{r}{g_1}\ar{d}{f_1}
    & X_2,
    & 
    \\
    X_1
    &
    & 
  \end{tikzcd}
\end{equation*}
where $X_{1,2,3}$ is the fibred product of $X_{1,2}$ and $X_{2,3}$ over $X_2$.

The monoidal product on $\Corr_\cE(\cC)$ is the cartesian product.

\subsubsection{}

There are canonical functors
\begin{equation}\label{eq:CtoCorr}
  \cC^\op \to \Corr_\cE(\cC),
  \quad \cE \to \Corr_\cE(\cC)
\end{equation}
given by identity on objects and by sending morphisms $f : X \to Y$ in $\cC$, resp. $g : X \to Y$ in $\cE$, to the correspondences
\begin{equation*}
  \begin{tikzcd}
    X \ar{d}{f}\ar[equals]{r}
    & X,
    \\
    Y
  \end{tikzcd}
  \qquad\text{resp.}~
  \begin{tikzcd}
    X \ar[equals]{d}\ar{r}{g}
    & Y,
    \\
    X
  \end{tikzcd}
\end{equation*}
respectively.
The functor $\cC^\op \to \Corr_\cE(\cC)$ is symmetric monoidal.

\subsection{Preweaves}

\subsubsection{}
\label{sssec:leftpreweave}

Let $\cC$ be an \inftyCat with fibred products, and $\cE$ a class of morphisms in $\cC$ which is closed under base change.

\begin{defn}
  A \emph{left preweave} on $(\cC,\cE)$\footnote{%
    or simply ``on $\cC$'', when $\cE$ is clear from context
  } is a lax symmetric monoidal functor
  \begin{equation}\label{eq:D^*_!}
    \D^*_! : \Corr_\cE(\cC) \to \Cat,
  \end{equation}
  with respect to the cartesian monoidal structure on $\Cat$.
\end{defn}

Given a left preweave $\D^*_!$ on $(\cC,\cE)$, we will say that a morphism in $\cC$ is \emph{shriekable} if it belongs to $\cE$.

\begin{defn}
  A \emph{preweave} $\D_*^!$ is a left preweave such that the functor $\Corr(\cC) \to \Cat$ factors through the \inftyCat $\CatL$ of \inftyCats and left adjoint functors.
  By passage to right adjoints, $\D^*_!$ determines a functor
  \[ \D_*^! : \Corr_\cE(\cC)^\op \to \Cat. \]
\end{defn}

\subsubsection{}

Given a left preweave $\D^*_!$ on $(\cC,\cE)$, we obtain by restriction along the functors \eqref{eq:CtoCorr} a lax symmetric monoidal functor
\begin{equation}
  \D^* : \cC^\op \to \Cat
\end{equation}
and a functor
\begin{equation}
  \D_! : \cE \to \Cat.
\end{equation}

If $\D^*_!$ is a preweave, we also obtain functors
\begin{equation}
  \D_* : \cC \to \Cat
\end{equation}
and
\begin{equation}
  \D^! : \cE^\op \to \Cat.
\end{equation}

\subsubsection{}

Given a left preweave $\D^*_!$ on $(\cC,\cE)$, the functor \eqref{eq:D^*_!} associates with every object $X \in \cC$ an \inftyCat $\D(X)$.

Any object $X$ in $\cC$ admits a tautological commutative comonoid structure with respect to the cartesian product, which gives rise to a commutative monoid structure on $\D(X)$ in $\Cat$, i.e., to a symmetric monoidal structure on the \inftyCat $\D(X)$.
If $\D^*_!$ is a preweave, this is a closed symmetric monoidal structure: there is an internal hom bifunctor $\uHom_{\D(X)}(-, -) : \D(X)^\op \times \D(X) \to \D(X)$ for every $X$ in $\cC$.

For every morphism $f : X \to Y$ in $\cC$, we have a symmetric monoidal functor $f^* : \D(Y) \to \D(X)$.
If $\D^*_!$ is a preweave, this admits a right adjoint $f_*$.

For every shriekable morphism $f : X \to Y$, we have a functor $f_! : \D(X) \to \D(Y)$ which is $\D(X)$-linear, where the $\D(X)$-module structure on $\D(Y)$ is via the symmetric monoidal functor $f^*$.
In particular, there is a canonical isomorphism
\begin{equation}
  \Ex^{\otimes,*}_! : f_!(-) \otimes (-) \simeq f_!(- \otimes f^*(-))
\end{equation}
called the \emph{projection formula}.
If $\D^*_!$ is a preweave, $f_!$ admits a right adjoint $f^!$.

For every cartesian square
\begin{equation}\label{eq:paterfamiliar}\begin{tikzcd}
  X' \ar{r}{g}\ar{d}{p}
  & Y' \ar{d}{q}
  \\
  X \ar{r}{f}
  & Y,
\end{tikzcd}\end{equation}
the functor $\D^*_!$ encodes a canonical isomorphism
\begin{equation}\label{eq:drawbore}
  \Ex^*_! : q^*f_! \simeq g_!p^*
\end{equation}
called the \emph{base change isomorphism}.

\subsubsection{}

We say that a preweave $\D^*_!$ is \emph{presentable} if every \inftyCat $\D(X)$ is presentable for every $X \in \cC$.
Equivalently, $\D^*_!$ takes values in $\PrL$.

\subsection{Smooth axioms}

\subsubsection{}
\label{sssec:smooth axioms}

Suppose given a left preweave on $(\cC, \cE)$.
Let $f : X \to Y$ be a morphism in $\cC$ and form, for every morphism $q: Y' \to Y$ in $\cC$, the cartesian square
\begin{equation}
  \begin{tikzcd}
    X' \ar{r}{g}\ar{d}{p}
    & Y' \ar{d}{q}
    \\
    X \ar{r}{f}
    & Y.
  \end{tikzcd}
\end{equation}
Consider the following conditions:
\begin{defnlist}[label={(Sm\arabic*)}, ref={(Sm\arabic*)}]
  \item\label{item:Sm1}
  The functor $g^*$ admits a left adjoint $g_\sharp$.

  \item\label{item:Sm2}
  The functor $g_\sharp$ satisfies the projection formula.
  That is, the exchange transformation $\Ex^{\otimes,*}_\sharp : g_\sharp(- \otimes g^*(-)) \to g_\sharp(-) \otimes (-)$ \eqref{eq:projectionshp} is invertible.

  \item\label{item:Sm3}
  The functor $f_\sharp$ commutes with $*$-inverse image.
  That is, the exchange transformation $\Ex^*_\sharp : g_\sharp p^* \to q^*f_\sharp$ \eqref{eq:Ex_sharp^*} is invertible.

  \item\label{item:Sm4}
  The functor $f_\sharp$ commutes with (shriekable) $!$-direct image.
  That is, if $q : Y' \to Y$ is shriekable, then the exchange transformation $\Ex_{\sharp,!} : f_\sharp p_! \to q_! g_\sharp$ is invertible.

  \item\label{item:Sm5}
  For any shriekable section $i : Y' \to X'$ of $g : X' \to Y'$, the functor $g_\sharp i_!$ is an equivalence.
\end{defnlist}
We say that $\D^*_!$ \emph{admits $\sharp$-direct image} for $f$ if all these conditions hold.
Given a lax symmetric monoidal functor $\D^* : \cC^\op \to \Cat$, we say that $\D^*$ admits $\sharp$-direct image for $f$ if the conditions \itemref{item:Sm1}, \itemref{item:Sm2}, and \itemref{item:Sm3} hold.\footnote{%
  Alternatively, $\D^*$ may be regarded as a left preweave $\D^*_!$ on $(\cC, \cC^{\mrm{iso}})$ where $\cC^{\mrm{iso}}$ is the subcategory of isomorphisms.
}

\subsection{Proper axioms}

Suppose given a left preweave on $(\cC, \cE)$.
Let $f : X \to Y$ be a morphism in $\cC$ and form, for every morphism $q: Y' \to Y$ in $\cC$, the cartesian square
\begin{equation}
  \begin{tikzcd}
    X' \ar{r}{g}\ar{d}{p}
    & Y' \ar{d}{q}
    \\
    X \ar{r}{f}
    & Y.
  \end{tikzcd}
\end{equation}
Consider the following conditions:
\begin{defnlist}[label={(Pr\arabic*)}, ref={(Pr\arabic*)}]
  \item\label{item:Pr1}
  The functor $g^*$ admits a right adjoint $g_*$.\footnote{%
    If $\D^*_!$ is a preweave, this condition should be interpreted in $\CatL$ rather than $\Cat$; i.e., the right adjoint $g_*$ is also required to be colimit-preserving.
  }

  \item\label{item:Pr2}
  The functor $g_*$ satisfies the projection formula.
  That is, the exchange transformation $\Ex^{\otimes,*}_* : g_*(-) \otimes (-) \to g_*(- \otimes g^*(-))$ \eqref{eq:projection*} is invertible.

  \item\label{item:Pr3}
  The functor $f_*$ commutes with $*$-inverse image.
  That is, the exchange transformation $\Ex^*_* : q^* f_* \to g_* p^*$ \eqref{eq:Ex_*^*} is invertible.

  \item\label{item:Pr4}
  The functor $f_*$ commutes with (shriekable) $!$-direct image.
  That is, the exchange transformation $\Ex_{!,*} : q_!g_* \to f_*p_!$ \eqref{eq:Ex_!*} is invertible.

  \item\label{item:Pr5}
  For any shriekable section $i : Y' \to X'$ of $g : X' \to Y'$, the functor $g_* i_!$ is an equivalence.
\end{defnlist}
We say that $\D^*_!$ \emph{admits $*$-direct image} for $f$ if all these conditions hold.
Given a lax symmetric monoidal functor $\D^* : \cC^\op \to \Cat$, we say that $\D^*$ admits $*$-direct image for $f$ if the conditions \itemref{item:Pr1}, \itemref{item:Pr2}, and \itemref{item:Pr3} hold.

\subsection{Weaves}

\subsubsection{}
\label{sssec:Csmpr}

Suppose given the following data:
\begin{itemize}
  \item $\cC$ is an \inftyCat with fibred products.
  \item $\cE$ is a class of morphisms in $\cC$ which is closed under base change.
  \item $\cC^{\sm}$ is a class of morphisms in $\cC$ which is closed under base change and contained in $\cE$.
  \item $\cC^{\pr}$ is a class of morphisms in $\cC$ which is closed under base change and contained in $\cE$.
\end{itemize}

We say a morphism is \emph{``smooth''} or \emph{``proper''} if it belongs to $\cC^\sm$ or $\cC^\pr$, respectively.
The quotation marks are part of the terminology here, but we will omit them when we are in an abstract context where there is no risk of confusion.

We say a morphism is \emph{shriekable} if it belongs to $\cE$.

\subsubsection{}

A (left) \emph{weave} on $(\cC,\cE,\cC^\sm,\cC^\pr)$ is a (left) preweave $\D^*_!$ on $(\cC,\cE)$ which admits $\sharp$-direct image for every smooth morphism in $\cC$ and $*$-direct image for every proper morphism in $\cC$.

\section{Poincaré duality and descent}

Let $\cC$ be an \inftyCat with fibred products, and $\cE$ a class of morphisms in $\cC$ which is closed under base change.
We let $\D^*_!$ be a left preweave on $(\cC,\cE)$.

\subsection{Poincaré duality}
\label{ssec:Poinc}

\subsubsection{}

Given a morphism $f : X \to Y$ in $\cC$ satisfying \itemref{item:Sm1} with shriekable diagonal, we set
\begin{equation}
  \Sigma_f := \pr_{2,\sharp} \Delta_! : \D(X) \to \D(X)
\end{equation}
where $\Delta : X \to X \fibprod_Y X$ is the diagonal and $\pr_2 : X \fibprod_Y X \to X$ is the second projection.

If $f$ satisfies \itemref{item:Sm2}, then $\Sigma_f$ is $\D(X)$-linear endofunctor of $\D(X)$, i.e., we have $\Sigma_f(-) \simeq (-) \otimes \Sigma_f(\un_X)$.
If $f$ moreover satisfies \itemref{item:Sm5}, then the object $\Sigma_f(\un_X) \in \D(X)$ is $\otimes$-invertible.

\begin{prop}\label{prop:fshptwist}
  If $\D^*_!$ satisfies \itemref{item:Sm1} and \itemref{item:Sm4} for a shriekable morphism $f : X \to Y$, then there is a canonical isomorphism $f_\sharp \simeq f_! \Sigma_f$.
\end{prop}
\begin{proof}
  Form the cartesian square
  \[\begin{tikzcd}
    X\fibprod_YX \ar{d}{\pr_1}\ar{r}{\pr_2}
    & X \ar{d}{f}
    \\
    X \ar{r}{f}
    & Y.
  \end{tikzcd}\]
  By \itemref{item:Sm4}, the exchange transformation $\Ex_{\sharp,!} : f_\sharp \pr_{1,!} \to f_! \pr_{2,\sharp}$ is invertible.
  Applying $\Delta_!$ on the right yields the canonical isomorphism
  \[
    f_\sharp
    \simeq f_\sharp \pr_{1,!} \Delta_!
    \xrightarrow{\Ex_{\sharp,!}} f_! \pr_{2,\sharp} \Delta_!
    \simeq f_! \Sigma_f
  \]
  asserted.
\end{proof}

\subsubsection{}

Passing to right adjoints from \propref{prop:fshptwist} yields:

\begin{cor}[Poincaré duality]\label{cor:poinc}
  If $\D^*_!$ satisfies \itemref{item:Sm1}, \itemref{item:Sm4} and \itemref{item:Sm5} for a shriekable morphism $f : X \to Y$, then there is a canonical isomorphism
  \begin{equation}
    f^! \simeq \Sigma_f f^* \simeq f^*(-) \otimes \Sigma_f(\un_X).
  \end{equation}
  In particular, the object $f^!(\un_Y) \simeq \Sigma_f(\un_X)$ is $\otimes$-invertible.
\end{cor}

\subsubsection{}

The following criterion is useful for checking \itemref{item:Sm5}.

\begin{lem}\label{lem:checkSm5}
  If $\D^*_!$ satisfies \itemref{item:Sm1}, \itemref{item:Sm2} and \itemref{item:Sm4} for a shriekable morphism $f : X \to Y$, then it satisfies \itemref{item:Sm5} for $f$ if and only if the object $\Sigma_g(\un_{X'}) \in \D(X')$ is $\otimes$-invertible for every base change $g : X' \to Y'$ of $f$ along a morphism $q: Y' \to Y$ in $\cC$.
\end{lem}
\begin{proof}
  Necessity is clear.
  Conversely, let $q: Y' \to Y$ be a morphism, $g : X' \to Y'$ the base change of $f$, and $i : Y' \to X'$ a shriekable section of $g$.
  The claim is that $g_\sharp i_!$ is invertible, or equivalently that $g_\sharp i_!(\un_{X'})$ is $\otimes$-invertible (since $g_\sharp$ and $i_!$ satisfy the projection formula, the former by \itemref{item:Sm2}).
  By \propref{prop:fshptwist}, we compute:
  \begin{equation*}
    g_\sharp i_!(\un_{X'})
    \simeq g_! (i_! (\un_{X'}) \otimes \Sigma_g(\un_{X'}))
    \simeq g_! (i_! i^* \Sigma_g(\un_{X'}))
    \simeq i^* \Sigma_g(\un_{X'})
  \end{equation*}
  using the projection formula for $g_\sharp$ \itemref{item:Sm2}, the projection formula for $i_!$, and that $g\circ i \simeq \id$ by assumption.
  Since $i^*$ is symmetric monoidal, it follows that if $\Sigma_g(\un_{X'}) \in \D(X')$ is $\otimes$-invertible, then so is $g_\sharp i_!(\un_{X'})$.
\end{proof}

\subsection{Descent for presheaves}

\subsubsection{}

Let $\cC$ be an \inftyCat with fibred products and finite coproducts, and $\cV$ an \inftyCat admitting totalizations.
Let $F : \cC^\op \to \cV$ be a $\cV$-valued presheaf of \inftyCats.
Let $(f_\alpha : Y_\alpha \to X)_\alpha$ be a collection of morphisms in $\cC$.

We say that $F$ satisfies \emph{\v{C}ech descent} along $(f_\alpha)_\alpha$ if the canonical morphism
\begin{equation}
  F(X) \to \Tot(F(\widetilde{Y}_\bullet))
\end{equation}
is invertible, where $\widetilde{Y}_\bullet$ is the \v{C}ech nerve of $\widetilde{Y} := \coprod_\alpha Y_\alpha \to X$.

When $F$ is \emph{radditive}\footnote{%
  cf. \cite{VoevodskyRadditive}
}, i.e., sends finite coproducts in $\cC$ to products in $\cV$, this amounts to the condition that the following is a limit diagram in $\cV$:
\[
  F(X)
  \to \prod_\alpha F(Y_\alpha)
  \rightrightarrows \prod_{\alpha,\beta} F(Y_{\alpha,\beta})
  \rightrightrightarrows \prod_{\alpha,\beta,\gamma} F(Y_{\alpha,\beta,\gamma})
  \rightrightrightrightarrows \cdots,
\]
where $Y_{\alpha_1,\ldots,\alpha_n} := Y_{\alpha_1} \fibprod_X \cdots \fibprod_X Y_{\alpha_n}$ for any subset of indices $\alpha_1,\ldots,\alpha_n$.

\subsubsection{}

Let $\cC^{\sm}$ be a class of morphisms in $\cC$ which is closed under base change.
Morphisms belonging to $\cC^{\sm}$ will be called \emph{smooth}.

\begin{lem}\label{lem:desc*}
  Let $\D^* : \cC^\op \to \Cat$ be a radditive presheaf of \inftyCats on $\cC$.
  Suppose that for every smooth morphism $f$ in $\cC$, $\D^*$ satisfies \itemref{item:Sm1} and \itemref{item:Sm3} for $f$ and $f^*$ admits a right adjoint $f_*$.
  Let $(f_\alpha : Y_\alpha \to X)_\alpha$ be a finite collection of smooth morphisms in $\cC$ and write
  $$f_{\alpha_1,\ldots,\alpha_n} : Y_{\alpha_1,\ldots,\alpha_n} \to X$$
  for any subset of indices $\alpha_1,\ldots,\alpha_n$.
  Then the following conditions are equivalent:
  \begin{thmlist}[leftmargin=3em,widest={(ii.b)}]
    \item\label{item:desc*/cat}
    The presheaf $\D^*$ satisfies \v{C}ech descent along $(f_\alpha : Y_\alpha \to X)_\alpha$.
  \end{thmlist}
  \begin{thmlist}[start=2,label={\em(\roman*.a)},ref={\upshape(\roman*.a)},leftmargin=3em,widest={(ii.b)}]
    \item\label{item:desc*/sheaf}
    For every $\sF \in \D(X)$ the following is a colimit diagram in $\D(X)$:
    \[
      \cdots
      \rightrightrightarrows \bigoplus_{\alpha,\beta} f_{\alpha,\beta,\sharp}f_{\alpha,\beta}^*(\sF)
      \rightrightarrows \bigoplus_\alpha f_{\alpha,\sharp}f_\alpha^*(\sF)
      \to \sF.
    \]
  \end{thmlist}
  \begin{thmlist}[start=2,label={\em(\roman*.b)},ref={\upshape(\roman*.b)},leftmargin=3em,widest={(ii.b)}]
    \item\label{item:desc*/sheafb}
    For every $\sF \in \D(X)$ the following is a limit diagram in $\D(X)$:
    \[
      \sF
      \to \prod_\alpha f_{\alpha,*}f_\alpha^*(\sF)
      \rightrightarrows \prod_{\alpha,\beta} f_{\alpha,\beta,*}f_{\alpha,\beta}^*(\sF)
      \rightrightrightarrows
      \cdots
    \]
  \end{thmlist}
  \begin{thmlist}[start=3,leftmargin=3em,widest={(ii.b)}]
    \item\label{item:desc*/cons}
    The family of functors $(f_\alpha^*)_\alpha$ is jointly conservative.
  \end{thmlist}
\end{lem}
\begin{proof}
  The implications \itemref{item:desc*/cat} $\Rightarrow$ \itemref{item:desc*/sheaf} and \itemref{item:desc*/cat} $\Rightarrow$ \itemref{item:desc*/sheafb} follow from \ssecref{ssec:limstar} and \ssecref{ssec:colimshp}, respectively.
  It is obvious that \itemref{item:desc*/sheaf} $\Rightarrow$ \itemref{item:desc*/cons} and \itemref{item:desc*/sheafb} $\Rightarrow$ \itemref{item:desc*/cons}.
  It remains to show \itemref{item:desc*/cons} $\Rightarrow$ \itemref{item:desc*/cat}.
  By additivity, the claim is that the functor
  \[
    F^* : \D(X)
    \to \Tot(\D(\widetilde{Y}_\bullet))
    \simeq \Tot\Big(\prod_{\alpha_1,\ldots,\alpha_\bullet} \D(Y_{\alpha_1,\ldots,\alpha_\bullet})\Big),
  \]
  as in \eqref{eq:F^*}, is an equivalence, where $\widetilde{Y} = \coprod_{\alpha} Y_\alpha$ and $\widetilde{Y}_\bullet$ is its \v{C}ech nerve.
  For this we apply the dual of \cite[Cor.~4.7.5.3]{LurieHA}, whose conditions are verified in view of our assumption that for every smooth morphism $f$, $f^*$ preserves limits and colimits and that $f_*$ commutes with $*$-inverse image along smooth morphisms (by the dual of \itemref{item:Sm3}).
  The conclusion is that the right adjoint $F_*$ \eqref{eq:F_*} is fully faithful, and that $F^*$ is an equivalence if and only if the functors $f_\alpha^*$ are jointly conservative.
\end{proof}

\subsection{Descent for preweaves}

\subsubsection{}

Assume that $\cC$ admits finite coproducts.
A left preweave $\D^*_!$ is \emph{radditive} if $\D^*$ and $\D^!$ are radditive.

\begin{prop}\label{prop:descweave}
  Let $\D^*_!$ be a preweave on $(\cC,\cE)$ which is radditive.
  Suppose that $\D^*_!$ admits $\sharp$-direct image for every smooth shriekable morphism $f$ in $\cC$.
  Let $(f_\alpha : Y_\alpha \to X)_\alpha$ be a finite collection of smooth shriekable morphisms in $\cC$.
  Then the following conditions are equivalent:
  \begin{thmlist}
    \item\label{item:descweave/cons*}
    The family $(f_\alpha^*)_\alpha$ is jointly conservative.
    \item\label{item:descweave/cech*}
    The presheaf $\D^* : \cC^\op \to \Cat$ satisfies\footnote{%
      If $\D(X)$ is presentable for every $X$, then by \cite[Prop.~5.5.3.13]{LurieHTT} this is equivalent to the same statement for $\D^* : \cC^\op \to \PrL$.
    } \v{C}ech descent along $(f_\alpha)_\alpha$.
    \item\label{item:descweave/cons!}
    The family $(f_\alpha^!)_\alpha$ is jointly conservative.
    \item\label{item:descweave/cech!}
    The presheaf $\D^! : \cC^\op \to \Cat$ satisfies\footnote{%
      If $\D(X)$ is presentable for every $X$, then by \cite[Thm.~5.5.3.18]{LurieHTT} this is equivalent to the same statement for $\D^! : \cC^\op \to \PrR$.
    } \v{C}ech descent along $(f_\alpha)_\alpha$.
  \end{thmlist}
\end{prop}
\begin{proof}
  Applying \lemref{lem:desc*} to the presheaf $\D^*$ shows that \itemref{item:descweave/cons*} $\Leftrightarrow$ \itemref{item:descweave/cech*}.
  For every smooth shriekable morphism $f$ in $\cC$, we have $f^! \simeq f^*(-) \otimes \Sigma_f(\un)$ by Poincaré duality, with $\Sigma_f(\un)$ a $\otimes$-invertible object \ssecref{ssec:Poinc}.
  This shows that \itemref{item:descweave/cons*} $\Leftrightarrow$ \itemref{item:descweave/cons!}.
  Finally, to show \itemref{item:descweave/cech!} $\Leftrightarrow$ \itemref{item:descweave/cons!} by applying \lemref{lem:desc*} to $\D^!$, which applies because for every smooth morphism $f$, $f^!$ admits a left adjoint $f_!$ commuting with $*$-inverse image, and $f^! \simeq f^*(-) \otimes \Sigma_f(\un)$ admits a right adjoint $f_*(- \otimes \Sigma_f(\un)^{\otimes -1})$.
\end{proof}

\section{Lisse extension}

\subsection{Artin stacks}
\label{ssec:Art}

\subsubsection{}
\label{sssec:Csmcov}

Let $(\cC,\cC^\sm,\cC^\etcov)$ be the following data:
\begin{itemize}
  \item $\cC$ is an \inftyCat with fibred products.
  \item $\cC^{\sm}$ is a class of morphisms in $\cC$ which is closed under base change.
  \item $\cC^{\etcov}$ is a class of morphisms in $\cC$ which is closed under base change and contained in $\cC^\sm$.
\end{itemize}
We assume that whenever $f : X \to Y$ is a morphism in $\cC$ and $p : X' \to X$ is a morphism in $\cC^\etcov$, if the composite $f \circ p$ belongs to $\cC^\sm$ then so does $f$.

\subsubsection{}

We fix a context $(\cC, \cC^\sm, \cC^\etcov)$ as in \sssecref{sssec:Csmcov}.

A \emph{space} is an object of $\cC$.
A morphism of spaces is \emph{smooth}, resp. \emph{étale covering}, if it belongs to $\cC^\sm$, resp. $\cC^\etcov$.

We also write $\Spc(\cC)$, or simply $\Spc$ when there is no risk of confusion, for the \inftyCat of spaces (i.e., $\Spc := \Spc(\cC) := \cC$).

\subsubsection{}
\label{sssec:smooth covering}

A morphism of spaces $f : X \to Y$ \emph{admits étale-local sections} if there exists an étale covering morphism $Y' \to Y$ such that the base change $X\fibprod_Y Y' \to Y'$ admits a section.
A \emph{smooth covering} morphism is a smooth morphism which admits étale-local sections.

\subsubsection{}

A \emph{stack} is a presheaf of anima on $\Spc$ which satisfies \v{C}ech descent with respect to étale covering morphisms.

\subsubsection{}
\label{sssec:nArt}

A stack is \emph{$0$-Artin} if it is a space\footnote{%
  We will identify spaces with the stacks they represent.
}.
For $n>0$ we define inductively:
\begin{defnlist}
  \item A morphism of stacks $f : \sX \to \sY$ is \emph{$(n-1)$-representable} if for every space $Y$ and every morphism $Y \to \sY$, the fibred product $\sX \fibprod_\sY Y$ is $(n-1)$-Artin.
  \item A $(n-1)$-representable morphism $f : \sX \to \sY$ is \emph{smooth}, resp. \emph{étale covering}, if for every space $Y$ and every morphism $Y \to \sY$, the base change $\sX \fibprod_\sY Y \to Y$ is a smooth, resp. étale covering, morphism.
  \item A stack $\sX$ is \emph{$n$-Artin} if its diagonal $\Delta_\sX : \sX \to \sX \times \sX$ is $(n-1)$-representable, and there exists a space $X$ and a morphism\footnote{%
    which is automatically $(n-1)$-representable when the diagonal $\Delta_\sX$ is $(n-1)$-representable
  } $p : X \to \sX$ which is smooth and admits étale-local sections.
  \item If $Y$ is a space and $\sX$ is $n$-Artin, a morphism $f : \sX \to Y$ is \emph{smooth}, resp. \emph{étale covering}, if there exists\footnote{%
    It follows from the condition in \sssecref{sssec:Csmcov} that if this holds for a single such $(X,p)$, then it holds for any such.
  } a space $X$ and a smooth morphism $p : X \to \sX$ admitting étale-local sections such that the composite $f\circ p : X \to Y$ is a smooth, resp. étale covering, morphism of spaces.
\end{defnlist}
We say that a stack is \emph{Artin} if it is $n$-Artin for some $n$, and a morphism of stacks is \emph{eventually representable} if it is $n$-representable for some $n$.
Any morphism of $n$-Artin stacks is $n$-representable.

We denote by $\Art(\cC)$, or simply $\Art$ when there is no risk of confusion, the \inftyCat of Artin stacks (by definition a full subcategory of the \inftyCat of stacks).
We denote by $\Art_n(\cC)$, or simply $\Art_n$, the full subcategory of $\Art$ spanned by $n$-Artin stacks.

\subsubsection{}

A morphism of Artin stacks $f : \sX \to \sY$ \emph{admits étale-local sections} if for every space $Y$ and every morphism $Y \to \sY$, there exists an étale covering morphism of spaces $Y' \to Y$ such that $\sX\fibprod_\sY Y' \to Y'$ admits a section.

A morphism of Artin stacks $f : \sX \to \sY$ is \emph{smooth covering} if it is smooth and admits étale-local sections.

\subsubsection{}
\label{sssec:smoothloc}

Let $\cE_0 \sub \Spc$ be a class of morphisms closed under base change.
We say $\cE_0$ satisfies \emph{smooth-locality on the target} if a morphism of spaces $f : X \to Y$ lies in $\cE_0$ if and only if there exists a smooth covering morphism $Y' \surj Y$ such that the base change $X \fibprod_Y Y' \to Y'$ belongs to $\cE_0$.

\subsubsection{}
\label{sssec:E_lis^0}

Let $\cE_0 \sub \Spc$ be a class of morphisms closed under base change and satisfying smooth-locality on the target.

We define a new class $\cE_0^\lis$ of morphisms in $\Art$ as follows.
A representable morphism of Artin stacks $f : \sX \to \sY$ belongs to $\cE_0^\lis$ if and only if for every space $V$ and every smooth covering morphism $v : V \surj \sY$, the morphism of spaces $\sX \fibprod_\sY V \to V$ belongs to $\cE_0$.

It is clear that this property is closed under composition and indeed defines a subcategory $\cE_\lis^0 \sub \Art$, and that it is closed under base change.
If $\cE_0$ is closed under sections, then $\cE_\lis^0$ is closed under sections.

\subsubsection{}

Let $\cE_0 \sub \Spc$ be a class of morphisms closed under base change.
We say $\cE_0$ satisfies \emph{smooth-locality on the source and target} if a morphism of spaces $f : X \to Y$ lies in $\cE_0$ if and only if there exists a commutative square
\[\begin{tikzcd}
  U \ar{r}{f_0}\ar[twoheadrightarrow]{d}{u}
  & V \ar[twoheadrightarrow]{d}{v}
  \\
  X \ar{r}{f}
  & Y
\end{tikzcd}\]
where $f_0$ is a morphism in $\cE_0$ and $v : V \surj Y$ and $(u,f_0) : U \surj X \fibprod_Y V$ are smooth covering morphisms.

\subsubsection{}
\label{sssec:E_lis}

Let $\cE_0 \sub \Spc$ be a class of morphisms closed under base change and satisfying smooth-locality on the source and target.

We define a new class $\cE_\lis$ of morphisms in $\Art$ as follows.
A morphism of Artin stacks $f : \sX \to \sY$ belongs to $\cE_\lis$ if and only if there exists a commutative square
\[\begin{tikzcd}
  X \ar{r}{f_0}\ar{d}{u}
  & Y \ar{d}{v}
  \\
  \sX \ar{r}{f}
  & \sY
\end{tikzcd}\]
where $X$ and $Y$ are spaces, $f_0$ is a morphism in $\cE_0$, and $v : Y \surj \sY$ and $(u,f_0) : X \surj \sX \fibprod_\sY Y$ are smooth covering morphisms.\footnote{%
  Equivalently (see e.g. \cite[Tag~06FL]{Stacks}), for \emph{any} commutative square as above where $X$ and $Y$ are spaces and $v : Y \surj \sY$ and $(u,f_0) : X \surj \sX\fibprod_\sY Y$ are smooth covering, the morphism $f_0$ belongs to $\cE_0$.
}

This property is closed under composition and indeed defines a subcategory $\cE_\lis \sub \Art$, see e.g. \cite[Tag~06FL]{Stacks}.
It is also closed under base change (see e.g. \cite[Tag~06FQ]{Stacks}).
If $\cE_0$ is closed under sections, then so is $\cE_\lis$.\footnote{%
  Let $f : \sX \to \sY$ and $g : \sY \to \sS$ be morphisms of Artin stacks with $g \in \cE_\lis$ and $g \circ f \in \cE_\lis$.
  Let $S \surj \sS$ be a smooth covering where $S$ is a space, $Y \surj \sY \fibprod_\sS S$ a smooth covering where $Y$ is a space, and $X \surj \sX \fibprod_\sX Y$ a smooth covering where $X$ is a space.
  Since $g$ in $\cE_\lis$ and $g \circ f \in \cE_\lis$ we deduce that $Y \to S$ lies in $\cE_0$ and $X \to Y \to S$ lies in $\cE_0$.
  Since $\cE_0$ is closed under two-of-three, it follows that $X \to Y$ lies in $\cE_0$ and hence $f : \sX \to \sY$ lies in $\cE_\lis$.
}

Note that the class $\cE_\lis^0$ \sssecref{sssec:E_lis^0} consists of the morphisms of $\Art$ that are representable and belong to $\cE_\lis$.

\subsection{Lisse extension of presheaves}
\label{ssec:lisextpresheaf}

We fix a context $(\cC, \cC^\sm, \cC^\etcov)$ as in \sssecref{sssec:Csmcov} and adopt the terminology of \ssecref{ssec:Art}.

\subsubsection{}

Let $\sX$ be an Artin stack.
Consider the slice \inftyCat $\Art_{/\sX}$, whose objects are pairs $(\sS,s)$ consisting of an Artin stack $\sS$ and a morphism $s : \sS \to \sX$, and whose morphisms $f : (\sS',s') \to (\sS,s)$ are commutative triangles
\[ \begin{tikzcd}
  S' \ar{rr}{f}\ar[swap]{rd}{s'}
  & & S \ar{ld}{s}
  \\
  & \sX. &
\end{tikzcd} \]
We denote by $\Pt_\sX$ the full subcategory spanned by pairs $(S,s)$ where $S$ is a space.

We denote by $\Lis_\sX$ the full subcategory of $\Pt_\sX$ spanned by pairs $(S, s)$ where $s : S \to \sX$ is \emph{smooth}.
Similarly, $\LisStk_\sX$ is the full subcategory of $\Art_{/\sX}$ spanned by pairs $(\sS,s)$ where $s : \sS \to \sX$ is smooth.

\subsubsection{}

Let $\sX$ be an Artin stack.
Let $F : \Lis_\sX^\op \to \cV$ be a presheaf valued in an \inftyCat $\cV$ with limits.
The \emph{lisse extension} of $F$ is the presheaf $$F^\lis : \LisStk_\sX^{\op} \to \cV$$ defined as the right Kan extension of $F$ along the fully faithful functor $\Lis_\sX \hook \LisStk_\sX$.
In particular, we have
\[
  F^\lis(\sX) \simeq \lim_{(S,s)} F(S)
\]
where the limit is taken over $(S,s)\in\Lis_\sX$.

\subsubsection{}
\label{sssec:spclis}

Given a presheaf $F : (\Spc)^{\op} \to \cV$, its \emph{lisse extension} $F^\lis$ is its right Kan extension
$$F^\lis : (\Art)^{\op} \to \cV$$
along $\Spc \hook \Art$.
In particular, for every $\sX \in \Art$ we have
\[
  F^\lis(\sX) \simeq \lim_{(S,s)} F(S),
\]
where the limit is taken over pairs $(S,s) \in \Pt_\sX$.

\subsubsection{}

Let $\cV$ be an \inftyCat with limits.
The following statement is key to our analysis of lisse extensions.

\begin{prop}[Descent]\label{prop:coccus}\leavevmode
  \begin{thmlist}
    \item
    If $F : (\Spc)^{\op} \to \cV$ satisfies \v{C}ech descent along étale covering morphisms of spaces, then $F^\lis$ satisfies \v{C}ech descent along any smooth covering morphism of Artin stacks.

    \item
    Let $\sX$ be an Artin stack.
    If $F : (\Lis_\sX)^{\op} \to \cV$ satisfies \v{C}ech descent along étale covering morphisms of spaces, then $F^\lis$ satisfies \v{C}ech descent along any smooth covering morphism of Artin stacks.
  \end{thmlist}
\end{prop}

We will prove \propref{prop:coccus} in \ssecref{ssec:coccusproof} below.
We first derive some important consequences.

Denote by $\Shv(\Spc)_\cV$ the \inftyCat of presheaves on $\Spc$ satisfying \v{C}ech descent along smooth covering morphisms, and by $\Shv(\Art)_\cV$ the \inftyCat of presheaves on $\Art$ satisfying \v{C}ech descent along smooth covering morphisms of Artin stacks.

\begin{cor}\label{cor:Shvreseq}\leavevmode
  \begin{thmlist}
    \item
    The restriction functor
    \begin{equation}\label{eq:Shvreseq1}
      \Shv(\Art)_\cV \to \Shv(\Spc)_\cV
    \end{equation}
    is an equivalence, with inverse $F\mapsto F^\lis$.

    \item
    For any Artin stack $\sX$, the restriction functor
    \begin{equation}\label{eq:Shvreseq2}
      \Shv(\LisStk_\sX)_\cV \to \Shv(\Lis_\sX)_\cV
    \end{equation}
    is an equivalence, with inverse $F\mapsto F^\lis$.
  \end{thmlist}
\end{cor}
\begin{proof}
  It follows from \propref{prop:coccus} that the assignments $F \mapsto F^\lis$ preserve the descent conditions and restrict to right adjoints of the restriction functors \eqref{eq:Shvreseq1} and \eqref{eq:Shvreseq2}.
  We will only write the proof for \eqref{eq:Shvreseq1}, as the same argument works for \eqref{eq:Shvreseq2}.

  Note that the case of general $\cV$ follows from the universal case $\cV=\Ani$.
  Indeed, denoting by $\Shv(\Spc)$ and $\Shv(\Art)$ the respective \inftyCats formed with $\cV=\Ani$, note that $\Shv(\Spc)_\cV$ is equivalent to the \inftyCat of limit-preserving functors $\Shv(\Spc)^\op \to \cV$, while $\Shv(\Art)_\cV$ is equivalent to the \inftyCat of limit-preserving functors $\Shv(\Art)^\op \to \cV$.
  This follows from the universal properties describing colimit-preserving functors out of $\Shv(\Spc)$ and $\Shv(\Art)$ into an \inftyCat with colimits.
  
  Since $\Ani$ admits colimits, the restriction functor \eqref{eq:Shvreseq1} in this case also admits a \emph{left} adjoint, given by left Kan extending along $\Spc\hook\Art$ and then localizing.
  This is fully faithful and colimit-preserving, so it will suffice to show that it generates under colimits.
  By induction it is enough to show that if $\sY$ is an $(n+1)$-Artin stack, $n\ge 0$, then it belongs to the full subcategory of $\Shv(\Art)$ generated under colimits by $n$-Artin stacks.
  Let $p : Y \surj \sY$ a smooth covering morphism where $Y$ is a space.
  By construction, $\abs{Y_\bullet} \to \sY$ is invertible in $\Shv(\Art)$.
  Since $\sY$ has $n$-representable diagonal, each $Y_n$ is $n$-Artin.
  The claim follows.
\end{proof}

\begin{cor}\label{cor:entropionize}
  If $F : (\Spc)^\op \to \cV$ satisfies \v{C}ech descent along étale covering morphisms, then for every $\sX \in \Art$ there is a canonical isomorphism
  \[
    F^\lis|_{\LisStk_\sX} \to (F|_{\Lis_\sX})^\lis
  \]
  of presheaves on $\LisStk_\sX$.
  In particular, the canonical map
  \begin{equation}
    F^\lis(\sX) \simeq \lim_{(S,s)\in\Pt_\sX} F(S)
    \to \lim_{(S,s)\in\Lis_\sX} F(S).
  \end{equation}
  is invertible.
\end{cor}

\begin{proof}
  It follows from \corref{cor:Shvreseq} that $(F|_{\Lis_\sX})^\lis$ is the unique presheaf $\LisStk_\sX^\op \to \cV$ which
  \begin{inlinelist}
    \item extends $F|_{\Lis_\sX}$, and
    \item satisfies \v{C}ech descent along smooth covering morphisms in $\LisStk_\sX$.
  \end{inlinelist}
  It will thus suffice to check the same properties for $F^\lis|_{\LisStk_\sX}$.
  The first holds because $(F^\lis|_{\LisStk_\sX})|_{\Lis_\sX} \simeq (F^\lis|_{\Spc})|_{\Lis_\sX}$, and $F^\lis|_{\Spc} \simeq F$ by construction.
  The second holds because $F^\lis$ satisfies \v{C}ech descent along smooth covering morphisms in $\Art$ (\propref{prop:coccus}), so its restriction $(F^\lis)|_{\LisStk_\sX}$ satisfies \v{C}ech descent along smooth covering morphisms in $\LisStk_\sX$.
\end{proof}

\begin{cor}[Saturation]\label{cor:sat}\leavevmode
  \begin{thmlist}
    \item \label{item:sat/one}
    If $F : (\Spc)^\op \to \cV$ satisfies \v{C}ech descent along étale covering morphisms of spaces, then for every smooth morphism of Artin stacks $f : \sX \to \sY$, the canonical map
    \begin{equation}\label{eq:sat1}
      F^\lis(\sX) \to \lim_{(T,t)\in\Lis_\sY} F^\lis(\sX \fibprod_\sY T)
    \end{equation}
    is invertible.

    \item \label{item:sat/two}
    Let $\sY$ be an Artin stack and $F : (\Lis_\sY)^\op \to \cV$ a presheaf.
    If $F$ satisfies \v{C}ech descent along étale covering morphisms of spaces, then for every morphism of Artin stacks $f : \sX \to \sY$, the canonical map
    \begin{equation}\label{eq:sat2}
      F^\lis(\sX) \to \lim_{(T,t)\in\Lis_\sY} F^\lis(\sX \fibprod_\sY T)
    \end{equation}
    is invertible.
  \end{thmlist}
\end{cor}
\begin{proof}
  \itemref{item:sat/two}:
  Denote by $i_\sY : \Lis_\sY \hook \LisStk_\sY$ the inclusion.
  Given a morphism $f : \sX \to \sY$ of Artin stacks, denote by
  \[\underline{\smash{\phi}}_f : \LisStk_\sY \to \LisStk_\sX\]
  the base change functor and
  \[\phi_f : \Lis_\sY \to \LisStk_\sX\]
  its restriction along $i_\sY$.
  Since smooth covering morphisms in $\LisStk_\sY$ are preserved under base change along $\sX \to \sY$, the restriction functors $\underline{\smash{\phi}}_f^*$ and $\phi_f^*$ preserve the descent conditions.
  Consider the natural transformation
  \[
    \underline{\smash{\phi}}_f^*
    \xrightarrow{\mrm{unit}} i_{\sY,*} i^*_\sY \underline{\smash{\phi}}_f^*
    \simeq i_{\sY,*} \phi_f^*
  \]
  of functors $\Shv_{\sm}(\LisStk_\sX)_\cV \to \Shv_{\sm}(\LisStk_\sY)_\cV$.
  Since $i_\sY^*$ is an equivalence (\corref{cor:Shvreseq}), this natural transformation is invertible.
  Applying this to $F \in \Shv_{\sm}(\LisStk_\sY)_\cV$ and evaluating on the object $\sY \in \LisStk_\sY$ yields that the canonical map
  \[
    F^\lis(\sX) \to \lim_{(T,t)\in\Lis_\sY} F(\sX \fibprod_\sY T)
  \]
  is invertible, as claimed.

  \noindent\itemref{item:sat/one}:
  Given $F : (\Spc)^\op \to \cV$ and a smooth morphism $f : \sX \to \sY$, note that by \corref{cor:entropionize} we have
  \[ F^\lis(\sX) \simeq F^\lis|_{\LisStk_\sY}(\sX) \simeq (F|_{\Lis_\sY})^\lis(\sX) \]
  and similarly $F^\lis(\sX\fibprod_\sY T) \simeq (F|_{\Lis_\sY})^\lis(\sX\fibprod_\sY T)$ for every $(T,t)\in\Lis_\sY$.
  Under these identifications, the map in question is the isomorphism of \itemref{item:sat/two} for the presheaf $(F|_{\Lis_\sY})$.
\end{proof}

\subsection{Proof of \propref{prop:coccus}}
\label{ssec:coccusproof}

We let $F : (\Spc)^\op \to \cV$ be a presheaf and $F^\lis : (\Art)^\op \to \cV$ its lisse extension.

\begin{lem}\label{lem:cofin1}
  Let $f : \sX \to \sY$ be an $n$-representable morphism of Artin stacks.
  Then the canonical map
  \begin{equation}\label{eq:cofin1}
    F^\lis(\sX) \to \lim_{(\sT,t)\in\Pt_\sY(\Art_n)} F^\lis(\sX \fibprod_\sY \sT)
  \end{equation}
  is invertible.
\end{lem}
\begin{proof}
  Since $f$ is $n$-representable, there is a base change functor $f^* : \Pt_\sY(\Art_n) \to \Pt_\sX(\Art_n)$.
  Note that it admits a left adjoint, sending $(\sS,s)\in\Pt_\sX(\Art_n)$ to $(\sS,f\circ s) \in \Pt_\sY(\Art_n)$, so it is cofinal.
\end{proof}

\begin{lem}
  Let $F : (\Spc)^\op \to \cV$ be a presheaf satisfying \v{C}ech descent along smooth covering morphisms.
  Then $F^\lis : (\Art)^\op \to \cV$ satisfies \v{C}ech descent along smooth covering morphisms of Artin stacks.
\end{lem}
\begin{proof}
  By induction, it will suffice to show that if $F^\lis$ satisfies \v{C}ech descent along smooth covering morphisms of $(n-1)$-Artin stacks, $n>0$, then it satisfies \v{C}ech descent along smooth covering morphisms of $n$-Artin stacks.

  Let $f : \sX \surj \sY$ be an $(n-1)$-representable smooth covering morphism of $n$-Artin stacks.
  By \lemref{lem:cofin1}, the canonical map
  \begin{equation}\label{eq:usurluxt}
    F^\lis(\sY) \to \Tot(F^\lis(\sX_\bullet))
  \end{equation}
  is the limit over $(\sT,t) \in \Pt_\sY(\Art_{n-1})$ of the maps
  \begin{equation}\label{eq:jhqaivsy}
    F^\lis(\sT) \to \Tot(F^\lis(\sX_\bullet \fibprod_\sY \sT)).
  \end{equation}
  Since the base change $\sX\fibprod_\sY\sT \surj \sT$ is a smooth covering morphism of $(n-1)$-Artin stacks, \eqref{eq:jhqaivsy} is invertible by the inductive hypothesis on $F$.
  Hence \eqref{eq:usurluxt} is invertible.

  Now let $f : \sX \surj \sY$ be any smooth covering morphism of $n$-Artin stacks and let us show that the canonical map
  \begin{equation}\label{eq:usurluxt2}
    F^\lis(\sY) \to \Tot(F^\lis(\sX_\bullet))
  \end{equation}
  is invertible.
  Let $q : Y \surj \sY$ be a smooth covering morphism where $Y$ is a space.
  Since $f$ is smooth covering, we may assume (up to replacing $Y$ by an étale cover $Y' \surj Y$) that $\sX\fibprod_\sY Y \surj Y$ admits a section.
  Since $q$ is $(n-1)$-representable, $F^\lis$ satisfies \v{C}ech descent along $q$ and its base change $\sX \fibprod_\sY Y \surj \sX$ by the discussion above.
  Hence \eqref{eq:usurluxt2} is the totalization (in the $\filledsquare$ direction) of the cosimplicial diagram of maps
  \begin{equation}\label{eq:jhqaivsy2}
    F^\lis(Y_\filledsquare) \to \Tot_\bullet(F^\lis(\sX_\bullet \fibprod_\sY Y_\filledsquare)),
  \end{equation}
  where $\Tot_\bullet$ indicates the totalization formed in the $\bullet$ direction.
  This is levelwise invertible since each $\sX_\bullet\fibprod_\sY Y_n \surj Y_n$, $[n]\in\bDelta$, admits a section.
\end{proof}

\subsection{Lisse extension of preweaves}

\subsubsection{}

We fix a context $(\cC, \cC^\sm, \cC^\etcov)$ as in \sssecref{sssec:Csmcov} and adopt the terminology of \ssecref{ssec:Art}.

\subsubsection{}
\label{sssec:mainthms}

Let $\cE_0$ be a class of morphisms in $\Spc$ which is closed under base change and sections, contains all smooth morphisms of spaces, and satisfies smooth-locality on the source and target \sssecref{sssec:smoothloc}.
Let $\cE_\lis^0$ and $\cE_\lis$ be the classes of morphisms in $\Art$ defined in \sssecref{sssec:E_lis^0} and \sssecref{sssec:E_lis}, respectively; recall that $\cE_\lis^0$ consists of the \emph{representable} morphisms in $\cE_\lis$.

Our first main result is as follows.

\begin{thm}\label{thm:extweave}
  Let $\D^*_!$ be a left preweave on $(\Spc,\cE_0)$.
  Then we have:
  \begin{thmlist}
    \item\label{item:extweave/repr}
    There exists a unique extension of $\D^*_!$ to a left preweave $\D^{\lis,*}_!$ on $(\Art, \cE_\lis^0)$ whose underlying presheaf $\D^{\lis,*} : (\Art)^\op \to \Cat$ is the lisse extension of $\D^* : (\Spc)^\op \to \Cat$.
    If $\D^*_!$ is a presentable preweave, then $\D^{\lis,*}_!$ is a presentable preweave.

    \item\label{item:extweave/main}
    If $\D^*_!$ is a presentable preweave which admits $\sharp$-direct image with respect to smooth morphisms of spaces, and $\D^* : (\Spc)^\op \to \Cat$ satisfies \v{C}ech descent along étale covering morphisms of spaces, then $\D^*_!$ extends uniquely to a presentable preweave $\D^{\lis,*}_!$ on $(\Art, \cE_\lis)$ whose underlying presheaf $\D^{\lis,*} : (\Art)^\op \to \Cat$ satisfies \v{C}ech descent along smooth covering morphisms of Artin stacks.
  \end{thmlist}
\end{thm}

\subsubsection{}
\label{sssec:Cpr}

For the next statement, fix another class of morphisms $\cC^\pr \sub \cC$ which is closed under base change.
A morphism of spaces is \emph{``proper''} if the corresponding morphism of $\cC$ lies in $\cC^\pr$.
Consider the induced class $(\cC^\pr)_\lis^0$ of \emph{``proper representable''} morphisms in $\Art$ as defined in \sssecref{sssec:E_lis^0}.\footnote{%
  As before, we omit the quotation marks in these terminologies when we are in an abstract context where there is no risk of ambiguity.
}

\begin{thm}\label{thm:presaxioms}
  Let $\D^*_!$ be a presentable preweave on $(\Art,\cE_\lis)$ whose underlying presheaf $\D^*$ satisfies \v{C}ech descent along smooth covering morphisms.
  Assume that the restriction of $\D^*_!$ to $(\Spc,\cE_0)$ admits $\sharp$-direct image for all smooth morphisms of spaces.
  Then we have:
  \begin{thmlist}
    \item\label{item:presaxioms/sm}
    $\D^*_!$ admits $\sharp$-direct image for smooth morphisms of Artin stacks.
    
    \item\label{item:presaxioms/pr}
    $\D^*_!$ admits $*$-direct image for proper representable morphisms of Artin stacks if and only if its restriction to $(\Spc,\cE_0)$ admits $*$-direct image for proper morphisms of spaces.
  \end{thmlist}
\end{thm}

\subsubsection{}

Let $\cC, \cC^\sm, \cC^\etcov, \cC^\pr, \cE_0, \cE_\lis$ be as above, and terminology as in \ssecref{ssec:Art}.
Combining Theorems~\ref{thm:extweave} and \ref{thm:presaxioms} shows that any presentable weave on $(\Spc,\cE,\cC^\sm,\cC^\pr)$ extends to a presentable weave on $(\Art,\cE_\lis,(\cC^\sm)_\lis,(\cC^\pr)_\lis^0)$.
Here $(\cC^\sm)_\lis$, defined as in \sssecref{sssec:E_lis}, is the class of smooth morphisms of Artin stacks, and $(\cC^\pr)_\lis^0$, defined as in \sssecref{sssec:E_lis^0}, is the class of proper representable morphisms of Artin stacks.

\begin{cor}\label{cor:weave}
  Let $\D^*_!$ be a presentable weave on $(\Spc,\cE,\cC^\sm,\cC^\pr)$.
  If the underlying presheaf $\D^*$ satisfies \v{C}ech descent along étale covering morphisms, then there is a unique presentable weave $\D^{\lis,*}_!$ on $$(\Art,\cE_\lis,(\cC^\sm)_\lis,(\cC^\pr)_\lis^0)$$ whose underlying presheaf $\D^{\lis,*}$ satisfies \v{C}ech descent along smooth covering morphisms of Artin stacks.
\end{cor}

\sssec{}

We can moreover extend \corref{cor:weave} to non-Artin stacks.

Let $\cP$ be a class of morphisms in $\Art$.
We denote by $^{+}\!\cP$ the class of morphisms in $\Stk$ defined as follows: a morphism of stacks $X \to Y$ belongs to $^{+}\!\cP$ if and only if it is eventually representable, and for every morphism $V \to Y$ with $V \in \Spc$, the base change $X\fibprod_Y V \to V$ belongs to $\cP$.

\begin{cor}\label{cor:weavenonart}
  Let $\D^*_!$ be a presentable weave on $(\Spc,\cE,\cC^\sm,\cC^\pr)$.
  If the underlying presheaf $\D^*$ satisfies \v{C}ech descent along étale covering morphisms, then there is a unique presentable weave $\D^{\lis,*}_!$ on $$(\Stk,^{+}\!\cE_\lis,^{+}\!(\cC^\sm)_\lis,^{+}\!(\cC^\pr)_\lis^0)$$ whose underlying presheaf $\D^{\lis,*}$ satisfies \v{C}ech descent along smooth covering morphisms of Artin stacks.
\end{cor}
\begin{proof}
  Follows immediately from \propref{prop:extobj} below.
\end{proof}

\section{Smooth axioms}

In this section we prove \thmref{thm:presaxioms}\itemref{item:presaxioms/sm}.

\subsection{(Sm1)}
\label{ssec:Sm1}

\subsubsection{}

Note that \itemref{item:Sm1} only depends on the underlying presheaf $\D^{\lis,*}$.
We have:

\begin{prop}\label{prop:Sm1}
  Let $\D^* : (\Art)^\op \to \PrL$ be a presentable presheaf of \inftyCats which satisfies \v{C}ech descent along smooth covering morphisms.
  If its restriction to $\Spc$ satisfies \itemref{item:Sm1} for every smooth morphism of spaces, then it satisfies \itemref{item:Sm1} for every smooth morphism of Artin stacks.
\end{prop}

\begin{proof}
  Let $f : \sX \to \sY$ be a smooth morphism of Artin stacks.
  The functor $f^* : \D(\sY) \to \D(\sX)$ is identified by Corollaries~\ref{cor:entropionize} and \ref{cor:sat} with the limit over $(T,t)\in\Lis_\sY$ and $(S,s)\in\Lis_{\sX_T}$ of the functors
  \begin{equation*}
    \D(T)
    \xrightarrow{f_T^*} \D(\sX_T)
    \xrightarrow{s^*} \D(S),
  \end{equation*}
  where $f_T : \sX_T \to T$ is the base change.
  Since $f_T \circ s$ is a smooth morphism of spaces, $s^*f_T^*$ is a right adjoint by assumption.
  Since the forgetful functor $\PrR \to \Cat$ preserves limits (see \cite[Thm.~5.5.3.18]{LurieHTT}), where $\PrR$ is the \inftyCat of presentable \inftyCats and right adjoint functors, it follows that $f^*$ is also a right adjoint.
\end{proof}

\subsubsection{}

For a smooth morphism of Artin stacks $f : \sX \to \sY$, we write $f_\sharp$ for the left adjoint of $f^*$ in the situation of \propref{prop:Sm1}.

\subsubsection{}

For later use, we describe $f_\sharp$ as the limit of its base changes $f_{T,\sharp}$ over $(T,t)\in\Lis_\sY$, assuming a smooth base change formula that we will prove later.

More generally, let $f : \sX \to \sY$ be a smooth morphism and $g : \sX' \to \sY'$ its base change along a morphism $q : \sY' \to \sY$.
Given $(T,t)\in\Lis_\sY$, form the cube:
\begin{equation}\label{eq:tzwzverx}
  \begin{tikzcd}[matrix xscale=0.5, matrix yscale=0.5]
    &
    \sX'
    \ar{rr}{g}
    \ar[swap,near end]{dd}{p}
    & & \sY'
    \ar{dd}{q}
    \\
    \sX'_T
    \ar{ur}{t_{\sX'}}
    \ar[crossing over, near end]{rr}{g_T}
    \ar[swap]{dd}{p_T}
    & & \sY'_T \ar[swap, near start]{ur}{t_{\sY'}}
    \\
    &
    \sX
    \ar[near start]{rr}{f}
    & & \sY
    \\
    \sX_T \ar[near end]{ur}{t_\sX}
    \ar{rr}{f_T}
    & & T
    \ar[crossing over, leftarrow, near end, swap]{uu}{q_T}
    \ar{ur}{t}
  \end{tikzcd}
\end{equation}
Given a morphism $u : (T_1, t_1) \to (T_2, t_2)$ in $\Lis_\sY$, consider the diagram of cartesian squares
\begin{equation}\label{eq:bbkvuase}
  \begin{tikzcd}
    \sX'_{T_1} \ar{r}{g_{T_1}}\ar{d}{u_{\sX'}}
    & \sY'_{T_1} \ar{d}{u_{\sY'}}\ar{r}{q_{T_1}}
    & T_1 \ar{d}{u}
    \\
    \sX'_{T_2} \ar{r}{g_{T_2}}
    & \sY'_{T_2} \ar{r}{q_{T_2}}
    & T_2.
  \end{tikzcd}
\end{equation}
Suppose that the exchange transformation $\Ex_\sharp^*$ for the left-hand square is invertible, i.e., that $g_{T_2,\sharp}$ commutes with $u_{\sY'}^*$.
In that case, passing to limits vertically in the commutative squares
\[\begin{tikzcd}
  \D(\sX'_{T_2}) \ar{r}{g_{T_2,\sharp}}\ar{d}{u_{\sX'}^*}
  & \D(\sY'_{T_2}) \ar{d}{u_{\sY'}^*}
  \\
  \D(\sX'_{T_1}) \ar{r}{g_{T_1,\sharp}}
  & \D(\sY'_{T_1})
\end{tikzcd}\]
and using the identifications of \corref{cor:sat}, the horizontal functors give rise to a canonical functor
\begin{equation}\label{eq:limgTshp}
  \D(\sX') \simeq \lim_{(T,t)\in\Lis_\sY} \D(\sX'_T)
  \to \lim_{(T,t)\in\Lis_\sY} \D(\sY'_T) \simeq \D(\sY')
\end{equation}
where the limits are taken along $*$-inverse images.

\begin{lem}\label{lem:limgTshp}
  Let $f : \sX \to \sY$ be a smooth morphism and $g : \sX' \to \sY'$ its base change along a morphism $q : \sY' \to \sY$.
  Suppose that for every morphism $u : (T_1, t_1) \to (T_2, t_2)$ in $\Lis_\sY$, $g_{T_2,\sharp}$ commutes with $u_{\sY'}^*$.
  Then the functor $g_\sharp : \D(\sX') \to \D(\sY')$ is identified with \eqref{eq:limgTshp}.
  In other words, $g_\sharp$ is uniquely determined by commutative squares
  \[\begin{tikzcd}
    \D(\sX') \ar{r}{g_{\sharp}}\ar{d}{t_{\sX'}^*}
    & \D(\sY') \ar{d}{t_{\sY'}^*}
    \\
    \D(\sX'_T) \ar{r}{g_{T,\sharp}}
    & \D(\sY'_T)
  \end{tikzcd}\]
  for all $(T,t)\in\Lis_\sY$.
\end{lem}
\begin{proof}
  Denote by $g_?$ the functor \eqref{eq:limgTshp}.
  Let $\sF \in \D(\sX')$ and $\sG \in \D(\sY')$.
  For every $(T,t)\in\Lis_\sY$, we have the adjunction identity
  \[
    \Maps_{\D(\sY'_T)}(g_{T,\sharp} t_{\sX'}^*\sG, t_{\sY'}^*\sG)
    \simeq \Maps_{\D(\sX'_T)}(t_{\sX'}^*\sF, g_T^*t_{\sY'}^*\sG)
  \]
  which may be written equivalently as
  \[
    \Maps_{\D(\sY'_T)}(t_{\sY'}^* g_? \sF, t_{\sY'}^* \sG)
    \simeq \Maps_{\D(\sX'_T)}(t_{\sX'}^* \sF, t_{\sX'}^*g^*\sG)
  \]
  by definition of $g_?$.
  Passing to limits over $(T,t)$ on both sides and using \corref{cor:sat}, we get the identity
  \[
    \Maps_{\D(\sY')}(g_? \sF, \sG)
    \simeq \Maps_{\D(\sX')}(\sF, g^*\sG)
  \]
  functorially in $\sF$ and $\sG$.
  It follows that $g_?$ is left adjoint to $g^*$, i.e., $g_? \simeq g_\sharp$.
\end{proof}

\subsubsection{}
\label{sssec:gofuigqa}

If $f : \sX \to \sY$ is smooth and representable, then applying \lemref{lem:limgTshp} with $q = \id_\sY$ yields that $f_\sharp$ is determined by the identities
\begin{equation}
  t^* f_\sharp \simeq f_{T,\sharp} t_\sX^*
\end{equation}
for all $(T,t)\in\Lis_\sY$.
Indeed, the assumption is satisfied as every term in the left-hand square of \eqref{eq:bbkvuase} is a space.

More generally, if $f : \sX \to \sY$ is smooth and representable and $q : \sY' \to \sY$ is representable, then we find that $g_\sharp$ is determined by the identities
\begin{equation}
  t_{\sY'}^* g_\sharp \simeq g_{T,\sharp} t_{\sX'}^*
\end{equation}
for all $(T,t)\in\Lis_\sY$, where the assumption is satisfied again for the same reason.

\subsection{(Sm3)}
\label{ssec:Sm3}

\subsubsection{}

Like \itemref{item:Sm1}, \itemref{item:Sm3} only depends on the underlying presheaf $\D^{\lis,*}$.
We have:

\begin{prop}\label{prop:Sm3}
  Let $\D^* : (\Art)^\op \to \PrL$ be a presentable presheaf of \inftyCats which satisfies \v{C}ech descent along smooth covering morphisms.
  If its restriction to $\Spc$ satisfies \itemref{item:Sm1} and \itemref{item:Sm3} for every smooth morphism of spaces, then it satisfies \itemref{item:Sm3} for every smooth morphism of Artin stacks.
\end{prop}

We have seen in \propref{prop:Sm1} that $\D^*$ satisfies \itemref{item:Sm1}.
The claim is that for every cartesian square of Artin stacks
\begin{equation}\label{eq:smbcart}
  \begin{tikzcd}
    \sX' \ar{r}{g}\ar{d}{p}
    & \sY' \ar{d}{q}
    \\
    \sX \ar{r}{f}
    & \sY,
  \end{tikzcd}
\end{equation}
where $f$ is smooth, $f_\sharp$ commutes with $q^*$: that is, the exchange transformation $\Ex^*_\sharp : g_\sharp p^* \to q^*f_\sharp$ \eqref{eq:Ex_sharp^*} is invertible.

\subsubsection{}
\label{sssec:smbcredY'}

Given $(T',t')\in\Lis_{\sY'}$, form the diagram of cartesian squares
\begin{equation}\label{eq:smbcredY'}
  \begin{tikzcd}
    \sX'_{T'} \ar{r}{g_{T'}}\ar{d}{t_{\sX'}}
    & T' \ar{d}{t'}
    \\
    \sX' \ar{r}{g}\ar{d}{p}
    & \sY' \ar{d}{q}
    \\
    \sX \ar{r}{f}
    & \sY.
  \end{tikzcd}
\end{equation}

\begin{lem}\label{lem:smbcredY'}
  Suppose that for every $(T',t')\in\Lis_{\sY'}$,
  \begin{inlinelist}
    \item\label{item:smbcredY'/t'}
    $g_\sharp$ commutes with $t'^*$;
    \item\label{item:smbcredY'/qt'}
    $f_{\sharp}$ commutes with $(q \circ t')^*$.
  \end{inlinelist}
  Then $f_\sharp$ commutes with $q^*$.
\end{lem}
\begin{proof}
  For the square \eqref{eq:smbcart}, the exchange transformation $\Ex^*_\sharp : g_\sharp p^* \to q^*f_\sharp$ \eqref{eq:Ex_sharp^*} is invertible if and only if for every $(T',t')\in\Lis_{\sY'}$, $t'^* \circ \eqref{eq:Ex_sharp^*}$ is invertible.
  Under the identifications
  \[
    t'^* g_\sharp p^*
    \simeq g_{T',\sharp} t_{\sX'}^* p^*
    \simeq g_{T',\sharp} (p \circ t_{\sX'})^*,
  \]
  of assumption \itemref{item:smbcredY'/t'}, $t'^* \circ \eqref{eq:Ex_sharp^*}$ is identified with $\Ex_{\sharp}^*$ for the outer composite square in \eqref{eq:smbcredY'}.
  This is invertible by assumption \itemref{item:smbcredY'/qt'}.
\end{proof}

\subsubsection{}

Given $(T,t)\in\Lis_{\sY}$, form the cube of cartesian squares
\begin{equation}\label{eq:smbcredY}
  \begin{tikzcd}[matrix xscale=0.5, matrix yscale=0.5]
    &
    \sX'
    \ar{rr}{g}
    \ar[swap,near end]{dd}{p}
    & & \sY'
    \ar{dd}{q}
    \\
    \sX'_T
    \ar{ur}{t_{\sX'}}
    \ar[crossing over, near end]{rr}{g_T}
    \ar[swap]{dd}{p_T}
    & & \sY'_T \ar[swap, near start]{ur}{t_{\sY'}}
    \\
    &
    \sX
    \ar[near start]{rr}{f}
    & & \sY
    \\
    \sX_T \ar[near end]{ur}{t_\sX}
    \ar{rr}{f_T}
    & & T
    \ar[crossing over, leftarrow, near end, swap]{uu}{q_T}
    \ar{ur}{t}
  \end{tikzcd}
\end{equation}

\begin{lem}\label{lem:smbcredY}
  Suppose that for every $(T,t)\in\Lis_{\sY}$,
  \begin{inlinelist}
    \item\label{item:smbcredY/t}
    $f_\sharp$ commutes with $t^*$;
    \item\label{item:smbcredY/t'}
    $g_\sharp$ commutes with $t_{\sY'}^*$;
    \item\label{item:smbcredY/f_T}
    $f_{T,\sharp}$ commutes with $q_T^*$.
  \end{inlinelist}
  Then $f_\sharp$ commutes with $q^*$.
\end{lem}
\begin{proof}
  By \corref{cor:sat}, it will suffice to show that $t_{\sY'}^* \circ \Ex_\sharp^*$ is invertible for every $(T,t) \in \Lis_\sY$.
  Under the canonical isomorphism
  \[ t_{\sY'}^* g_\sharp \simeq g_{T,\sharp} t_{\sX'}^* \]
  of assumption \itemref{item:smbcredY/t'}, and the isomorphism
  \[
    t_{\sY'}^* q^* f_\sharp
    \simeq q_T^* t^* f_\sharp
    \simeq q_T^* f_{T,\sharp} t_\sX^*
  \]
  of assumption \itemref{item:smbcredY/t}, $t_{\sY'}^* \circ \Ex_\sharp^*$ is identified with the natural transformation
  \[
    \Ex_{\sharp}^* \circ t_\sX^* : g_{T,\sharp} p_T^* t_{\sX}^* \to q_T^* f_{T,\sharp} t_\sX^*,
  \]
  which is invertible by assumption \itemref{item:smbcredY/f_T}.
\end{proof}

\subsubsection{}
\label{sssec:smbcredX}

Given $(S,s)\in\Lis_{\sX}$, form the diagram of cartesian squares
\begin{equation}\label{eq:smbcredX}
  \begin{tikzcd}
    \sX'_S \ar{r}{s'}\ar{d}{p_S}
    & \sX' \ar{r}{g}\ar{d}{p}
    & \sY' \ar{d}{q}
    \\
    S \ar{r}{s}
    & \sX \ar{r}{f}
    & \sY.
  \end{tikzcd}
\end{equation}

\begin{lem}\label{lem:smbcredX}
  Suppose that for every $(S,s)\in\Lis_{\sX}$,
  \begin{inlinelist}
    \item\label{item:smbcredX/s}
    $s_\sharp$ commutes with $p^*$;
    \item\label{item:smbcredX/fs}
    $(f\circ s)_{\sharp}$ commutes with $q^*$.
  \end{inlinelist}
  Then $f_\sharp$ commutes with $q^*$.
\end{lem}
\begin{proof}
  For the square \eqref{eq:smbcart}, the exchange transformation $\Ex^*_\sharp : g_\sharp p^* \to q^*f_\sharp$ \eqref{eq:Ex_sharp^*} is invertible if and only if for every $(S,s)\in\Lis_{\sX}$, $\eqref{eq:Ex_sharp^*} \circ s_\sharp$ is invertible (by \ssecref{ssec:colimshp} and using that each functor involved in $\Ex^*_\sharp$ commutes with colimits).
  Under the identifications
  \[
    g_\sharp p^* s_\sharp
    \simeq g_\sharp s'_\sharp p_S^*
    \simeq (g \circ s')_\sharp p_S^*
  \]
  of assumption \itemref{item:smbcredX/s}, $\eqref{eq:Ex_sharp^*} \circ s_\sharp$ is identified with $\Ex_{\sharp}^*$ for the outer composite square in \eqref{eq:smbcredX}.
  This is invertible by assumption \itemref{item:smbcredX/fs}.
\end{proof}

\subsubsection{Proof of \propref{prop:Sm3}}

Suppose given a cartesian square \eqref{eq:smbcart} where $f$ is smooth.

We first assume that $f$ is representable and argue by induction on the smallest integer $n$ such that  $q$ is $n$-representable.

\noindent\emph{$f$ representable, $q$ representable:}
For the case where $q$ is also representable, we apply \lemref{lem:smbcredY}.
For every $(T,t)\in\Lis_\sY$, assumptions~\itemref{item:smbcredY/t} and \itemref{item:smbcredY/t'} hold since $f_\sharp$ commutes with $t^*$ and $g_\sharp$ commutes with $t_{\sY'}^*$ by \sssecref{sssec:gofuigqa}.
Since $f$ and $q$ are representable, so that $f_T : \sX_T \to T$ and $q_T : \sY'_T \to T$ are maps of spaces, $f_{T,\sharp}$ commutes with $q_T^*$ by \itemref{item:Sm3} for spaces; this shows assumption~\itemref{item:smbcredY/f_T}.
The conclusion is that $f_\sharp$ commutes with $q^*$.

\noindent\emph{$f$ representable, $q$ $n$-representable:}
Assume the claim is known whenever $f$ is representable and $q$ is $(n-1)$-representable.
We first establish the following intermediary case:

$(\ast)$ $\sX=X$ and $\sY=Y$ are spaces, $\sY'$ is $n$-Artin.

To prove $(\ast)$ we apply \lemref{lem:smbcredY'}.
For every $(T',t')\in\Lis_{\sY'}$, it is enough to check that $f_\sharp$ commutes with $(q\circ t')^*$ and $g_\sharp$ commutes with $t'^*$.
Since $f : X \to Y$ is representable, so is $g : \sX' \to \sY'$.
Note that $q \circ t': T' \to Y$ is also representable, and $t' : T \to \sY'$ is $(n-1)$-representable since $\sY'$ is $1$-Artin.
We thus conclude by the induction hypothesis.

We now show that $f_\sharp$ commutes with $q^*$ when $f$ is representable and $q$ is $n$-representable.
Note that for every morphism $u : (T_1,t_1)\to(T_2,t_2)$ in $\Lis_\sY$, the morphisms $f_{T_2} : \sX_{T_2} \to T_2$, $g_{T_2} : \sX'_{T_2} \to \sY'_{T_2}$, $u_{\sY} : \sY_{T_1} \to \sY_{T_2}$, and $u_{\sY'} : \sY'_{T_1} \to \sY'_{T_2}$ are all representable (as base changes of $f$, $g$, $u$ and $u$, respectively).
Hence by the base case of the induction, $f_{T_2,\sharp}$ commutes with $u_{\sY}^*$ and $g_{T_2,\sharp}$ commutes with $u_{\sY'}^*$.
Applying \lemref{lem:limgTshp}, we deduce that for every $(T,t)\in\Lis_\sY$, $f_\sharp$ commutes with $t^*$ and $g_\sharp$ commutes with $t_{\sY'}^*$.
Since $f_T : \sX \to T$ is a morphism of spaces, and $\sY'_T$ is $n$-Artin, we have by $(\ast)$ that $f_{T,\sharp}$ commutes with $q_T^*$.
Thus we conclude by applying \lemref{lem:smbcredY}.

By induction we now have the claim when $f$ is representable and $q$ is arbitrary.
We now proceed by another induction on the smallest integer $m$ such that $f$ is $m$-representable.

\noindent\emph{$f$ $m$-representable, $q$ arbitrary:}
Assume the claim is known whenever $f$ is $(m-1)$-representable (and $q$ is arbitrary).
We begin with the following special cases:

$(\dagger)$ $\sX$ and $\sY$ are $m$-Artin.
Indeed, we apply \lemref{lem:smbcredX}: for $(S,s)\in\Lis_\sX$, $s : S \to \sX$ and $f\circ s : S \to \sY$ are $(m-1)$-representable, so by $s_\sharp$ commutes with $p^*$ and $(f\circ s)_\sharp$ with $q^*$ by the induction hypothesis.

$(\ddagger)$ $f$ $m$-representable, $q$ $m$-representable.
Indeed, we apply \lemref{lem:limgTshp}: for every morphism $u : (T_1,t_1)\to(T_2,t_2)$ in $\Lis_\sY$, 
$f_{T_2} : \sX_{T_2} \to T_2$ and $g_{T_2} : \sX'_{T_2} \to \sY'_{T_2}$ are morphisms of $m$-Artin stacks.
Hence by $(\dagger)$, $f_{T_2,\sharp}$ commutes with $u_{\sY}^*$ and $g_{T_2,\sharp}$ commutes with $u_{\sY'}^*$.
We thus find that $f_\sharp$ commutes with $t^*$ and $g_\sharp$ commutes with $t_{\sY'}^*$ for every $(T,t)\in\Lis_\sY$.
We then apply \lemref{lem:smbcredY}, whose assumptions~\itemref{item:smbcredY/t} and \itemref{item:smbcredY/t'} we have just verified.
The remaining assumption~\itemref{item:smbcredY/f_T}, commutativity of $f_{T,\sharp}$ and $q_T^*$, holds by $(\dagger)$ since $f_T : \sX_T \to T$ is a morphism of $m$-Artin stacks.

$(\S)$ $f_\sharp$ commutes with $t^*$ and $g_\sharp$ commutes with $t_{\sY'}^*$ for every $(T,t)\in\Lis_\sY$.
Indeed, applying \lemref{lem:limgTshp} it suffices to observe that for every morphism $u : (T_1,t_1)\to(T_2,t_2)$ in $\Lis_\sY$, $u : T_1 \to T_2$ and $u_{\sY'} : \sY'_{T_1} \to \sY'_{T_2}$ are representable, so $(\ddagger)$ implies that $f_{T_2,\sharp}$ commutes with $u^*$ and $g_{T_2,\sharp}$ commutes with $u_{\sY'}^*$.

We now apply \lemref{lem:smbcredY} to check that $f_\sharp$ commutes with $q^*$, where $f$ is $m$-representable and $q$ is arbitrary.
Assumptions~\itemref{item:smbcredY/t} and \itemref{item:smbcredY/t'} hold by $(\S)$.
It remains to verify assumption~\itemref{item:smbcredY/f_T}, i.e., that for every $(T,t)\in\Lis_\sY$, $f_{T,\sharp}$ commutes with $q_T^*$.
For this we apply \lemref{lem:smbcredX}: for every $(S,s)\in\Lis_\sX$, consider the cartesian squares
\[\begin{tikzcd}
  \sX'_{S} \ar{r}{s'}\ar{d}{p_S}
  & \sX'_T \ar{r}{g_T}\ar{d}{p_T}
  & \sY'_T \ar{d}{q_T}
  \\
  S \ar{r}{s}
  & \sX_T \ar{r}{f_T}
  & T.
\end{tikzcd}\]
We need to check that $s_\sharp$ commutes with $p^*$ and $(f\circ s)_\sharp$ commutes with $q^*$.
Since $\sX_T$ is $m$-Artin, $s$ is $(m-1)$-representable.
Since $S$ and $T$ are spaces, $f\circ s$ is representable.
Thus we conclude by the inductive hypothesis.

\subsection{(Sm2)}
\label{ssec:Sm2}

\itemref{item:Sm2} only depends on the underlying presheaf $\D^{\lis,*}$ as a lax symmetric monoidal functor.
We have:

\begin{prop}\label{prop:Sm2}
  Let $\D^* : (\Art)^\op \to \PrL$ be a lax symmetric monoidal functor which satisfies \v{C}ech descent along étale covering morphisms.
  If its restriction to $\Spc$ satisfies \itemref{item:Sm1} and \itemref{item:Sm2} for every smooth morphism of spaces, then it satisfies \itemref{item:Sm2} for every smooth morphism of Artin stacks.
\end{prop}

We have seen in \propref{prop:Sm1} that $\D^*$ satisfies \itemref{item:Sm1}.
The claim is that for every smooth morphism of Artin stacks $f : \sX \to \sY$ and $\sG \in \D(\sY)$, the exchange transformation \eqref{eq:projectionshp}
\[
  \Ex^{\otimes,*}_\sharp : f_\sharp(- \otimes f^*(\sG))
  \to f_\sharp(-) \otimes \sG
\]
is invertible.

\subsubsection{}
\label{sssec:PrY}

Let $(T,t)\in\Lis_\sY$.
Since $t^*$ is symmetric monoidal and commutes with $f_\sharp$ (by \propref{prop:Sm3}), this is identified with the exchange transformation
\[
  \Ex^{\otimes,*}_\sharp : f_{T,\sharp}(t_\sX^*(-) \otimes f_T^*t^*\sG)
  \to f_{T,\sharp}(t_\sX^*(-) \otimes t^*(\sG))
\]
associated with $f_{T,\sharp}$.

\subsubsection{}

Let $f$ be $n$-representable for some $n\ge 0$.
For $n=0$, we may use \sssecref{sssec:PrY} to assume that $\sY=Y$ and hence $\sX=X$ are spaces, in which case the claim holds by assumption.

We argue by induction on $n>0$, assuming that the claim holds for all $(n-1)$-representable morphisms.
Again by \sssecref{sssec:PrY} we may assume that $\sY=Y$ is a space and hence $\sX$ is $n$-Artin.
Since $f_\sharp$ commutes with colimits, it will suffice to show that for every $(S,s)\in\Lis_\sX$ the natural transformation
\[
  \Ex^{\otimes,*}_\sharp \circ s_\sharp
  : f_\sharp(s_\sharp(-) \otimes f^*(\sG))
  \to f_\sharp (s_\sharp(-)) \otimes \sG
\]
is invertible.
Since $s$ is $(n-1)$-representable, $s_\sharp$ satisfies the projection formula by the inductive hypothesis.
Thus $\Ex^{\otimes,*}_\sharp \circ s_\sharp$ is identified with
\[
  \Pr^*_\sharp : f_\sharp s_\sharp(- \otimes s^*f^*(\sG))
  \to f_\sharp s_\sharp(-) \otimes \sG,
\]
the exchange transformation associated with $f \circ s : S \to \sX \to Y$.
Since the latter is a morphism of spaces, this is invertible by (Sm2) for spaces.

\subsection{(Sm4)}
\label{ssec:Sm4}

\subsubsection{}
\label{sssec:E_n}

In this subsection we will prove the following two statements.
Let $\cE_0$ and $\cE_\lis$ be as in \sssecref{sssec:E_lis}.
Denote by
$$\cE_n \sub \Art_n$$
the class consisting of morphisms in $\cE_\lis \sub \Art$ whose source and target are $n$-Artin, and by
$$\cE_n^k \sub \Art_n$$
the class of morphisms which are further $k$-representable, where $0\le k\le n$ (note that $\cE_n^n = \cE_n$).

\begin{prop}\label{prop:Sm4extobjs}
  Let $\D^*_!$ be a presentable preweave on $(\Art_n,\cE_n)$ for some $n>0$.
  If its underlying presheaf $\D^*$ satisfies \v{C}ech descent along smooth covering morphisms, then the following conditions are equivalent:
  \begin{thmlist}
    \item 
    It satisfies \itemref{item:Sm4} for all smooth morphisms of $n$-Artin stacks.
    
    \item
    Its restriction to $(\Art_n, \cE_n^{n-1})$ satisfies \itemref{item:Sm4} for all smooth morphisms of $n$-Artin stacks.
  \end{thmlist}
\end{prop}

\begin{prop}\label{prop:Sm4extmorph}
  Let $\D^*_!$ be a presentable preweave on $(\Art_n, \cE_n^{n-1})$ for some $n>0$.
  If its underlying presheaf $\D^*$ satisfies \v{C}ech descent along smooth covering morphisms, then the following conditions are equivalent:
  \begin{thmlist}
    \item\label{item:Sm4extmorph/n}
    It satisfies \itemref{item:Sm4} for all smooth morphisms of $n$-Artin stacks.

    \item\label{item:Sm4extmorph/frep}
    It satisfies \itemref{item:Sm4} for all smooth $(n-1)$-representable morphisms of $n$-Artin stacks.

    \item\label{item:Sm4extmorph/n-1}
    Its restriction to $(\Art_{n-1}, \cE_{n-1}^{n-1}) = (\Art_{n-1}, \cE_{n-1})$ satisfies \itemref{item:Sm4} for all smooth morphisms of $(n-1)$-Artin stacks.
  \end{thmlist}
\end{prop}

\subsubsection{}

We begin with some general reductions.
Let $\cA \sub \Art$ be a full subcategory of Artin stacks which contains $\Spc$ and is closed under fibred products, and let $\cE \sub \cA$ be a class of morphisms which is closed under base change and sections and contains $\cE_\lis^0$.
Let $\D^*_!$ be a presentable preweave on $(\cA,\cE)$ whose underlying presheaf $\D^*$ satisfies \v{C}ech descent along smooth covering morphisms.
Following the convention of \sssecref{sssec:leftpreweave}, we will say that a morphism of Artin stacks is \emph{shriekable} if it belongs to $\cE$.

\subsubsection{}

Given a cartesian square
\begin{equation}\label{eq:sm!art}
  \begin{tikzcd}
    \sX' \ar{r}{g}\ar{d}{p}
    & \sY' \ar{d}{q}
    \\
    \sX \ar{r}{f}
    & \sY
  \end{tikzcd}
\end{equation}
in $\cA$ with $f$ smooth and $q$ shriekable, recall the exchange transformation $\Ex_{\sharp,!} : f_\sharp p_! \to q_! g_\sharp$ \eqref{eq:Ex_shp,!}.
Recall that when this is invertible, we say that \emph{$f_\sharp$ commutes with $q_!$}.

\subsubsection{}

We will require the following reduction lemmas:

\begin{lem}\label{lem:sm!artredS}
  Suppose given a cartesian square of the form \eqref{eq:sm!art}.
  Suppose that for every $(S,s)\in\Lis_\sX$,
  \begin{inlinelist}
    \item\label{item:bfwzlwir}
    $s_\sharp$ commutes with $p_{S,!}$;
    \item\label{item:fjdnfwbh}
    $(f \circ s)_\sharp$ commutes with $p_{S,!}$.
  \end{inlinelist}
  Then $f_\sharp$ commutes with $p_!$.
\end{lem}
\begin{proof}
  Since each term in \eqref{eq:Ex_shp,!} commutes with colimits, it will suffice using \ssecref{ssec:colimshp} to show that $\eqref{eq:Ex_shp,!} \circ s_{\sX',\sharp}$ is invertible for every $(S,s)\in\Lis_\sX$.
  Under assumption~\itemref{item:bfwzlwir}, this is identified with
  \[
    \Ex_{\sharp,!} : (f \circ s)_\sharp p_{S,!} \to q_! (g\circ s_{\sX'})_\sharp,
  \]
  which is invertible by \itemref{item:fjdnfwbh}.
\end{proof}

\begin{lem}\label{lem:sm!artredT}
  Suppose given a cartesian square of the form \eqref{eq:sm!art}.
  Suppose that for every $(T,t)\in\Lis_\sY$, $f_{T,\sharp}$ commutes with $p_{T,!}$.
  Then $f_\sharp$ commutes with $p_!$.
\end{lem}
\begin{proof}
  It will suffice to show that $t^* \eqref{eq:Ex_shp,!}$ is invertible.
  Since $p_!$, $q_!$, $f_\sharp$ and $g_\sharp$ commute with $*$-inverse image (the latter two by \itemref{item:Sm3}), $t^* \eqref{eq:Ex_shp,!}$ is identified with
  \[
    \Ex_{\sharp,!} \circ t_{\sX'}^* : f_{T,\sharp} p_{T,!} t_{\sX'}^*
    \to q_{T,!} g_{T,\sharp} t_{\sX'}^*,
  \]
  which is invertible by the assumption.
\end{proof}

\begin{lem}\label{lem:sm!artredT'}
  Suppose given a cartesian square of the form \eqref{eq:sm!art}.
  Suppose that for every $(T',t')\in\Lis_{\sY'}$,
  \begin{inlinelist}
    \item\label{item:esjrlgtd}
    $g_\sharp$ commutes with $t'_{\sX',!}$;
    \item\label{item:fmygscsc}
    $f_\sharp$ commutes with $(p \circ t'_{\sX'})_!$.
  \end{inlinelist}
  Then $f_\sharp$ commutes with $p_!$.
\end{lem}
\begin{proof}
  By \ssecref{ssec:colimshp} and since each functor in \eqref{eq:Ex_shp,!} preserves colimits, it will suffice to show that $\eqref{eq:Ex_shp,!}\circ t'_{\sX',!}$ is invertible.
  By assumption~\itemref{item:esjrlgtd}, this is identified with $\Ex_{\sharp,!}: f_\sharp (p\circ t'_{\sX'})_! \to (q \circ t)'_! g_{T',\sharp}$, which is invertible by assumption~\itemref{item:fmygscsc}.
\end{proof}

\subsubsection{Proof of \propref{prop:Sm4extobjs}}

Suppose given a cartesian square as in \eqref{eq:sm!art}, where $f : \sX \to \sY$ is a smooth morphism of $n$-Artin stacks and $q : \sY' \to \sY$ lies in $\cE_n$.
Given $(T,t)\in\Lis_\sY$ and $(T',t')\in\Lis_{\sY'_T}$, form the diagram of cartesian squares
\begin{equation*}
  \begin{tikzcd}
    \sX'_{T'} \ar{r}{g_{T'}}\ar{d}{t'_{\sX'}}
    & T' \ar{d}{t'}
    \\
    \sX'_T \ar{r}{g_T}\ar{d}{p_T}
    & \sY'_T \ar{d}{q_T}
    \\
    \sX_T \ar{r}{f_T}
    & T.
  \end{tikzcd}
\end{equation*}
For every $(T',t')$, $g_{T,\sharp}$ commutes with $t'_!$ since $t'$ lies in $\cE_n^{n-1}$.
Similarly, $f_{T,\sharp}$ commutes with $(q_T\circ t')_!$ since $q_T \circ t'$ lies in $\cE_n^{0}$.
Applying \lemref{lem:sm!artredT'} as $(T',t')$ varies, we deduce that $f_{T,\sharp}$ commutes with $q_{T,!}$.
Applying \lemref{lem:sm!artredT} as $(T,t)$ varies, we then deduce that $f_\sharp$ commutes with $q_!$.

\subsubsection{Proof of \propref{prop:Sm4extmorph}}

The implications \itemref{item:Sm4extmorph/n} $\Rightarrow$ \itemref{item:Sm4extmorph/frep} $\Rightarrow$ \itemref{item:Sm4extmorph/n-1} are clear.

\noindent\itemref{item:Sm4extmorph/n-1} $\Rightarrow$ \itemref{item:Sm4extmorph/frep}:
Suppose given a cartesian square of $n$-Artin stacks \eqref{eq:sm!art} where $f : \sX \to \sY$ is smooth $(n-1)$-representable and $q : \sY' \to \sY$ lies in $\cE_n^{n-1}$.
Given $(T,t)\in\Lis_\sY$ and $(S,s)\in\Lis_{\sX_T}$, form the diagram of cartesian squares
\begin{equation*}
  \begin{tikzcd}
    \sX'_{S'} \ar{r}{s_{\sX'}}\ar{d}{p_{S}}
    & \sX'_T \ar{r}{g_T}\ar{d}{p_T}
    & \sY'_T \ar{d}{q_T}
    \\
    S \ar{r}{s}
    & \sX_T \ar{r}{f_T}
    & T.
  \end{tikzcd}
\end{equation*}
Since $f$ and $q$ are $(n-1)$-representable, the right-hand square consists of $(n-1)$-Artin stacks.
Hence for every $(S,s)$, $s$ is a smooth morphism of $(n-1)$-Artin stacks and $p_T$ lies in $\cE_{n-1}$.
By \itemref{item:Sm4extmorph/n-1}, $s_\sharp$ commutes with $p_{T,!}$.
Similarly, $f_T \circ s$ is a smooth morphism of spaces and $q_T$ lies in $\cE_{n-1}$, so $(f_T\circ s)_\sharp$ also commutes with $q_{T,!}$ by \itemref{item:Sm4extmorph/n-1}.
Applying \lemref{lem:sm!artredS} as $(S,s)$ varies, we deduce that $f_{T,\sharp}$ commutes with $q_{T,!}$ for every $(T,t)$.
Applying \lemref{lem:sm!artredT} as $(T,t)$ varies, we then deduce that $f_\sharp$ commutes with $q_!$.

\noindent\itemref{item:Sm4extmorph/frep} $\Rightarrow$ \itemref{item:Sm4extmorph/n}:
Suppose given a cartesian square of $n$-Artin stacks \eqref{eq:sm!art} where $f : \sX \to \sY$ is a smooth morphism of $n$-Artin stacks and $q : \sY' \to \sY$ lies in $\cE_n^{n-1}$.
Given $(T,t)\in\Lis_\sY$ and $(S,s)\in\Lis_{\sX_T}$, form the diagram of cartesian squares
\begin{equation*}
  \begin{tikzcd}
    S' \ar{r}{s'}\ar{d}{p_{S}}
    & \sX'_T \ar{r}{g_T}\ar{d}{p_T}
    & \sY'_T \ar{d}{q_T}
    \\
    S \ar{r}{s}
    & \sX_T \ar{r}{f_T}
    & T.
  \end{tikzcd}
\end{equation*}
For every $(S,s)$, $s$ is smooth $(n-1)$-representable, and $q_T, p_T$ lie in $\cE_n^{n-1}$.
Hence by \itemref{item:Sm4extmorph/frep}, $s_\sharp$ commutes with $p_{T,!}$ and $(f_T\circ s)_\sharp$ commutes with $q_{T,!}$.
Applying \lemref{lem:sm!artredS} as $(S,s)$ varies, we deduce that $f_{T,\sharp}$ commutes with $q_{T,!}$ for every $(T,t)$.
Applying \lemref{lem:sm!artredT} as $(T,t)$ varies, we then deduce that $f_\sharp$ commutes with $q_!$.

\subsection{(Sm5)}
\label{ssec:Sm5}

\subsubsection{}

Let $\cE_n$ and $\cE_n^k$ be as in \sssecref{sssec:E_n}.
In this subsection we will prove:

\begin{prop}\label{prop:Sm5extmorph}
  Let $\D^*_!$ be a presentable preweave on $(\Art_n, \cE_n^{n-1})$ for some $n>0$.
  If it satisfies \itemref{item:Sm1}, \itemref{item:Sm2}, \itemref{item:Sm3} and \itemref{item:Sm4} for all smooth morphisms in $\Art_n$, and its underlying presheaf $\D^*$ satisfies \v{C}ech descent along smooth covering morphisms, then the following conditions are equivalent:
  \begin{thmlist}
    \item\label{item:Sm5extmorph/n}
    It satisfies \itemref{item:Sm5} for all smooth morphisms of $n$-Artin stacks.

    \item\label{item:Sm5extmorph/frep}
    It satisfies \itemref{item:Sm5} for all smooth $(n-1)$-representable morphisms of $n$-Artin stacks.

    \item\label{item:Sm5extmorph/n-1}
    Its restriction to $(\Art_{n-1}, \cE_{n-1}^{n-1}) = (\Art_{n-1}, \cE_{n-1})$ satisfies \itemref{item:Sm5} for all smooth morphisms of $(n-1)$-Artin stacks.
  \end{thmlist}
\end{prop}

\subsubsection{}

We begin with several preparatory lemmas.
Let $\cA \sub \Art$ be a full subcategory of Artin stacks which contains $\Spc$ and is closed under fibred products, and let $\cE \sub \cA$ be a class of morphisms which is closed under base change and sections and contains $\cE_\lis^0$.
Let $\D^*_!$ be a presentable preweave on $(\cA,\cE)$ satisfying 
\itemref{item:Sm1}, \itemref{item:Sm2}, \itemref{item:Sm3} and \itemref{item:Sm4} for all smooth morphisms in $\cA$, and whose underlying presheaf $\D^*$ satisfies \v{C}ech descent along smooth covering morphisms.
As in \sssecref{sssec:leftpreweave}, we will say that a morphism of Artin stacks is \emph{shriekable} if it belongs to $\cE$.

\begin{lem}\label{lem:tetuqsba}
  Let $\sX \in \cA$ be an Artin stack and $\sF \in \D(\sX)$ an object.
  Assume that for every $(S,s) \in \Lis_\sX$, $\D^*$ satisfies \itemref{item:Sm2} for $s : S \to \sX$.
  Then $\sF$ is $\otimes$-invertible if and only if $s^*(\sF) \in \D(S)$ is invertible for every $(S,s)\in\Lis_\sX$.
\end{lem}
\begin{proof}
  Since $\D(\sX)$ is closed symmetric monoidal, $\sF$ is $\otimes$-invertible if and only if the evaluation morphism
  \begin{equation}\label{eq:otpabymk}
    \sF \otimes \uHom_\sX(\sF, \un_\sX) \to \un_\sX
  \end{equation}
  is invertible.
  This holds if and only if $s^* \eqref{eq:otpabymk}$ is invertible for all $(S,s) \in \Lis_\sX$.
  Since each $s$ is smooth and thus satisfies the projection formula by \itemref{item:Sm2}, $s^*$ commutes with $\uHom$ by \eqref{eq:f^*Hom}.
  The claim follows.
\end{proof}

\begin{lem}\label{lem:Sm5redT}
  Let $f : \sX \to \sY$ be a smooth morphism of Artin stacks in $\cA$.
  Assume that $\D^*_!$ satisfies \itemref{item:Sm2} and \itemref{item:Sm3} for $f$, as well as \itemref{item:Sm5} for every base change $f_T : \sX_T \to T$ along $(T,t)\in\Lis_\sY$.
  Then it satisfies \itemref{item:Sm5} for $f$.
\end{lem}
\begin{proof}
  Let $i : \sY \to \sX$ be a shriekable section of $f$.
  Form the diagram of cartesian squares
  \begin{equation*}
    \begin{tikzcd}
      T \ar{r}{i_T}\ar{d}{t}
      & \sX_T \ar{r}{f_T}\ar{d}{t_\sX}
      & T \ar{d}{t}
      \\
      \sY \ar{r}{i}
      & \sX \ar{r}{f}
      & \sY.
    \end{tikzcd}
  \end{equation*}
  Since $f_\sharp$ and $i_!$ satisfy the projection formula (the former by \itemref{item:Sm2}), it will suffice to show that $f_\sharp i_!(\un_\sX) \in \D(\sX)$ is $\otimes$-invertible.
  By \lemref{lem:tetuqsba} and the commutativity of $f_\sharp$ and $i_!$ with $*$-inverse image (the former by \itemref{item:Sm3}), this is equivalent to the $\otimes$-invertibility of $f_{T,\sharp} i_{T,!}(\un_T) \in \D(T)$, which follows from the fact that $f_T$ satisfies \itemref{item:Sm5}.
\end{proof}

\begin{lem}\label{lem:uetdxbiz}
  Let $f : \sX \to \sY$ be a smooth morphism of Artin stacks in $\cA$.
  Assume that for every $(S,s) \in \Lis_\sX$, $\D^*_!$ satisfies \itemref{item:Sm2} and \itemref{item:Sm4} for $f\circ s : S \to \sY$, and that $f\circ s$ shriekable.
  If $\Sigma_{f \circ s}(\un_S)$ is $\otimes$-invertible, then $f_\sharp i_!$ is an equivalence for every shriekable section $i$ of $f$.
\end{lem}
\begin{proof}
  Given $(S,s) \in \Lis_\sX$, form the commutative diagram
  \[\begin{tikzcd}
    \sY_S \ar{r}{i_S}\ar{d}{s_\sY}
    & S \ar{d}{s}\ar{rd}{f\circ s}
    &
    \\
    \sY \ar{r}{i}
    & \sX \ar{r}{f}
    & \sY.
  \end{tikzcd}\]
  Since $f\circ s$ is shriekable and satisfies \itemref{item:Sm2} and \itemref{item:Sm4}, we have
  \[
    (f\circ s)_\sharp
    \simeq (f\circ s)_! \Sigma_{f\circ s}
    \simeq (f\circ s)_! (- \otimes \Sigma_{f\circ s}(\un))
  \]
  by \propref{prop:fshptwist}.
  Using \ssecref{ssec:colimshp} we thus compute
  \begin{align*}
    f_\sharp i_!(\un)
    &\simeq \colim_{(S,s)} (f\circ s)_\sharp i_{S,!}(\un)\\
    &\simeq \colim_{(S,s)} (f\circ s)_! (i_{S,!}(\un) \otimes \Sigma_{f\circ s}(\un))\\
    &\simeq \colim_{(S,s)} (f\circ s)_! i_{S,!}i_S^*\Sigma_{f\circ s}(\un)\\
    &\simeq \colim_{(S,s)} s_{\sY,!} i_S^*\Sigma_{f\circ s}(\un)\\
    &\simeq i^* \colim_{(S,s)} s_! \Sigma_{f\circ s}(\un),
  \end{align*}
  using the projection formula for $(f\circ s)_!$, the identity $f\circ i \simeq \id$, the base change formula for $s_!$ and the commutativity of $i^*$ with colimits.
  In the notation of \ssecref{ssec:niavzhjn}, with $F$ the diagram $\Lis_\sX \to \cS_{/\sX}$, we may further write this as $i^* F_! (\Sigma_{f\circ s}(\un))$.
  Since $\Sigma_{f\circ s}(\un)$ is $\otimes$-invertible by \itemref{item:Sm5}, $F_!$ is an equivalence, and $i^*$ is symmetric monoidal, this object is $\otimes$-invertible.
\end{proof}

\subsubsection{Proof of \propref{prop:Sm5extmorph}}

The implications \itemref{item:Sm5extmorph/n} $\Rightarrow$ \itemref{item:Sm5extmorph/frep} $\Rightarrow$ \itemref{item:Sm5extmorph/n-1} are clear.

\noindent\itemref{item:Sm5extmorph/n-1} $\Rightarrow$ \itemref{item:Sm5extmorph/frep}:
Let $f : \sX \to \sY$ be a smooth $(n-1)$-representable morphism of $n$-Artin stacks and $i : \sY \to \sX$ a section.
By \lemref{lem:Sm5redT} it will suffice to show that for every $(T,t)\in\Lis_\sY$, the base change $f_T : \sX_T \to T$ satisfies \itemref{item:Sm5}.
Since $\sX_T$ is $(n-1)$-Artin, this holds by assumption~\itemref{item:Sm5extmorph/n-1}.

\noindent\itemref{item:Sm5extmorph/frep} $\Rightarrow$ \itemref{item:Sm5extmorph/n}:
Let $f : \sX \to \sY$ be a smooth morphism of $n$-Artin stacks and $i : \sY \to \sX$ a section.
Given $(S,s) \in \Lis_\sX$, note that $f \circ s : S \to \sY$ is smooth and $(n-1)$-representable.
Hence it is shriekable and satisfies \itemref{item:Sm2} and \itemref{item:Sm4} (by assumption).
Moreover, assumption~\ref{item:Sm5extmorph/frep} implies that $\Sigma_{f \circ s}(\un_S)$ is $\otimes$-invertible.
Thus we may apply \lemref{lem:uetdxbiz} to deduce that $f_\sharp i_!$ is an equivalence.

\subsection{Proof of \thmref{thm:presaxioms}\itemref{item:presaxioms/sm}}

Let $\D^*_!$ be a presentable preweave on $(\Art,\cE_\lis)$ whose underlying presheaf $\D^{\lis,*}$ satisfies \v{C}ech descent along smooth covering morphisms.
We assume that the restriction of $\D^*_!$ to $(\Spc,\cE_0)$ admits $\sharp$-direct image with respect to smooth morphisms of spaces.

By Propositions~\ref{prop:Sm1}, \ref{prop:Sm2}, and \ref{prop:Sm3}, it satisfies \itemref{item:Sm1}, \itemref{item:Sm2} and \itemref{item:Sm3} for all smooth morphisms of Artin stacks.
Applying Propositions~\ref{prop:Sm4extobjs} and \ref{prop:Sm4extmorph} inductively, we see that it satisfies \itemref{item:Sm4} for all smooth morphisms of Artin stacks.
Applying \propref{prop:Sm5extmorph} inductively, we see that it satisfies \itemref{item:Sm5} for all smooth morphisms of Artin stacks.
In summary, $\D^*_!$ admits $\sharp$-direct image for all smooth morphisms of Artin stacks.

\section{Proper axioms}

In this section we prove \thmref{thm:presaxioms}\itemref{item:presaxioms/pr}.

\subsection{Preliminaries}

Let $\D^* : (\Art)^\op \to \PrL$ be a presentable presheaf of \inftyCats which satisfies \v{C}ech descent along smooth covering morphisms.
We assume that its restriction to $\Spc$ admits $\sharp$-direct image for every smooth morphism of spaces and $*$-direct image for every proper morphism of spaces.

\subsubsection{}

Let $f : \sX \to \sY$ be a morphism of Artin stacks.
We begin by describing $f_*$ as the limit of its base changes $f_{T,*}$ over $(T,t)\in\Lis_\sY$, assuming a base change formula.

More generally, let $f : \sX \to \sY$ be a morphism and $g : \sX' \to \sY'$ its base change along a morphism $q : \sY' \to \sY$.
Given $(T,t)\in\Lis_\sY$, form the cube:
\begin{equation}\label{eq:tzwzverxpr}
  \begin{tikzcd}[matrix xscale=0.5, matrix yscale=0.5]
    &
    \sX'
    \ar{rr}{g}
    \ar[swap,near end]{dd}{p}
    & & \sY'
    \ar{dd}{q}
    \\
    \sX'_T
    \ar{ur}{t_{\sX'}}
    \ar[crossing over, near end]{rr}{g_T}
    \ar[swap]{dd}{p_T}
    & & \sY'_T \ar[swap, near start]{ur}{t_{\sY'}}
    \\
    &
    \sX
    \ar[near start]{rr}{f}
    & & \sY
    \\
    \sX_T \ar[near end]{ur}{t_\sX}
    \ar{rr}{f_T}
    & & T
    \ar[crossing over, leftarrow, near end, swap]{uu}{q_T}
    \ar{ur}{t}
  \end{tikzcd}
\end{equation}
Given a morphism $u : (T_1, t_1) \to (T_2, t_2)$ in $\Lis_\sY$, consider the diagram of cartesian squares
\begin{equation}\label{eq:bbkvuasepr}
  \begin{tikzcd}
    \sX'_{T_1} \ar{r}{g_{T_1}}\ar{d}{u_{\sX'}}
    & \sY'_{T_1} \ar{d}{u_{\sY'}}\ar{r}{q_{T_1}}
    & T_1 \ar{d}{u}
    \\
    \sX'_{T_2} \ar{r}{g_{T_2}}
    & \sY'_{T_2} \ar{r}{q_{T_2}}
    & T_2.
  \end{tikzcd}
\end{equation}
Suppose that the exchange transformation $\Ex_*^*$ \eqref{eq:Ex_*^*} for the left-hand square is invertible, i.e., that $g_{T_2,*}$ commutes with $u_{\sY'}^*$.
In that case, passing to limits vertically in the commutative squares
\[\begin{tikzcd}
  \D(\sX'_{T_2}) \ar{r}{g_{T_2,*}}\ar{d}{u_{\sX'}^*}
  & \D(\sY'_{T_2}) \ar{d}{u_{\sY'}^*}
  \\
  \D(\sX'_{T_1}) \ar{r}{g_{T_1,*}}
  & \D(\sY'_{T_1})
\end{tikzcd}\]
and using the identifications of \corref{cor:sat}, the horizontal functors give rise to a canonical functor
\begin{equation}\label{eq:limgT*}
  \D(\sX') \simeq \lim_{(T,t)\in\Lis_\sY} \D(\sX'_T)
  \to \lim_{(T,t)\in\Lis_\sY} \D(\sY'_T) \simeq \D(\sY')
\end{equation}
where the limits are taken along $*$-inverse images.

\begin{lem}\label{lem:limgT*}
  Let $f : \sX \to \sY$ be a morphism and $g : \sX' \to \sY'$ its base change along a morphism $q : \sY' \to \sY$.
  Suppose that for every morphism $u : (T_1, t_1) \to (T_2, t_2)$ in $\Lis_\sY$, $g_{T_2,*}$ commutes with $u_{\sY'}^*$.
  Then the functor $g_* : \D(\sX') \to \D(\sY')$ is identified with \eqref{eq:limgT*}.
  In other words, $g_*$ is uniquely determined by commutative squares
  \[\begin{tikzcd}
    \D(\sX') \ar{r}{g_*}\ar{d}{t_{\sX'}^*}
    & \D(\sY') \ar{d}{t_{\sY'}^*}
    \\
    \D(\sX'_T) \ar{r}{g_{T,*}}
    & \D(\sY'_T)
  \end{tikzcd}\]
  for all $(T,t)\in\Lis_\sY$.
\end{lem}
\begin{proof}
  Denote by $g_?$ the functor \eqref{eq:limgT*}.
  Let $\sF \in \D(\sX')$ and $\sG \in \D(\sY')$.
  For every $(T,t)\in\Lis_\sY$, we have the adjunction identity
  \[
    \Maps_{\D(\sX'_T)}(g_T^* t_{\sY'}^*\sG, t_{\sX'}^*\sF)
    \simeq \Maps_{\D(\sY'_T)}(t_{\sY'}^*\sG, g_{T,*} t_{\sX'}^*\sF)
  \]
  which may be written equivalently as
  \[
    \Maps_{\D(\sX'_T)}(t_{\sX'}^*g^* \sG, t_{\sX'}^*\sF)
    \simeq \Maps_{\D(\sY'_T)}(t_{\sY'}^*\sG, t_{\sY'}^* g_? \sF)
  \]
  by definition of $g_?$.
  Passing to limits over $(T,t)$ on both sides and using \corref{cor:sat}, we get the identity
  \[
    \Maps_{\D(\sX')}(g^*\sG, \sF)
    \simeq \Maps_{\D(\sY')}(\sG, g_? \sF)
  \]
  functorially in $\sF$ and $\sG$.
  It follows that $g_?$ is right adjoint to $g_*$, i.e., $g_? \simeq g_*$.
\end{proof}

\subsubsection{}
\label{sssec:xpcdvcno}

If $f : \sX \to \sY$ is proper and representable, then applying \lemref{lem:limgT*} with $q = \id_\sY$ yields that $f_*$ is determined by the identities
\begin{equation}
  t^* f_* \simeq f_{T,*} t_\sX^*
\end{equation}
for all $(T,t)\in\Lis_\sY$.
Indeed, the assumption is satisfied as every term in the left-hand square of \eqref{eq:bbkvuasepr} is a space, hence has invertible $\Ex^*_*$ by the assumption of \itemref{item:Pr2} for spaces.

More generally, if $f : \sX \to \sY$ is proper representable and $q : \sY' \to \sY$ is representable, then we find that $g_*$ is determined by the identities
\begin{equation}
  t_{\sY'}^* g_* \simeq g_{T,*} t_{\sX'}^*
\end{equation}
for all $(T,t)\in\Lis_\sY$, where the assumption is satisfied again for the same reason.

\subsection{\itemref{item:Pr1}}
\label{ssec:Pr1}

Let $\D^* : (\Art)^\op \to \PrL$ be a presentable presheaf of \inftyCats which satisfies \v{C}ech descent along smooth covering morphisms.
Assume that its restriction to $\Spc$ admits $\sharp$-direct image along smooth morphisms of spaces and $*$-direct image along proper morphisms of spaces.

Let $f : \sX \to \sY$ be a proper representable morphism of Artin stacks.
Given a diagram $(\sF_i)_i$ in $\D(\sX)$, the claim is that the canonical morphism in $\D(\sY)$
\begin{equation}\label{eq:Pr1}
  \colim_i f_*(\sF_i)
  \to f_*(\colim_i \sF_i)
\end{equation}
is invertible.
It will suffice to check that it becomes invertible after applying $t^*$ for any $(T,t)\in\Lis_\sY$.
By \sssecref{sssec:xpcdvcno} and the fact that $t^*$ commutes with colimits, $t^* \eqref{eq:Pr1}$ is identified with 
\[
  \colim_i f_{T,*}(t_\sX^*\sF_i)
  \to f_{T,*}(t_\sX^*\colim_i \sF_i)
\]
which is invertible by \itemref{item:Pr1} for the morphism of spaces $f_T : \sX_T \to T$.

\subsection{\itemref{item:Pr2}}
\label{ssec:Pr2}

Let $\D^* : (\Art)^\op \to \PrL$ be a presentable presheaf of \inftyCats which satisfies \v{C}ech descent along smooth covering morphisms.
Assume that its restriction to $\Spc$ admits $\sharp$-direct image along smooth morphisms of spaces and $*$-direct image along proper morphisms of spaces.

Let $f : \sX \to \sY$ be a proper representable morphism of Artin stacks.
We claim that the exchange transformation $\Ex^{\otimes,*}_* : f_*(-) \otimes (-) \to f_*(- \otimes f^*(-))$ \eqref{eq:projection*} is invertible.
It will suffice to show that for every $(T,t)\in\Lis_\sY$, the induced natural transformation $t^* \Ex^{\otimes,*}_*$ is invertible.
Using \sssecref{sssec:xpcdvcno} and the symmetric monoidality of $t^*$ and $t_\sX^*$, this is identified with
\[
  \Ex^{\otimes,*}_* : f_{T,*}(t_\sX^*(-)) \otimes t^*(-)
  \to f_{T,*}(t_\sX^*(-) \otimes f_T^*t^*(-)).
\]
Since $f_T : \sX_T \to T$ is a proper morphism of spaces, this is invertible by \itemref{item:Pr2}.

\subsection{\itemref{item:Pr3}}
\label{ssec:Pr3}

Let $\D^* : (\Art)^\op \to \PrL$ be a presentable presheaf of \inftyCats which satisfies \v{C}ech descent along smooth covering morphisms.
Assume that its restriction to $\Spc$ admits $\sharp$-direct image along smooth morphisms of spaces and $*$-direct image along proper morphisms of spaces.

Suppose given a cartesian square of Artin stacks
\begin{equation}\label{eq:prbcart}
  \begin{tikzcd}
    \sX' \ar{r}{g}\ar{d}{p}
    & \sY' \ar{d}{q}
    \\
    \sX \ar{r}{f}
    & \sY
  \end{tikzcd}
\end{equation}
where $f$ is proper representable.
We claim that $f_*$ commutes with $q^*$: that is, the exchange transformation $\Ex^*_* : q^* f_* \to g_* p^*$ \eqref{eq:Ex_*^*} is invertible.

\subsubsection{}
\label{sssec:prbcredY'}

Given $(T',t')\in\Lis_{\sY'}$, form the diagram of cartesian squares
\begin{equation}\label{eq:prbcredY'}
  \begin{tikzcd}
    \sX'_{T'} \ar{r}{g_{T'}}\ar{d}{t_{\sX'}}
    & T' \ar{d}{t'}
    \\
    \sX' \ar{r}{g}\ar{d}{p}
    & \sY' \ar{d}{q}
    \\
    \sX \ar{r}{f}
    & \sY.
  \end{tikzcd}
\end{equation}

\begin{lem}\label{lem:prbcredY'}
  Suppose that for every $(T',t')\in\Lis_{\sY'}$,
  \begin{inlinelist}
    \item\label{item:prbcredY'/t'}
    $g_*$ commutes with $t'^*$;
    \item\label{item:prbcredY'/qt'}
    $f_*$ commutes with $(q \circ t')^*$.
  \end{inlinelist}
  Then $f_*$ commutes with $q^*$.
\end{lem}
\begin{proof}
  For the square \eqref{eq:prbcart}, the exchange transformation $\Ex^*_* : g_*p^* \to q^*f_*$ \eqref{eq:Ex_*^*} is invertible if and only if for every $(T',t')\in\Lis_{\sY'}$, $t'^* \circ \eqref{eq:Ex_sharp^*}$ is invertible.
  Under the identifications
  \[
    t'^* g_* p^*
    \simeq g_{T',*} t_{\sX'}^* p^*
    \simeq g_{T',*} (p \circ t_{\sX'})^*,
  \]
  of assumption \itemref{item:prbcredY'/t'}, $t'^* \circ \eqref{eq:Ex_sharp^*}$ is identified with $\Ex_{*}^*$ for the outer composite square in \eqref{eq:prbcredY'}.
  This is invertible by assumption \itemref{item:prbcredY'/qt'}.
\end{proof}

\subsubsection{}

Given $(T,t)\in\Lis_{\sY}$, form the cube of cartesian squares
\begin{equation}\label{eq:prbcredY}
  \begin{tikzcd}[matrix xscale=0.5, matrix yscale=0.5]
    &
    \sX'
    \ar{rr}{g}
    \ar[swap,near end]{dd}{p}
    & & \sY'
    \ar{dd}{q}
    \\
    \sX'_T
    \ar{ur}{t_{\sX'}}
    \ar[crossing over, near end]{rr}{g_T}
    \ar[swap]{dd}{p_T}
    & & \sY'_T \ar[swap, near start]{ur}{t_{\sY'}}
    \\
    &
    \sX
    \ar[near start]{rr}{f}
    & & \sY
    \\
    \sX_T \ar[near end]{ur}{t_\sX}
    \ar{rr}{f_T}
    & & T
    \ar[crossing over, leftarrow, near end, swap]{uu}{q_T}
    \ar{ur}{t}
  \end{tikzcd}
\end{equation}

\begin{lem}\label{lem:prbcredY}
  Suppose that for every $(T,t)\in\Lis_{\sY}$,
  \begin{inlinelist}
    \item\label{item:prbcredY/t}
    $f_*$ commutes with $t^*$;
    \item\label{item:prbcredY/t'}
    $g_*$ commutes with $t_{\sY'}^*$;
    \item\label{item:prbcredY/f_T}
    $f_{T,*}$ commutes with $q_T^*$.
  \end{inlinelist}
  Then $f_*$ commutes with $q^*$.
\end{lem}
\begin{proof}
  By \corref{cor:sat}, it will suffice to show that $t_{\sY'}^* \circ \Ex_*^*$ is invertible for every $(T,t) \in \Lis_\sY$.
  Under the canonical isomorphism
  \[ t_{\sY'}^* g_* \simeq g_{T,*} t_{\sX'}^* \]
  of assumption \itemref{item:prbcredY/t'}, and the isomorphism
  \[
    t_{\sY'}^* q^* f_*
    \simeq q_T^* t^* f_*
    \simeq q_T^* f_{T,*} t_\sX^*
  \]
  of assumption \itemref{item:prbcredY/t}, $t_{\sY'}^* \circ \Ex_*^*$ is identified with the natural transformation
  \[
    \Ex_{*}^* \circ t_\sX^* : g_{T,*} p_T^* t_{\sX}^* \to q_T^* f_{T,*} t_\sX^*,
  \]
  which is invertible by assumption \itemref{item:prbcredY/f_T}.
\end{proof}

\subsubsection{Proof of \itemref{item:Pr3}}

Suppose given a cartesian square \eqref{eq:prbcart} where $f$ is proper and representable.

\noindent\emph{Case 1:}
$q$ is representable.
We may apply \lemref{lem:prbcredY}, where assumptions~\itemref{item:prbcredY/t} and \itemref{item:prbcredY/t'} hold by \sssecref{sssec:xpcdvcno} and \itemref{item:prbcredY/f_T} holds by \itemref{item:Pr3} for spaces.

\noindent\emph{Case 2:}
$\sX=X$ and $\sY=Y$ are spaces, $q$ is arbitrary.
We apply \lemref{lem:prbcredY'}, where the assumptions are verified by Case~1 since $t'$ and $q\circ t'$ are representable.

\noindent\emph{General:}
We apply \lemref{lem:prbcredY}.
For assumptions~\itemref{item:prbcredY/t} and \itemref{item:prbcredY/t'}, use \lemref{lem:limgT*}, whose assumption is verified by Case~1.
For assumption~\itemref{item:prbcredY/f_T}, use Case~2 above (since $f_T$ is a morphism of spaces).

\subsection{\itemref{item:Pr4}}
\label{ssec:Pr4}

Let $\D^*_!$ be a presentable preweave on $(\Art,\cE_\lis)$ whose underlying presheaf $\D^*$ satisfies \v{C}ech descent along smooth covering morphisms.
We assume that the restriction of $\D^*_!$ to $(\Spc,\cE_0)$ admits $\sharp$-direct image for smooth morphisms of spaces, and $*$-direct image for proper morphisms of spaces.

Given a cartesian square
\begin{equation}\label{eq:pr!art}
  \begin{tikzcd}
    \sX' \ar{r}{g}\ar{d}{p}
    & \sY' \ar{d}{q}
    \\
    \sX \ar{r}{f}
    & \sY
  \end{tikzcd}
\end{equation}
with $q$ shriekable, we will show that $f_*$ commutes with $p_!$, i.e., that the exchange transformation $\Ex_{!,*} : q_! g_* \to f_* p_!$ \eqref{eq:Ex_!*} is invertible.

\subsubsection{}

We will require the following reduction lemmas:

\begin{lem}\label{lem:pr!artredT}
  Suppose given a cartesian square of the form \eqref{eq:pr!art}.
  Suppose that for every $(T,t)\in\Lis_\sY$, $f_{T,*}$ commutes with $p_{T,!}$.
  Then $f_*$ commutes with $p_!$.
\end{lem}
\begin{proof}
  It will suffice to show that $t^* \Ex_{!,*}$ is invertible.
  Since $p_!$, $q_!$, $f_*$ and $g_*$ commute with $*$-inverse image (the latter two by \itemref{item:Pr3}), this is identified with
  \[
    \Ex_{!,*} \circ t_{\sX'}^* :
    q_{T,!} g_{T,*} t_{\sX'}^*
    \to f_{T,*} p_{T,!} t_{\sX'}^*,
  \]
  which is invertible by the assumption.
\end{proof}

\begin{lem}\label{lem:pr!artredT'}
  Suppose given a cartesian square of the form \eqref{eq:pr!art}.
  Suppose that for every $(T',t')\in\Lis_{\sY'}$,
  \begin{inlinelist}
    \item\label{item:esjrlgtd2}
    $g_*$ commutes with $t'_{\sX',!}$;
    \item\label{item:fmygscsc2}
    $f_*$ commutes with $(p \circ t'_{\sX'})_!$.
  \end{inlinelist}
  Then $f_*$ commutes with $p_!$.
\end{lem}
\begin{proof}
  By \ssecref{ssec:colimshp} and since each functor in $\Ex_{!,*}$ preserves colimits ($f_*$ and $g_*$ by \itemref{item:Pr1}), it will suffice to show that $\Ex_{!,*} \circ t'_{\sX',!}$ is invertible.
  By assumption~\itemref{item:esjrlgtd2}, this is identified with $\Ex_{!,*}: (q \circ t)'_! g_{T',*} \to f_* (p\circ t'_{\sX'})_!$, which is invertible by assumption~\itemref{item:fmygscsc2}.
\end{proof}

\subsubsection{Proof of \itemref{item:Pr4}}

Suppose given a cartesian square
\begin{equation*}
  \begin{tikzcd}
    \sX'\ar{r}{g}\ar{d}{p}
    & \sY'\ar{d}{q}
    \\
    \sX\ar{r}{f}
    & \sY
  \end{tikzcd}
\end{equation*}
where $f$ is proper representable.
We check that $\Ex_{!,*} : q_!g_* \to f_*p_!$ \eqref{eq:Ex_!*} is invertible.

We argue by induction on the representability of $q$.
Assume first that $q$ is representable.
For every $(T,t)\in\Lis_\sY$, the base changed square
\begin{equation*}
  \begin{tikzcd}
    \sX'_T \ar{r}{g_T}\ar{d}{p_T}
    & \sY'_T \ar{d}{q_T}
    \\
    \sX_T \ar{r}{f_T}
    & T
  \end{tikzcd}
\end{equation*}
consists of spaces.
Thus $f_{T,*}$ commutes with $p_{T,!}$ by \itemref{item:Pr4} for spaces.
By \lemref{lem:pr!artredT} we conclude that $f_*$ commutes with $p_!$ in this case.

Next suppose that $q$ is $n$-representable, and that the claim is known for $f$ representable and $q$ $(n-1)$-representable.
Given $(T,t)\in\Lis_\sY$ and $(T',t')\in\Lis_{\sY'_T}$, form the diagram of cartesian squares
\begin{equation*}
  \begin{tikzcd}
    \sX'_{T'} \ar{r}{g_{T'}}\ar{d}{t'_{\sX'}}
    & T' \ar{d}{t'}
    \\
    \sX'_T \ar{r}{g_T}\ar{d}{p_T}
    & \sY'_T \ar{d}{q_T}
    \\
    \sX_T \ar{r}{f_T}
    & T.
  \end{tikzcd}
\end{equation*}
For every $(T',t')$, $t' : T' \to \sY'_T$ is $(n-1)$-representable since $\sY'_T$ is $n$-Artin.
Thus by the inductive hypothesis, $g_{T,*}$ commutes with $t'_!$.
Similarly, $f_{T,*}$ commutes with $(q_T\circ t')_!$ since $q_T \circ t'$ is representable and shriekable.
Applying \lemref{lem:pr!artredT'} as $(T',t')$ varies, we deduce that $f_{T,*}$ commutes with $q_{T,!}$.
Applying \lemref{lem:pr!artredT} as $(T,t)$ varies, we then deduce that $f_*$ commutes with $q_!$.

\subsection{\itemref{item:Pr5}}
\label{ssec:Pr5}

Let $\D^*_!$ be a presentable preweave on $(\Art,\cE_\lis)$ whose underlying presheaf $\D^*$ satisfies \v{C}ech descent along smooth covering morphisms.
We assume that the restriction of $\D^*_!$ to $(\Spc,\cE_0)$ admits $\sharp$-direct image for smooth morphisms of spaces, and $*$-direct image for proper morphisms of spaces.

\begin{lem}\label{lem:Pr5redT}
  Let $f : \sX \to \sY$ be a proper representable morphism of Artin stacks.
  Assume that $\D^*_!$ satisfies \itemref{item:Pr2} and \itemref{item:Pr3} for $f$, as well as \itemref{item:Pr5} for every base change $f_T : \sX_T \to T$ along $(T,t)\in\Lis_\sY$.
  Then it satisfies \itemref{item:Pr5} for $f$.
\end{lem}
\begin{proof}
  Let $i : \sY \to \sX$ be a shriekable section of $f$.
  Form the diagram of cartesian squares
  \begin{equation*}
    \begin{tikzcd}
      T \ar{r}{i_T}\ar{d}{t}
      & \sX_T \ar{r}{f_T}\ar{d}{t_\sX}
      & T \ar{d}{t}
      \\
      \sY \ar{r}{i}
      & \sX \ar{r}{f}
      & \sY.
    \end{tikzcd}
  \end{equation*}
  Since $f_*$ and $i_!$ satisfy the projection formula (the former by \itemref{item:Pr2}), it will suffice to show that $f_* i_!(\un_\sX) \in \D(\sX)$ is $\otimes$-invertible.
  By \lemref{lem:tetuqsba} and the commutativity of $f_*$ and $i_!$ with $*$-inverse image (the former by \itemref{item:Pr3}), this is equivalent to the $\otimes$-invertibility of $f_{T,*} i_{T,!}(\un_T) \in \D(T)$, which follows from the assumption that $f_T$ satisfies \itemref{item:Pr5}.
\end{proof}

\subsubsection{Proof of \itemref{item:Pr5}}

Let $f : \sX \to \sY$ be a proper representable morphism of Artin stacks and $i : \sY \to \sX$ a shriekable section.
By \lemref{lem:Pr5redT} it will suffice to show that for every $(T,t)\in\Lis_\sY$, the base change $f_T : \sX_T \to T$ satisfies \itemref{item:Pr5}.
Since $f_T : \sX_T \to T$ is a proper morphism of spaces, this holds by the assumption.

\subsection{Proof of \thmref{thm:presaxioms}\itemref{item:presaxioms/pr}}

Let $\D^*_!$ be a presentable preweave on $(\Art,\cE_\lis)$ whose underlying presheaf $\D^*$ satisfies \v{C}ech descent along smooth covering morphisms.
We assume that the restriction of $\D^*_!$ to $(\Spc,\cE_0)$ admits $\sharp$-direct image for smooth morphisms of spaces, and $*$-direct image for proper morphisms of spaces.

We have shown in Subsects.~\ref{ssec:Pr1}, \ref{ssec:Pr2}, \ref{ssec:Pr3}, \ref{ssec:Pr4} and \ref{item:Pr5} that $\D^*_!$ satisfies \itemref{item:Pr1}, \itemref{item:Pr2}, \itemref{item:Pr3}, \itemref{item:Pr4} and \itemref{item:Pr5}, respectively, for every proper representable morphism of Artin stacks.
In summary, $\D^*_!$ admits $*$-direct image for every proper representable morphism of Artin stacks.

\section{Construction of the extension}

\subsection{Proof of \thmref{thm:extweave}\itemref{item:extweave/repr}}

Let $\D^*_!$ be a left preweave on $(\Spc,\cE_0)$.
The claim follows by applying the following abstract extension result:

\begin{prop}[Liu--Zheng--Mann]\label{prop:extobj}
  Let $\cC$ be an \inftyCat and $\cE$ a class of morphisms in $\cC$ which is closed under base change.
  Let $\cC'$ be an \inftyCat containing $\cC$ as a full subcategory and $\cE'$ a class of morphisms in $\cC'$ which contains $\cE$.
  Suppose the following condition holds:
  \begin{thmlist}
    \item[$(\ast)$] For any morphism $f : X \to Y$ in $\cE'$, if $Y \in \cC$, then also $X \in \cC$ and $f \in \cE$.
  \end{thmlist}
  Then any left preweave $\D^*_!$ on $(\cC,\cE)$ extends uniquely to a left preweave $\D^{\lis,*}_!$ on $(\cC',\cE')$ such that the underlying presheaf $\D^{\lis,*} : \cC'^\op \to \Cat$ is the right Kan extension of $\D^* : \cC^\op \to \Cat$ along $\cC \sub \cC'$.
  Moreover, if $\D^*_!$ is a presentable preweave, then so is $\D^{\lis,*}_!$.
\end{prop}
\begin{proof}
  To apply \cite[Prop.~A.5.16]{Mann}, we also need to check the following condition:
  For any cartesian square in $\cC'$
  \[\begin{tikzcd}
    X' \ar{r}{g}\ar{d}{p}
    & Y' \ar{d}{q}
    \\
    X \ar{r}{f}
    & Y
  \end{tikzcd}\]
  where $X,Y,Y' \in \cC$ and $f \in \cE$, we must have $X' \in \cC$.
  But since $\cE$ is stable under base change, $g \in \cE$ and hence the condition in the statement implies that $X' \in \cC$.
\end{proof}

\subsection{Proof of \thmref{thm:extweave}\itemref{item:extweave/main}}

\subsubsection{}

Let $\D^*_!$ be a presentable preweave on $(\Spc,\cE_0)$ whose underlying presheaf $\D^*$ satisfies \v{C}ech descent along étale covering morphisms and which admits $\sharp$-direct image along smooth morphisms of spaces.

\subsubsection{}

Applying \thmref{thm:extweave}\itemref{item:extweave/repr}, we obtain a presentable preweave $\D^{\lis,*}_!$ on $(\Art, \cE_\lis^0)$ which is the unique extension of $\D^*_!$ with the property that its underlying presheaf $\D^{\lis,*} : (\Art)^\op \to \Cat$ is the lisse extension of $\D^* : (\Spc)^\op \to \Cat$.
Since $\D^*$ satisfies \v{C}ech descent along étale covering morphisms of spaces, this means that $\D^{\lis,*}$ is equivalently the unique extension of $\D^*$ satisfying \v{C}ech descent along smooth covering morphisms of Artin stacks (see \propref{prop:coccus} and \corref{cor:Shvreseq}).

\subsubsection{}

It remains to extend the preweave $\D^{\lis,*}_!$ from $(\Art, \cE_\lis^0)$ to $(\Art, \cE_\lis)$.
By induction, it will thus suffice to show:

\begin{prop}\label{prop:extweave!ind}
  Let $\D^*_!$ be a presentable preweave on $(\Art_n, \cE_n^{n-1})$, for some $n\ge 0$.
  If $\D^*_!$ admits $\sharp$-direct image for all smooth morphisms of $n$-Artin stacks and satisfies \v{C}ech descent along $(n-1)$-representable smooth covering morphisms of $n$-Artin stacks, then there exists a unique extension of $\D^*_!$ to a presentable preweave on $(\Art_n, \cE_n^n) = (\Art_n, \cE_n)$.
\end{prop}

\subsubsection{}

To prove \propref{prop:extweave!ind} we will apply the following abstract extension result (see \cite[Lem.~A.5.11]{Mann}).

\begin{prop}[Liu--Zheng--Mann]\label{prop:extmor}
  Let $\cC$ be an \inftyCat and $\cE$ a class of morphisms in $\cC$ which is closed under base change.
  Let $\cE'$ be a class of morphisms in $\cC$ which contains $\cE$.
  Suppose the following condition holds:
  \begin{thmlist}
    \item[($\ast$)] For every $X \in \cC$ and every $f : Y \to X$ in $\cE'$, the canonical functor
    \[
      \D(Y) \to \lim_{(S,s)} \D^!(S)
    \]
    is equivalence, where the limit is over pairs $(S,s)$ with $S \in \cC$ and $s : S \to Y$ a morphism in $\cC$ such that $s$ and $f\circ s : S \to X$ both lie in $\cE$.
  \end{thmlist}
  Then any presentable preweave $\D^*_!$ on $(\cC,\cE)$ extends uniquely to a presentable preweave on $(\cC,\cE')$.
\end{prop}

\subsubsection{}

The following is the condition~($\ast$) of \propref{prop:extmor} specialized to the situation of \propref{prop:extweave!ind}:

\begin{thm}\label{thm:!sat}
  Let $\D^*_!$ be a presentable preweave on $(\Art_n, \cE_n^{n-1})$, for some $n\ge 0$.
  Suppose $\D^*_!$ admits $\sharp$-direct image for all smooth morphisms of $n$-Artin stacks and satisfies \v{C}ech descent along $(n-1)$-representable smooth covering morphisms of $n$-Artin stacks.
  Then for every morphism of $n$-Artin stacks $f : \sY \to \sX$ that lies in $\cE_n$, the canonical functor
  \begin{equation}\label{eq:hynzljua}
    \D(\sY) \to \lim_{(\sS,s)} \D^{!}(\sS)
  \end{equation}
  is equivalence, where the limit is over $(\sS,s)$ with $\sS$ $n$-Artin and $s : \sS \to \sY$ an $(n-1)$-representable morphism in $\cE$ such that $f\circ s : \sS \to \sX$ is $(n-1)$-representable (and the transition functors are given by $!$-inverse image).
\end{thm}

We will prove \thmref{thm:!sat} in \ssecref{ssec:!satproof} below.
This will thus conclude the proof of \propref{prop:extweave!ind} and thus that of \thmref{thm:extweave}\itemref{item:extweave/main}.

\subsection{Digression}

In this subsection we provide an abstract criterion (\corref{cor:ofbcuwul}) for a presheaf to satisfy the condition appearing in \thmref{thm:!sat}.

\subsubsection{}

Recall the class of morphisms $\cE_0^\lis$ in $\Art$ \sssecref{sssec:E_lis^0}.

Let $\sX$ be an $(n+1)$-Artin stack, $n\ge 0$, and denote by $\cE_\sX(n)$ the \inftyCat whose objects are pairs $(\sS, s : \sS \to \sX)$ where $\sS$ is an $(n+1)$-Artin stack and $s$ lies in the class $\cE$, and whose morphisms $(\sS',s') \to (\sS,s)$ are \emph{$n$-representable} morphisms $\sS' \to \sS$ compatible with $s'$ and $s$.

Denote by $\cE^0_\sX(n)$ the \inftyCat whose objects are pairs $(\sS, s)$ where $\sS$ is an $(n+1)$-Artin stack and $s : \sS \to \sX$ is a \emph{$n$-representable} morphism that lies in $\cE$, and whose morphisms $(\sS',s') \to (\sS,s)$ are (necessarily $n$-representable) morphisms $\sS' \to \sS$ compatible with $s'$ and $s$.

Thus $\cE^0_\sX(n)$ is a full subcategory of $\cE_\sX(n)$.
An object of $\cE_\sX(n)$ of the form $(\sS,s)$ belongs to $\cE^0_\sX(n)$ if and only if $s : \sS \to \sX$ is $n$-representable.
We have the forgetful functors:
\[\begin{tikzcd}
  \cE^0_\sX(n) \ar[hookrightarrow]{rr}\ar{rd}
  & & \cE_\sX(n) \ar{ld}
  \\
  & \cE_\lis^0. &
\end{tikzcd}\]

\subsubsection{}

Given an \inftyCat with limits $\cV$, we denote by $F \mapsto F^{\mrm{RKE}}$ the functor of right Kan extension of $\cV$-valued presheaves along the above inclusion.
Thus we have for $\sY \in \cE_\sX(n)$,
\begin{equation}
  F^{\mrm{RKE}}(\sY) \simeq \lim_{(\sS,s)} F(\sS),
\end{equation}
where the limit is taken over the \inftyCat $\cI_{\sY/\sX}(n)$ of pairs $(\sS,s)$ where $\sS$ is an $(n+1)$-Artin stack and both morphisms $s : \sS \to \sY$ and $\sS \to \sY \to \sX$ are \emph{$n$-representable} and lie in $\cE$.
Morphisms $(\sS',s') \to (\sS,s)$ in $\cI_{\sY/\sX}(n)$ are (necessarily $n$-representable) morphisms $\sS' \to \sS$ compatible with $s'$ and $s$.

\begin{prop}\label{prop:RKEpresdesc}
  Let $F : (\cE^0_\sX(n))^\op \to \cV$ be a presheaf satisfying \v{C}ech descent along smooth covering morphisms of Artin stacks.
  Then its right Kan extension $F^{\mrm{RKE}} : (\cE_\sX(n))^\op \to \cV$ satisfies \v{C}ech descent along smooth covering morphisms of Artin stacks.
\end{prop}

\begin{lem}\label{lem:cofinE}
  Let $F : (\cE^0_\sX(n))^\op \to \cV$ be a presheaf satisfying \v{C}ech descent along smooth covering morphisms of Artin stacks.
  For every $n$-representable morphism $f : \sY_1 \to \sY_2$ over $\sX$, the canonical map
  \[
    F^{\mrm{RKE}}(\sY_1) \to \lim_{(\sT,t)\in\cI_{\sY_2/\sX}(n)} F(\sY_1 \fibprod_{\sY_2} \sT)
  \]
  is invertible.
\end{lem}
\begin{proof}
  Consider the base change functor
  \[
    f^* : \cI_{\sY_2/\sX}(n) \to \cI_{\sY_1/\sX}(n).
  \]
  Note that this is well-defined since if $(\sT, t : \sT \to \sY_2)$ is a pair with $t$ and $\sT \to \sY_2 \to \sX$ $n$-representable, then $\sT \fibprod_{\sY_2} \sY_1 \to \sY_1$ and $\sT \fibprod_{\sY_2} \sY_1 \to \sY_1 \to \sX$ are still $n$-representable.
  Note moreover that $f^*$ admits a left adjoint $f_\sharp$, sending a pair $(\sS, s : \sS \to \sY_1)$ to the composite $(\sS, f \circ s : \sS \to \sY_1 \to \sY_2)$; this is also well-defined since $f : \sY_1 \to \sY_2$ is $n$-representable.
  It follows that $f^*$ is cofinal.
\end{proof}

\begin{proof}[Proof of \propref{prop:RKEpresdesc}]
  Let $f : \sY_1 \surj \sY_2$ be a smooth covering morphism in $\cE_\sX(n)$, i.e., $f$ is a $n$-representable smooth covering morphism over $\sX$.
  For every $(\sS,s) \in \cI_{\sY_2/\sX}(n)$, the base change $\sY_1 \fibprod_{\sY_2} \sS \surj \sS$ is also a $n$-representable smooth covering morphism.
  Moreover, $\sS$ and $\sY_1\fibprod_{\sY_2}\sS$ belong to $\cE^0_\sX(n)$ since $\sS \to \sY_2 \to \sX$ and $\sY_1\fibprod_{\sY_2}\sS \to \sY_1 \to \sX$ are both $n$-representable.\footnote{%
    Indeed the latter factors as the composite of the $n$-representable morphisms $\sY_1\fibprod_{\sY_2}\sS \to \sS$ and $\sS \to \sX$.
  }
  Hence by assumption we have for each $(\sS,s)$ the canonical isomorphism
  \[
    F(\sS)
    \to \Tot(F(\sY_{1,\bullet}\fibprod_{\sY_2}\sS))
  \]
  where $\sY_{1,\bullet}$ denotes the \v{C}ech nerve of $\sY_1 \surj \sY_2$.
  Passing to the limit over $(\sS,s)$ and using \lemref{lem:cofinE}, we deduce that the canonical map
  \[
    F^{\mrm{RKE}}(\sY_2)
    \to \Tot(F^{\mrm{RKE}}(\sY_{1,\bullet}))
  \]
  is invertible.
\end{proof}

Denote by $\Shv(\cE_\sX(n))_\cV$ the full subcategory of presheaves on $\cE_\sX(n)$ satisfying \v{C}ech descent with respect to smooth covering morphisms, and similarly for $\Shv(\cE^0_\sX(n))_\cV$.

\begin{cor}\label{cor:ShvJres}
  The restriction functor
  \begin{equation}\label{eq:jfkggmqq}
    \Shv(\cE_\sX(n))_\cV \to \Shv(\cE^0_\sX(n))_\cV
  \end{equation}
  is an equivalence, whose inverse is $F \mapsto F^{\mrm{RKE}}$.
\end{cor}
\begin{proof}
  At the level of presheaves, $F \mapsto F^{\mrm{RKE}}$ is right adjoint to the restriction functor.
  Since it preserves the descent condition by \propref{prop:RKEpresdesc}, it restricts to a right adjoint to \eqref{eq:jfkggmqq}.
  To show that it is an equivalence, it will suffice as in the proof of \corref{cor:Shvreseq} to consider the case $\cV = \Ani$.
  Then \eqref{eq:jfkggmqq} also admits a left adjoint, given by left Kan extension (followed by localization), which is fully faithful and colimit-preserving.
  It will thus suffice to show that it generates under colimits.

  Let $(\sS,s)$ be an object of $\cE_\sX(n)$.
  Let $p : X \surj \sX$ be a smooth covering morphism where $X$ is a space and $q : S \surj \sS\fibprod_\sX X$ a smooth covering morphism where $S$ is a space.
  Consider the commutative diagram
  \[\begin{tikzcd}
    S \ar[twoheadrightarrow]{r}{q}\ar{rd}
    & \sS\fibprod_\sX X \ar{r}\ar{d}
    & \sS \ar{d}{s}
    \\
    & X \ar[twoheadrightarrow]{r}{p}
    & \sX.
  \end{tikzcd}\]
  Note that $p$ and $q$ are $n$-representable since $\sX$ and $\sS$ are $(n+1)$-Artin.
  In particular, $s_0 : S \to X \surj \sX$ is $n$-representable as the composite of $p$ and a morphism of spaces, hence $(S,s_0)$ belongs to $\cE^0_\sX(n)$.
  Since $S \surj \sS\fibprod_\sX X \to \sS$ is also $n$-representable, it determines a morphism $(S,s_0) \to (\sS,s)$ in $\cE_\sX(n)$.
  By construction, its \v{C}ech nerve $(S,s_0)_\bullet$ has geometric realization isomorphic to $(\sS,s)$ in $\Shv(\cE_\sX(n))$.
  Since every term of $(S,s_0)_\bullet$ is $n$-representable over $\sX$ and hence belongs to $\cE^0_\sX(n)$, the claim follows.
\end{proof}

Recall from \sssecref{sssec:E_n} that $\cE_{n+1}^n$ is the class of $n$-representable morphisms in $\Art_{n+1}$.

\begin{cor}
  Let $F : (\cE_{n+1}^n)^\op \to \cV$ be a presheaf satisfying \v{C}ech descent along $n$-representable smooth covering morphisms of Artin stacks.
  For every $(n+1)$-Artin stack $\sX$, $n\ge 0$, there is a canonical isomorphism
  \[
    F|_{\cE_\sX(n)}
    \to (F|_{\cE^0_\sX(n)})^{\mrm{RKE}}
  \]
  of $\cV$-valued presheaves on $\cE_\sX(n)$.
\end{cor}
\begin{proof}
  By \corref{cor:ShvJres}, the target is uniquely determined by the fact that it extends $F|_{\cE^0_\sX(n)}$ and satisfies \v{C}ech descent along smooth covering morphisms in $\cE_\sX(n)$.
  Let us show the same for $F|_{\cE_\sX(n)}$.
  First, we clearly have
  \[ (F|_{\cE_\sX(n)})|_{\cE^0_\sX(n)} \simeq F|_{\cE^0_\sX(n)}. \]
  Second, smooth covering morphisms in $\cE_\sX(n)$ are $n$-representable by definition, so $F$ satisfies \v{C}ech descent along them by the assumption.
\end{proof}

\begin{cor}\label{cor:ofbcuwul}
  Let $F : (\cE_{n+1}^n)^\op \to \cV$ be a presheaf satisfying \v{C}ech descent along $n$-representable smooth covering morphisms of Artin stacks.
  Then for every $(n+1)$-Artin stack $\sX$, $n\ge 0$, and every $(n+1)$-Artin stack $\sY$ over $\sX$, the canonical map
  \[
    F(\sY) \to \lim_{(\sS,s) \in \cI_{\sY/\sX}(n)} F(\sS)
  \]
  is invertible.
\end{cor}

\subsection{Proof of \thmref{thm:!sat}}
\label{ssec:!satproof}

\subsubsection{}

Let $\D^*_!$ be a presentable preweave on $(\Art_n, \cE_n^{n-1})$, and assume that $\D^*_!$ admits $\sharp$-direct image for all smooth morphisms of $n$-Artin stacks and satisfies \v{C}ech descent along $(n-1)$-representable smooth covering morphisms of $n$-Artin stacks.

Recall that $\D^{*} : (\Art)^\op \to \Cat$ satisfies \v{C}ech descent along any smooth covering morphism of Artin stacks (\propref{prop:coccus}).
Since every smooth $(n-1)$-representable morphism $f$ in $\Art_n$ is shriekable for $\D^*_!$, we have by \corref{cor:poinc} the Poincaré duality isomorphism $f^! \simeq \Sigma_f f^*$.
It then follows by \propref{prop:descweave} that $\D^{!} : (\cE_n^{n-1})^\op \to \Cat$ satisfies \v{C}ech descent along $(n-1)$-representable smooth covering morphisms of $n$-Artin stacks.
We may now apply \corref{cor:ofbcuwul}, which yields the claim.

\section{Example: weaves on algebraic stacks}

\subsection{Weaves satisfying étale descent}
\label{ssec:weaveschet}

\subsubsection{}

Let $k$ be a commutative ring and denote by $\Asp_k$ denote the category of locally of finite type $k$-algebraic spaces.
Denote by $\Asp_k^\sm$ the class of smooth morphisms, $\Asp_k^\etcov$ the class of surjective étale morphisms, $\Asp_k^\pr$ the class of proper morphisms.
The data $(\Asp_k,\Asp_k,\Asp_k^\sm,\Asp_k^\pr)$ satisfies the axioms of \sssecref{sssec:Csmpr}.

\subsubsection{}

Let $\D^*_!$ be a presentable weave on $(\Asp_k,\Asp_k,\Asp_k^\sm,\Asp_k^\pr)$.
That is, $\D^*_!$ is a presentable preweave on $\Asp_k$ where all morphisms are shriekable, which admits $\sharp$-direct image along smooth morphisms and $*$-direct image along proper morphisms.\footnote{%
  One can use an easier version of the following argument to uniquely lisse-extend weaves satisfying étale descent from schemes to algebraic spaces.
}

We assume that the underlying presheaf $\D^* : (\Asp_k)^\op \to \Cat$ satisfies étale descent, i.e., \v{C}ech descent along étale surjections.

Examples include:
\begin{enumerate}
  \item\emph{Betti sheaves:}
  If $k$ is a $\bC$-algebra, we may take $X \mapsto \D(X)$ sending $X$ to the \inftyCat of sheaves of $R$-modules on the topological space $X(\bC)$, where $R$ is a commutative ring\footnote{%
    or more generally, an $\Einfty$-ring spectrum; see \ssecref{ssec:ShvTop}
  }.

  \item\emph{Mixed Hodge modules:}
  If $k$ is a $\bC$-algebra, we may take $X \mapsto \D(X)$ sending $X$ to the \inftyCat $\on{Ind}\on{D}^{\mrm{b}}(\on{MHM}(-))$ of (ind-completed\footnote{%
    This means that we formally adjoin filtered colimits in order to turn the small \inftyCats $\on{D}^{\mrm{b}}(\on{MHM}(-))$ into presentable ones.
  }) mixed Hodge modules on $X$.
  (See \cite{Tubach} for the $\infty$-categorical enhancement of Saito's triangulated categories.)

  \item\emph{Étale sheaves (finite coefficients):}
  If $n \in k^\times$, we may take $X \mapsto \D(X)$ sending $X$ to the \inftyCat of sheaves of $\bZ/n\bZ$-modules on the small étale site of $X$.

  \item\emph{Étale sheaves ($\ell$-adic coefficients):}
  If $\ell \in k^\times$, we may take $X \mapsto \D(X)$ sending $X$ to the \inftyCat $\on{D}_\et(X, \bZ_\ell)$ of $\ell$-adic sheaves on the small étale site of $X$, i.e., the limit $\on{D}_\et(X, \bZ/\ell^n\bZ)$ over $n>0$.
\end{enumerate}
These examples can all be constructed as weaves using \cite[Thm.~2.51]{weaves} or \cite[Thm.~3.3]{DauserKuijper}.

\subsubsection{}

Let $\Art_k$ denote the \inftyCat of Artin stacks locally of finite type over $k$.
We denote by $\Art_k^\sm$ and $\Art_k^{\pr,\repr}$ the classes of  smooth morphisms and proper representable morphisms of Artin stacks, respectively.
We claim:

\begin{cor}
  Let $\D^*_!$ be a presentable weave on $(\Asp_k,\Asp_k,\Asp_k^\sm,\Asp_k^\pr)$.
  There exists a unique presentable weave $\D^{\lis,*}_!$ on $(\Art_k,\Art_k,\Art_k^\sm,\Art_k^{\pr,\repr})$ whose underlying presheaf $\D^{\lis,*}$ is the lisse extension of $\D^* : (\Asp_k)^\op \to \Cat$ (see \ssecref{ssec:lisextpresheaf}).
  That is, $\D^{\lis,*}_!$ is a presentable weave on $\Art_k$ where all morphisms are shriekable, the ``smooth'' morphisms are smooth morphisms of Artin stacks, and the ``proper'' morphisms are proper representable morphisms of Artin stacks.
\end{cor}
\begin{proof}
  The standard notion of (higher) Artin stacks in algebraic geometry is the same as the abstract notion defined in \sssecref{sssec:nArt} when we take the context $(\cC,\cC^\sm,\cC^\etcov)$ to be $(\Asp_k, \Asp_k^\sm, \Asp_k^\etcov)$ (recall that $\Asp_k^\etcov$ is the class of étale surjections.)
  Note that the smooth covering morphisms in the sense of \sssecref{sssec:smooth covering} are the smooth surjections.
  Hence the claim follows by applying \corref{cor:weave}.
\end{proof}

\sssec{}

More generally, using \corref{cor:weavenonart} we find that $\D^{\lis,*}_!$ extends to a presentable weave on the entire \inftyCat $\Stk_k$ of stacks locally of finite type over $k$.
The shriekable morphisms are those which are representable by Artin stacks, the ``smooth'' morphisms are those whose base change to any object of $\Asp_k$ is a smooth morphism of Artin stacks, and the ``proper'' morphisms are those whose base change to any object of $\Asp_k$ is a proper morphism of algebraic spaces.

\subsection{Weaves satisfying Nisnevich descent}
\label{ssec:weaveschnis}

\subsubsection{}

Let $k$ be a commutative ring and let $\Asp'_k$ be the full subcategory of $\Asp_k$ spanned by Zariski-locally quasi-separated algebraic spaces (that are locally of finite type over $k$).
We let $\D^*_!$ be a presentable weave on $\Asp'_k$ where all morphisms are shriekable, the ``smooth'' morphisms are the smooth morphisms, and the ``proper'' morphisms are the proper morphisms.

Unlike in \ssecref{ssec:weaveschet}, we only assume that the underlying presheaf $\D^* : (\Asp'_k)^\op \to \Cat$ satisfies \emph{Nisnevich} descent, i.e., \v{C}ech descent along étale morphisms that are \emph{cd-surjective}.
We say that a morphism in $\Asp_k$ is cd-surjective if it is surjective on $\kappa$-points for every field extension $\kappa/k$.

\subsubsection{}
\label{sssec:weavenisex}

Any topological weave in the sense of \cite[\S 3]{weaves} is an example of a weave satisfying Nisnevich descent.
In particular, the following examples define topological weaves by \cite[Cor.~3.33]{weaves} (see also \cite{six}):

\begin{enumerate}
  \item\emph{Motives:}
  Take $X \mapsto \D(X)$ sending $X$ to the \inftyCat $\on{D}_{\mrm{mot}}(X) := \on{D_{H\bZ}}(X)$ of modules over the integral motivic Eilenberg--MacLane spectrum $H\bZ_X$ as in \cite{Spitzweck}.
  (See \cite[Thm.~5.1]{CisinskiDegliseIntegral} and \cite[\S 14]{CisinskiDegliseBook} for comparisons with other constructions of derived categories of motives.)

  \item\emph{Cobordism motives:}
  Take $X\mapsto \D(X)$ sending $X$ to the \inftyCat $\on{D_{\mrm{MGL}}}$ of modules over Voevodsky's algebraic cobordism spectrum $\mrm{MGL}_X$.

  \item\emph{Motivic stable homotopy theory:}
  Take $X\mapsto \D(X)$ sending $X$ to the \inftyCat $\on{SH}(X)$ of motivic spectra on $X$.
\end{enumerate}

\subsubsection{}

Take the context $(\cC,\cC^\sm,\cC^\etcov)$ to be $(\Asp'_k, \Asp_k'^\sm, \Asp_k'^\etcov)$, where $\Asp_k'^\sm$ is the class of smooth morphisms in $\Asp_k'$ and $\Asp_k'^{\etcov}$ is the class of étale cd-surjective morphisms in $\Asp_k'$.
In this context, the smooth covering morphisms in the sense of \sssecref{sssec:smooth covering} are the smooth cd-surjective morphisms (see \cite[Lem.~0.6]{equilisse}).

Therefore, the abstract notion of ``Artin stacks'' (resp. ``$n$-Artin stacks'') defined in \sssecref{sssec:nArt} with respect to the context $(\Asp'_k, \Asp_k'^\sm, \Asp_k'^\etcov)$ is precisely the notion of ``$\Nis$-Artin stacks'' (resp. ``$(\Nis,n)$-Artin stacks'') as defined in \cite[0.2.2]{equilisse}.
Every $(\Nis,n)$-Artin stack is $n$-Artin, and any $1$-Artin stack which is quasi-separated with separated diagonal is in fact $(\Nis,1)$-Artin (see \cite[Thm.~0.7]{equilisse}).

Denote by $\Art_k'^\sm$ and $\Art_k'^{\pr,\repr}$ the classes of smooth morphisms and proper representable morphisms in $\Art_k'$, respectively.
In the notation of \sssecref{sssec:E_lis} and \sssecref{sssec:E_lis^0} we have $\Art_k'^\sm = (\Asp_k'^\sm)_\lis$ and $\Art_k'^{\pr,\repr} = (\Asp_k'^\pr)_\lis^0$.
Applying \corref{cor:weave} to this context thus yields:

\begin{cor}
  Let $\D^*_!$ be a presentable weave on $\Asp_k'$.
  There exists a unique presentable weave $\D^{\lis,*}_!$ on $(\Art_k',\Art_k',\Art_k'^\sm,\Art_k'^{\pr,\repr})$ whose underlying presheaf $\D^{\lis,*}$ is the lisse extension of $\D^* : (\Asp_k')^\op \to \Cat$ (see \ssecref{ssec:lisextpresheaf}).
\end{cor}

\sssec{}

More generally, using \corref{cor:weavenonart} we find that $\D^{\lis,*}_!$ extends to a presentable weave on the entire \inftyCat $\Stk_k$ of stacks locally of finite type over $k$.
The shriekable morphisms are those which are representable by Artin stacks, the ``smooth'' morphisms are those whose  base change to any object of $\Asp_k$ is a smooth morphism of Artin stacks, and the ``proper'' morphisms are those whose base change to any object of $\Asp_k$ is a proper morphism of algebraic spaces.

\section{Example: sheaves on topological stacks}
\label{sec:shvtop}

\subsection{Sheaves on topological spaces}
\label{ssec:ShvTop}

\subsubsection{}

All topological spaces are implicitly assumed locally compact and Hausdorff.
We write $\Top$ for the category of topological spaces; a \emph{morphism} of topological spaces is a continuous map.

\subsubsection{}

Let $R$ be an $\Einfty$-ring spectrum and denote by $\Mod_R$ the stable \inftyCat of $R$-modules.
For a topological space $X$, we denote by $\Shv(X; R)$ the stable \inftyCat of sheaves on $X$ with values in $\Mod_R$.
That is, its objects are presheaves $\mrm{Op}_X^\op \to \Mod_R$, where $\mrm{Op}_X$ is the poset of opens of $X$, which satisfy \v{C}ech descent.
Below, we will regard $R$ as fixed and write simply $\Shv(X) := \Shv(X; R)$.

\subsubsection{}

For a morphism of topological spaces $f : X \to Y$, the $*$-inverse image functor
\[ f^* : \Shv(Y) \to \Shv(X) \]
is the unique colimit-preserving functor sending the sheaf on $Y$ represented by an open $V \sub Y$ to the sheaf on $X$ represented by the open $f^{-1}(V) \sub X$.

Its right adjoint, the $*$-direct image functor $f_*$, is given informally by the formula $\Gamma(V, f_*(\sF)) = \Gamma(f^{-1}(V), \sF)$ where $\sF \in \Shv(X)$.

We denote by
\begin{equation}
  \Shv^* : \Top^\op \to \PrL
\end{equation}
the functor sending $X \mapsto \Shv(X)$ and $f \mapsto f^*$.

\subsubsection{}

The presheaf of \inftyCats $\Shv^*$ admits $\sharp$-direct image for topological submersions.
Indeed, let $f : X \to Y$ be topological submersion of topological spaces.
Then we have:

\noindent\itemref{item:Sm1}
$f^*$ admits a left adjoint $f_\sharp$: see \cite[Lem.~3.24]{Volpe}.

\noindent\itemref{item:Sm2}
$f_\sharp$ satisfies the projection formula: see \cite[Cor.~3.26]{Volpe}.

\noindent\itemref{item:Sm3}
$f_\sharp$ commutes with $*$-inverse image: see \cite[Lem.~3.25]{Volpe}.

\subsubsection{}

The presheaf of \inftyCats $\Shv^*$ admits $*$-direct image for proper morphisms.
Indeed, let $f : X \to Y$ be a \emph{proper} morphism of topological spaces.
Then we have:

\noindent\itemref{item:Pr1}
$f_*$ is colimit-preserving: combine \cite[Rem.~7.3.1.5]{LurieHTT}, \cite[Thm.~7.3.1.16]{LurieHTT}, and \cite[Cor.~3.12]{HaineBC}.
Alternatively, see \cite[Lem.~5.14]{Volpe}.

\noindent\itemref{item:Pr2}
$f_*$ satisfies the projection formula: see \cite[Lem.~6.3, Prop.~6.12]{Volpe}.

\noindent\itemref{item:Pr3}
$f_*$ commutes with $*$-inverse image: in this generality, see \cite{HaineBC}.

\subsubsection{}
\label{sssec:ShvTopdesc}

The presheaf of \inftyCats $\Shv^* : (\Top)^\op \to \Cat$ satisfies \v{C}ech descent with respect to surjective local homeomorphisms.
For example, this follows from \lemref{lem:desc*}: it is clear that $f^*$ is conservative for any surjective local homeomorphism $f$.

\subsubsection{}

Let $\Top^\sm$ be the class of topological submersions in $\Top$ and $\Top^\pr$ the class of proper morphisms in $\Top$.
The data $(\Top,\Top,\Top^\sm,\Top^\pr)$ satisfies the axioms of \sssecref{sssec:Csmpr}.
In other words, the ``smooth'' morphisms in $\Top$ are the topological submersions, the ``proper'' morphisms are the proper morphisms, and all morphisms are shriekable.

We extend the presheaf $\Shv^*$ to a weave $\Shv^*_!$ on $(\Top,\Top^\sm,\Top^\pr,\Top)$.
That is, $\Shv^*_!$ will be a preweave on $\Top$ where all morphisms are shriekable, topological submersions admit $\sharp$-direct image, and proper morphisms admit $*$-direct image.

\begin{thm}\label{thm:topweave}
  There exists a presentable weave
  \begin{equation}\label{eq:Shv^*_!}
    \Shv^*_! : \Corr(\Top) \to \PrL
  \end{equation}
  on $(\Top,\Top,\Top^\sm,\Top^\pr)$, which is uniquely characterized by the following identifications:
  \begin{thmlist}
    \item $\Shv^*_!|_{(\Top)^\op} \simeq \Shv^*$,
    \item $\Shv^*_!|_{\Top^\pr} \simeq \Shv_*$.
    \item $\Shv^*_!|_{\Top^\mrm{op}} \simeq \Shv_!$, where $\Top^\mrm{op}$ is the class of open embeddings in $\Top$.
  \end{thmlist}
\end{thm}

\subsection{Proof of \thmref{thm:topweave}}

We begin with some preliminary lemmas.

\begin{lem}\label{lem:compactfilt}
  Let $f : X \to Y$ be a morphism of topological spaces.
  Consider the category $\cC(f)$ of factorizations of $f$ of the form
  \[ X\xrightarrow{j} P \xrightarrow{p} Y, \]
  where $j$ is an open embedding and $p$ is proper, in which morphisms are commutative diagrams
  \begin{equation*}
    \begin{tikzcd}[matrix yscale=0.5]
      & P\ar{dd}\ar{rd}{p} &
      \\
      X \ar{ru}{j}\ar[swap]{rd}{j'}
      & 
      & Y.
      \\
      & P' \ar[swap]{ru}{p'} &
    \end{tikzcd}
  \end{equation*}
  Then $\cC(f)$ is cofiltered.
  In particular, its underlying \inftyGrpd is contractible.
\end{lem}
\begin{proof}
  Since $X$ is locally compact and Hausdorff by assumption, the canonical map $X \hook \beta X$ to the Stone-\v{C}ech compactification is an open embedding.
  By the universal property of the latter, $f$ factors through a morphism $\beta f : \beta X \to \beta Y$.
  As a morphism between compact spaces, $\beta f$ is proper.
  Thus $f$ factors canonically as the composite
  \begin{equation}\label{eq:SC}
    f : X \xrightarrow{j} \beta X \fibprod_{\beta Y} Y \xrightarrow{p} Y,
  \end{equation}
  where $j$ is an open embedding and $p$, the base change of $\beta f$, is proper.
  This shows that $\cC(f)$ is nonempty.
  
  Given two factorizations $X \hook P \to Y$ and $X \hook P' \to Y$, $X \hook P \fibprod_Y P' \to Y$ is a new factorization mapping to both.

  Given two factorizations $X \hook P \to Y$ and $X \hook P' \to Y$ with parallel arrows $P \rightrightarrows P'$, let $P_0$ denote the equalizer.
  Then $X \hook P$ factors through $X \hook P_0$.
  Choose a compactification $X \hook P'' \to P_0$ of the latter.
  Then $X \hook P'' \to Y$ is a compactification of $f$ which is equalized by $P \rightrightarrows P'$.
\end{proof}

\begin{lem}[K\"unneth formula]\label{lem:propf_*kunn}
  Let $f_1 : X_1 \to Y_1$ and $f_2 : X_2 \to Y_2$ be proper morphisms of topological spaces.
  Then the natural transformation
  \[
    f_{1,*}(-) \boxtimes f_{2,*}(-)
    \to (f_1 \times f_2)_*(- \boxtimes -),
  \]
  of functors $\Shv(X_1)\otimes\Shv(X_2) \to \Shv(X_1 \times X_2)$, is invertible.
\end{lem}
\begin{proof}
  This follows from the base change and projection formulas for proper $*$-direct image.
\end{proof}

\begin{lem}\label{lem:supp}
  For every cartesian square
  \[\begin{tikzcd}
    U \ar{r}{g}\ar{d}{u}
    & V \ar{d}{v}
    \\
    X \ar{r}{f}
    & Y
  \end{tikzcd}\]
  where $f$ is proper and $v$ is an open embedding, $f_*$ commutes with $v_\sharp$.
  That is, the natural transformation $\Ex_{\sharp,*} : v_\sharp g_* \to f_* u_\sharp$ \eqref{eq:Ex_shp,*} is invertible.
\end{lem}
\begin{proof}
  If $f$ is a closed embedding, this follows easily from the localization triangle, see e.g. \cite[Cor.~4.10]{Volpe}.
  In general, choose a compactification $\overline{X}$ of $X$ so that $f$ factors as the composite
  \[ X \xrightarrow{i} \overline{X} \times Y \xrightarrow{p} Y \]
  where $p$ is the projection.
  Since the source and target of $i$ are both proper over $Y$, $i$ is proper and hence a closed embedding.
  It will thus suffice to show the claim for $p$ instead of $f$.
  This follows easily from \lemref{lem:propf_*kunn}.
\end{proof}

\subsubsection{Proof of \thmref{thm:topweave}}

We apply the work of \cite{LiuZheng} as summarized in \cite[Prop.~A.5.10]{Mann}\footnote{%
  Alternatively, apply Theorems~3.2.2(b) and 5.2.4 of \cite[Chap.~7]{GaitsgoryRozenblyum}.
}, which provides the preweave $\Shv^*_!$ with the asserted properties, once the following conditions are verified:
\begin{enumerate}
  \item The class $\Top^\mrm{op}$ is closed under base change and sections (and hence also under two-of-three).
  \item The class $\Top^\pr$ is closed under base change and sections (and hence also under two-of-three).
  \item Any morphism that is in $\Top^\mrm{op} \cap \Top^\pr$ is a monomorphism.
  \item The category of compactifications of any morphism in $\Top$ is contractible (\lemref{lem:compactfilt}).
  \item $\Shv^*$ admits $\sharp$-direct image for open embeddings.
  \item $\Shv^*$ admits $*$-direct image for proper morphisms.
  \item For every proper morphism $f : X \to Y$ and every open embedding $j : V \hook Y$, $f_*$ commutes with $j_!$ (see \lemref{lem:supp}).
\end{enumerate}
Uniqueness holds by \cite[Thm.~3.3]{DauserKuijper}.

It remains to show that $\Shv^*_!$ is a weave, i.e., that it admits $\sharp$-direct image with respect to topological submersions and $*$-direct image with respect to proper morphisms.
The remaining axioms are as follows:

\noindent\itemref{item:Sm4}:
For every topological submersion $f : X \to Y$, $f_\sharp$ commutes with $!$-direct image.
This holds by \cite[Cor.~6.14]{Volpe}.

\noindent\itemref{item:Sm5}
For every topological submersion $f : X \to Y$ and every section $i : Y \to X$, the functor $f_\sharp i_!$ is an equivalence.
For this, combine \lemref{lem:checkSm5} and \cite[Prop.~7.7]{Volpe}.

\noindent\itemref{item:Pr4}
For every proper morphism $f : X \to Y$, $f_*$ commutes with $!$-direct image.
Indeed, to show that $f_*$ commutes with $q_!$ for a morphism $q : Y' \to Y$, we may assume that $q$ is either proper or an open embedding (by choosing a compactification of $q$).
When $q$ is an open embedding, the statement is \lemref{lem:supp}.
If $q$ is proper, consider the exchange transformation \eqref{eq:Ex_!*}
\[
  \Ex_{!,*} : q_!g_*
  \xrightarrow{\unit} q_!g_* p^!p_!
  \simeq q_! q^!f_* p_!
  \xrightarrow{\counit} f_*p_!.
\]
We have by construction of $\Shv^*_!$ that, under the identifications between $q_! \simeq q_*$ and $p_! \simeq p_*$, the base change isomorphism $\Ex_!^* : g_* p^! \simeq q^! f_*$ is identified with the proper base change isomorphism $\Ex_*^* : q^* f_* \simeq g_* p^*$ \eqref{eq:Ex_*^*}, and hence $\Ex_{!,*}$ is identified with the functoriality isomorphism $q_* g_* \simeq f_* p_*$.

\noindent\itemref{item:Pr5}
For every proper morphism $f : X \to Y$ and every section $i : Y \to X$, the functor $f_* i_!$ is an equivalence.
Indeed, we have by construction of $\Shv^*_!$ the tautological identification $f_! \simeq f_*$, under which the given functor is equivalent to the identity.

\subsection{Topological stacks}

Recall that we have defined the ``smooth'' morphisms in $\Top$ to be the topological submersions.
We define the ``étale covering'' morphisms to be the surjective local homeomorphisms.
Note that the condition of being a topological submersion is local on the source in the sense that if $f : X \to Y$ is a morphism such that there exists a surjective local homeomorphism $p : X' \surj X$ with $f \circ p$ a topological submersion, then $f$ is a topological submersion.
Thus if $\Top^\sm$ and $\Top^\etcov$ denote the classes of topological submersions and surjective local homeomorphisms, respectively, then the data $(\Top, \Top^\sm, \Top^\etcov)$ satisfy the axioms of \sssecref{sssec:Csmcov}.

As in \sssecref{sssec:smooth covering}, a ``smooth covering'' morphism of topological spaces is a topological submersion $f : X \to Y$ for which there exists a surjective local homeomorphism $q : Y' \surj Y$ such that the base change $X \fibprod_Y Y' \to Y'$ admits a section.
In other words, it is a surjective topological submersion.

We may thus define the notions of \emph{topological stack}, \emph{topological $n$-Artin stack} ($n\ge 0$), and \emph{topological Artin stack} as in \sssecref{sssec:nArt}.\footnote{%
  Our definitions are not quite the same as the ones found in the literature.
  First, we only work with locally compact Hausdorff topological spaces (rather than compactly generated topological spaces).
  Second, we allow higher stacks, i.e., sheaves with values in \inftyGrpds rather than $1$-groupoids.
  Third, we impose descent with respect to surjective local homeomorphisms rather than only open covers (but these conditions are in fact equivalent).
  Lastly, our notion of ``topological stacks'' is (the untruncated version of) what is called ``stacks on $\Top$'' in \cite{Noohi}, while our notion of ``topological $1$-Artin stack'' corresponds to what is called ``topological stacks'' in \cite[\S 13.2]{Noohi} when $\mbf{LF}$ in \emph{loc. cit.} is taken to be the class of topological submersions.
}
For example, a topological stack is a presheaf of anima $(\Top)^\op \to \Ani$ satisfying \v{C}ech descent with respect to surjective local homeomorphisms.
A topological stack $\sX$ is $1$-Artin if it has representable diagonal and there exists a topological space $X$ with a morphism $X \to \sX$ which is representable by a smooth covering morphism of spaces.
We write $\Art(\Top)$ for the \inftyCat of topological Artin stacks.

We define topological submersions and proper representable morphisms of Artin stacks as in \sssecref{sssec:nArt} and \sssecref{sssec:Cpr}, respectively.

Applying \corref{cor:weave} to the weave $\Shv^*_!$ on $\Top$ (\thmref{thm:topweave}), which satisfies descent with respect to surjective local homeomorphisms \sssecref{sssec:ShvTopdesc}, we obtain:

\begin{cor}\label{cor:weavetopstk}
  There exists a unique presentable weave $\Shv^{\lis,*}_!$ on $\Art(\Top)$, where all morphisms are shriekable, such that the underlying presheaf $\Shv^{\lis,*}$ is the lisse extension of $\Shv^*$ (see \ssecref{ssec:lisextpresheaf}).
\end{cor}

In particular, we have
\[
  \Shv^{\lis}(\sX) \simeq \lim_{(S,s)} \Shv(S)
\]
for every topological Artin stack $\sX$, where the limit is taken over $(S,s)\in\Lis_\sX$ and the transition functors are $*$-inverse image; see \ssecref{ssec:lisextpresheaf} for the definition of $\Lis_\sX$.
\corref{cor:weavetopstk} asserts that the \inftyCats $\Shv^\lis(\sX)$ are equipped with the operations $(\otimes,\uHom)$ and $(f^*,f_*)$ and $(f_!,f^!)$ for arbitrary morphisms of topological Artin stacks, satisfying the usual base change and projection formulas.
Moreover, we have by \corref{cor:poinc} the Poincaré duality isomorphisms $f^! \simeq \Sigma_f f^*$ for any topological submersion of stacks, where $\Sigma_f$ can be interpreted as tensoring with the Thom space of the relative tangent microbundle of $f$.
Dually, we have $f_! \simeq f_*$ for any proper representable morphism.

\sssec{}

More generally, using \corref{cor:weavenonart} we find that $\Shv^{\lis,*}_!$ extends to a presentable weave on the entire \inftyCat $\Stk$ of topological stacks.
The shriekable morphisms are those which are representable by topological Artin stacks, the ``smooth'' morphisms are those whose base change to any topological space is a topological submersion of topological Artin stacks, and the ``proper'' morphisms are those whose base change to any topological space is a proper morphism of topological spaces.

\appendix
\section{Exchange transformations}

Let $(\cC,\cE)$ be as in \sssecref{sssec:leftpreweave} and $\D^*_!$ a left preweave on $(\cC,\cE)$.

\subsection{Projection formulas}

Let $f : X \to Y$ be a morphism in $\cC$.

\subsubsection{}

If $f$ is shriekable, there is a canonical isomorphism
\begin{equation}\label{eq:projection!}
  \Ex^{\otimes,*}_! : f_!(-) \otimes (-) \simeq f_!(- \otimes f^*(-))
\end{equation}
encoded by the lax monoidality of the preweave $\D^*_!$.

\subsubsection{}

If $f$ is shriekable, there is a canonical exchange transformation
\begin{equation}\label{eq:Ex*!otimes}
  \Ex^{*,!}_{\otimes} :
  f^*(-) \otimes f^!(-) \to f^!(- \otimes -)
\end{equation}
which by definition makes the following diagram commute:
\begin{equation*}
  \begin{tikzcd}
    f^*(-) \otimes f^!(-) \ar{r}{\eqref{eq:Ex*!otimes}}\ar{d}{\mrm{unit}}
    & f^!((-) \otimes (-)). \ar[leftarrow]{d}{\mrm{counit}}
    \\
    f^!f_!(f^*(-) \otimes f^!(-)) \ar[equals]{r}{\eqref{eq:projection!}}
    & f^!(f_!f^!(-) \otimes (-))
  \end{tikzcd}
\end{equation*}

\subsubsection{}

If $f^*$ admits a left adjoint $f_\sharp$, then there is a canonical exchange transformation
\begin{equation}\label{eq:projectionshp}
  \Ex^{\otimes,*}_\sharp :
  f_\sharp(- \otimes f^*(\sG))
    \to f_\sharp(-) \otimes \sG
\end{equation}
which by definition makes the following diagram commute:
\begin{equation*}
  \begin{tikzcd}
    f_\sharp(- \otimes f^*(-)) \ar{r}{\eqref{eq:projectionshp}}\ar{d}{\mrm{unit}}
    & f_\sharp(-) \otimes (-). \ar[leftarrow]{d}{\mrm{counit}}
    \\
    f_\sharp(f^*f_\sharp(-) \otimes f^*(-)) \ar[equals]{r}
    & f_\sharp f^*(f_\sharp(-) \otimes (-))
  \end{tikzcd}
\end{equation*}

\subsubsection{}

If $f^*$ admits a right adjoint $f_*$, then there is a canonical exchange transformation
\begin{equation}\label{eq:projection*}
  \Ex^{\otimes,*}_* : f_*(-) \otimes (-) \to f_*(- \otimes f^*(-))
\end{equation}
which by definition makes the following diagram commute:
\begin{equation*}
  \begin{tikzcd}
    f_*(-) \otimes (-) \ar{r}{\eqref{eq:projectionshp}}\ar{d}{\mrm{unit}}
    & f_*(- \otimes f^*(-)). \ar[leftarrow]{d}{\mrm{counit}}
    \\
    f_*f^*(f_*(-) \otimes (-)) \ar[equals]{r}
    & f_*(f^*f_*(-) \otimes f^*(-))
  \end{tikzcd}
\end{equation*}

\subsubsection{}

Suppose that $f^*$ admits a right adjoint $f_*$ and that $\D(X)$ and $\D(Y)$ are closed symmetric monoidal, so that there are internal hom functors $\uHom_X(-, -)$ and $\uHom_Y(-, -)$.

For every morphism $f: X \to Y$, symmetric monoidality of the functor $f^*$ yields a canonical natural transformation
\begin{equation}\label{eq:vrwhvuhe}
  \Ex_{*}^{\uHom,*} : \uHom_Y(-, f_*(-))
  \to f_* \uHom_X(f^*(-), -)
\end{equation}
which by definition makes the following diagram commute for all $\sF \in \D(Y)$, $\sG \in \D(X)$:
\begin{equation*}
  \begin{tikzcd}[scale cd=0.9, matrix xscale=0.5]
    \uHom_Y(\sF, f_*\sG) \ar{d}{\mrm{unit}}\ar{r}{\eqref{eq:vrwhvuhe}}
    & f_* \uHom_X(f^*\sF, \sG). \ar[leftarrow]{d}{\mrm{counit}}
    \\
    f_*f^* \uHom_Y(\sF, f_*\sG) \ar{d}{\mrm{coev}}
    & f_* \uHom_X(f^*\sF, f^*f_*\sG) \ar[leftarrow]{d}{\mrm{ev}}
    \\
    f_* \uHom_X(f^*\sF, f^*\uHom_Y(\sF, f_*\sG) \otimes f^*\sF) \ar[equals]{r}
    & f_* \uHom_X(f^*\sF, f^*(\uHom_Y(\sF, f_*\sG) \otimes \sF))
  \end{tikzcd}
\end{equation*}
Moreover, \eqref{eq:vrwhvuhe} is \emph{invertible}, as one can check by applying $\Maps(\sH, -)$ for all $\sH \in \D(Y)$.

\subsubsection{}

Suppose $\D(X)$ and $\D(Y)$ are closed symmetric monoidal.
There is a canonical natural transformation
\begin{equation}\label{eq:f^*Hom}
  \Ex^{*}_{\uHom} : f^*\uHom_Y(-, -) \to \uHom_X(f^*(-), f^*(-))
\end{equation}
which by definition makes the square
\begin{equation*}
  \begin{tikzcd}
    f^*\uHom_Y(-, -) \ar{r}{\eqref{eq:f^*Hom}}\ar{d}{\mrm{unit}}
    & \uHom_X(f^*(-), f^*(-)). \ar[leftarrow]{d}{\mrm{counit}}
    \\
    f^*\uHom_Y(-, f_*f^*(-)) \ar{r}{\eqref{eq:vrwhvuhe}}
    & f^*f_* \uHom_X(f^*(-), f^*(-))
  \end{tikzcd}
\end{equation*}
commute.

When $f_\sharp$ exists, \eqref{eq:f^*Hom} is the right transpose of the natural transformation $\Ex_\sharp^{\otimes,*} : f_\sharp(- \otimes f^*(-)) \to f_\sharp(-) \otimes (-)$ \eqref{eq:projectionshp}.
Thus in that case, \eqref{eq:f^*Hom} is invertible if and only if $\Ex_\sharp^{\otimes,*}$ is invertible.

\subsection{Base change formulas}

Suppose given a cartesian square
\[\begin{tikzcd}
  X' \ar{r}{g}\ar{d}{p}
  & Y' \ar{d}{q}
  \\
  X \ar{r}{f}
  & Y
\end{tikzcd}\]
in $\cC$.

\subsubsection{}

If $f$ is shriekable, then there is an invertible exchange transformation
\begin{equation}\label{eq:Ex^*_!}
  \Ex^*_!: q^* f_! \to g_! p^*.
\end{equation}
If $f_!$ and $g_!$ admit right adjoints $f^!$ and $g^!$, and $p^*$ and $q^*$ admit right adjoints $p_*$ and $q_*$, then there is by transposition an invertible exchange transformation
\begin{equation}\label{eq:Ex^!_*}
  \Ex^!_* : p_* g^! \to f^! q_*.
\end{equation}

\subsubsection{}

If $f^*$ and $g^*$ admit left adjoints $f_\sharp$ and $g_\sharp$, there is an exchange transformation
\begin{equation}\label{eq:Ex_sharp^*}
  \Ex^*_\sharp : g_\sharp p^*
  \xrightarrow{\unit} g_\sharp p^* f^* f_\sharp
  \simeq g_\sharp g^* q^* f_\sharp
  \xrightarrow{\counit} q^*f_\sharp
\end{equation}
When this is invertible, we will say that \emph{$f_\sharp$ commutes with $q^*$}.

\subsubsection{}

If $f^*$ and $g^*$ admit right adjoints $f_*$ and $g_*$, and $p^*$ and $q^*$ admit right adjoints $p_*$ and $q_*$, there is an exchange transformation
\begin{equation}\label{eq:Ex_*^*}
  \Ex^*_* : q^* f_*
  \xrightarrow{\unit} q^* f_* p_* p^*
  \simeq q^* q_* g_* p^*
  \xrightarrow{\counit} g_* p^*
\end{equation}
When this is invertible, we will say that \emph{$f_*$ commutes with $q^*$}.
When $p^*$ and $q^*$ admit left adjoints, $\Ex^*_*$ is the right transpose of $\Ex^*_{\sharp} : p_\sharp f^* \to f^* q_\sharp$ \eqref{eq:Ex_sharp^*}.

\subsubsection{}

If $q$ is shriekable, and $p_!$ and $q_!$ admit right adjoints $p^!$ and $q^!$, then there is an exchange transformation
\begin{equation}\label{eq:Ex_!*}
  \Ex_{!,*} : g_!p^!
  \xrightarrow{\unit} g_!p^! f^!f_!
  \simeq g_! g^!q^! f_!
  \xrightarrow{\counit} q^! f_!.
\end{equation}

\subsubsection{}

If $f^*$ and $g^*$ admit left adjoints $f_\sharp$ and $g_\sharp$, there is an exchange transformation
\begin{equation}\label{eq:Ex_shp,!}
  \Ex_{\sharp,!} : f_\sharp p_!
  \xrightarrow{\unit} f_\sharp p_! g^* g_\sharp
  \simeq f_\sharp f^* q_! g_\sharp
  \xrightarrow{\counit} q_! g_\sharp,
\end{equation}
where the isomorphism in the middle is $\Ex^*_!$ \eqref{eq:Ex^*_!}.
When this is invertible, we will say that \emph{$f_\sharp$ commutes with $q_!$} (or $p_!$).

\subsubsection{}

If $f^*$ and $g^*$ admit right adjoints $f_*$ and $g_*$, $p^*$ and $q^*$ admit left adjoints $p_\sharp$ and $q_\sharp$, and $q_\sharp$ commutes with $f^*$, then there is an exchange transformation
\begin{equation}\label{eq:Ex_shp,*}
  \Ex_{\sharp,*} : q_\sharp g_*
  \xrightarrow{\unit} f_* f^* q_\sharp g_*
  \simeq f_* p_\sharp g^* g_*
  \xrightarrow{\counit} f_* p_\sharp,
\end{equation}
where the isomorphism in the middle is $\Ex_\sharp^*$ \eqref{eq:Ex_sharp^*}.
When this is invertible, we will say that \emph{$f_*$ commutes with $q_\sharp$} (or $p_\sharp$).

\subsubsection{}

If $q$ is shriekable, $p_!$ and $q_!$ admit right adjoints $p^!$ and $q^!$, and $f^*$ and $g^*$ admit right adjoints $f_*$ and $g_*$, there is an exchange transformation
\begin{equation}\label{eq:Ex^*,!}
  \Ex^{*,!} : g^* q^!
  \xrightarrow{\unit} g^* q^! f_* f^*
  \simeq g^* g_* p^! f^*
  \xrightarrow{\counit} p^! f^*,
\end{equation}
where the isomorphism in the middle is $\Ex^!_*$ \eqref{eq:Ex^!_*}.
When this is invertible, we will say that \emph{$f^*$ commutes with $q^!$} (or $p^!$).
When $f^*$ and $g^*$ admit left adjoints, $\Ex^{*,!}$ is the right transpose of $\Ex_{\sharp,!}$ \eqref{eq:Ex_shp,!}.

\subsubsection{}

If $q$ is shriekable, $f^*$ and $g^*$ admit right adjoints $f_*$ and $g_*$, and $p_!$ and $q_!$ admit right adjoints $p^!$ and $q^!$, then there is an exchange transformation
\begin{equation}\label{eq:Ex_!*}
  \Ex_{!,*} : q_!g_*
  \xrightarrow{\unit} q_!g_* p^!p_!
  \simeq q_! q^!f_* p_!
  \xrightarrow{\counit} f_*p_!
\end{equation}
where the isomorphism in the middle is $\Ex^!_*$ \eqref{eq:Ex^!_*}.

\changelocaltocdepth{1}
\section{Descent by diagrams}

Let $\cC$ be an \inftyCat with fibred products.
Let $X \in \cC$ and $F : \sI \to \cC_{/X}$ a diagram in $\cC_{/X}$ indexed by an \inftyCat $\sI$.
We will denote this informally by $(f_i : Y_i \to X)_{i \in \sI}$, where $Y_i := F(i)$ for $i \in \sI$.

\subsection{}
\label{ssec:limstar}

Let $\D^* : \cC^\op \to \Cat$ be a presheaf of \inftyCats.
Consider the canonical functor
\begin{equation}\label{eq:F^*}
  F^* : \D(X) \to \lim_{i \in I} \D(Y_i).
\end{equation}
If each $f_i^*$ admits a right adjoint $f_{i,*}$, then $F^*$ admits a right adjoint $F_*$ given informally by
\begin{equation}\label{eq:F_*}
  (\sF_i)_{i\in \sI}
  \mapsto
  \lim_{i\in \sI} f_{i,*} (\sF_i).
\end{equation}
If \eqref{eq:F^*} is an equivalence, then in particular the unit $\id \to F_* F^*$ is invertible.
In other words, the canonical morphism in $\D(X)$
\begin{equation}
  \sF \to \lim_{i \in I} f_{i,*} f_i^*(\sF)
\end{equation}
is invertible for all $\sF \in \D(X)$.

\subsection{}
\label{ssec:niavzhjn}

Let $\D^! : \cC^\op \to \Cat$ be a presheaf of \inftyCats.
Consider the canonical functor
\begin{equation}\label{eq:F^!}
  F^! : \D(X) \to \lim_{i \in I} \D(Y_i).
\end{equation}

If each $f_i^!$ admits a left adjoint $f_{i,!}$, then $F^!$ admits a left adjoint $F_!$ given informally by
\begin{equation}\label{eq:F_!}
  (\sF_i)_{i\in \sI}
  \mapsto
  \colim_{i\in \sI} f_{i,!} (\sF_i).
\end{equation}
If \eqref{eq:F^!} is an equivalence, then in particular the counit $F_! F^! \to \id$ is invertible.
In other words, the canonical morphism in $\D(X)$
\begin{equation}
  \colim_{i \in I} f_{i,!} f_i^!(\sF) \to \sF
\end{equation}
is invertible for all $\sF \in \D(X)$.

\subsection{}
\label{ssec:colimshp}

Suppose given a presheaf $\D^* : \cC^\op \to \Cat$ such that each $f_i^*$ admits a \emph{left} adjoint $f_{i,\sharp}$.
The discussion in \ssecref{ssec:niavzhjn} translates as follows.

The functor $F^*$ admits a left adjoint $F_\sharp$, given informally by
\begin{equation}\label{eq:F_sharp}
  (\sF_i)_{i\in \sI}
  \mapsto
  \colim_{i\in \sI} f_{i,\sharp} (\sF_i).
\end{equation}
If \eqref{eq:F^*} is an equivalence, then in particular the counit $F_\sharp F^* \to \id$ is invertible.
In other words, the canonical morphism in $\D(X)$
\begin{equation}
  \colim_{i \in I} f_{i,\sharp} f_i^*(\sF) \to \sF
\end{equation}
is invertible for all $\sF \in \D(X)$.


\bibliographystyle{halphanum}

{\small
\noindent
Institute of Mathematics, Academia Sinica, Taipei 106319, Taiwan}

\end{document}